\documentclass[a4paper,oneside,10pt]{article}

\usepackage[english]{babel}
\usepackage[utf8]{inputenc}
\usepackage[T1]{fontenc}
\usepackage{lmodern}
\usepackage{cite}
\usepackage{microtype}
\usepackage{tikz}
\usetikzlibrary{patterns,arrows,decorations.markings}
\tikzset{fleche/.style args={#1:#2}{postaction = decorate,decoration={name=markings,mark=at position #1 with {\arrow[>=stealth,#2,scale=2]{>}}}}}
\usepackage{perpage}
\MakePerPage{footnote}
\usepackage[font={footnotesize},labelfont=bf]{caption}

\usepackage{amssymb}
\usepackage{stmaryrd}
\usepackage{amsmath}
\usepackage{mathrsfs}
\usepackage{dsfont}
\usepackage{exscale}
\usepackage{relsize}
\usepackage{amsthm}

\usepackage[pdftex,linkcolor=black,citecolor=black,
    hypertexnames     = false,
    hyperfootnotes    = true,
    plainpages        = false,
    bookmarks         = true,     
    pdfmenubar        = true,
    bookmarksnumbered = true,     
    pdfpagelayout     = SinglePage,
    colorlinks        = true,      
    urlcolor          = magenta,   
    unicode%
    ]{hyperref}

\newcommand{\vv}{\mathrm{Z}} 
\newcommand{\yy}{\mathrm{Y}} 

\newcommand{\NP}{N} 

\newcommand{\N}{\mathbb{N}}\newcommand{\NN}{\mathbb{N}}
\newcommand{\Z}{\mathbb{Z}}
\newcommand{\R}{\mathbb{R}}
\newcommand{\ie}{\emph{i.e.} }
\renewcommand{\leq}{\leqslant}
\renewcommand{\geq}{\geqslant}
\renewcommand{\emptyset}{\varnothing}
\newcommand{\law}{\mathscr{L}}
\newcommand{\defeq}{:=}
\newcommand{\rinf}[1]{\underline{#1}} 					
\newcommand{\ind}[1]{\mathds{1}_{\{#1\}}}		
\newcommand{\indic}[1]{\mathds{1}_{#1}}			
\newcommand{\gfrac}[2]{\genfrac{}{}{0pt}{1}{#1}{#2}}

\newcommand{\Bu}{\smash{\underline{{B}}}}

\newcommand{\Pg}{\mathbf{P}}
\newcommand{\Eg}{\mathbf{E}}

\renewcommand{\P}{\mathbb{P}}
\newcommand{\E}{\mathbb{E}}

\renewcommand{\C}{\mathcal{C}}			
\newcommand{\Hl}{\mathcal{H}}			
\newcommand{\D}{\partial\mathcal{H}}	
\newcommand{\dist}{\mathbf{d}}			

\newcommand{\Rg}{\mathbf{R}}			
\newcommand{\Bg}{\mathbf{B}}			
\newcommand{\Vg}{\mathbf{V}}
\newcommand{\Vgr}{\mathbf{V}^r}			
\newcommand{\Vgb}{\mathbf{V}^b}						
\newcommand{\Sg}{\mathbf{S}}

\renewcommand{\thetag}{\mathbf{\Theta}}	
\newcommand{\Thetag}{\mathbf{\Theta}}	
\newcommand{\Deltag}{\mathbf{\Delta}}	
\renewcommand{\Eg}{\mathbf{E}}

\newcommand{\Rd}{R}			
\newcommand{\Bd}{B}			
\newcommand{\Vd}{V}			
\newcommand{\Vdr}{V^r}			
\newcommand{\Vdb}{V^b}			
\newcommand{\Sd}{S}			
\newcommand{\thetad}{\Theta}	
\newcommand{\Thetad}{\Theta}	
\newcommand{\Bdu}{\rinf{\Bd}}
\newcommand{\Rdu}{\rinf{\Rd}}
\newcommand{\Sdu}{\rinf{\Sd}}

\newcommand{\z}[2]{\smash{\mathrm{Z}^{(#1)}_{#2}}} 

\newcommand{\Rc}{\mathcal{R}}
\newcommand{\Bc}{\mathcal{B}}
\newcommand{\Vc}{\mathcal{V}}
\newcommand{\Vcr}{\mathcal{V}^r}
\newcommand{\Vcb}{\mathcal{V}^b}
\newcommand{\Sc}{\mathcal{S}}
\newcommand{\thetac}{\theta}	
\newcommand{\Thetac}{\theta}	
\newcommand{\Bcu}{\rinf{\Bc}}
\newcommand{\Rcu}{\rinf{\Rc}}

\newcommand{\Sh}{\Sc^{\scriptscriptstyle{+}}}
\newcommand{\Vh}{\Vc^{\scriptscriptstyle{+}}}
\newcommand{\St}{\smash{\check{\Sc}}}
\newcommand{\Vt}{\smash{\check{\Vc}}}
\newcommand{\Tu}{T^{\uparrow}}

\newcommand{\Td}{\mathbf{T}}
\newcommand{\Tud}{{\mathbf{T}^{\uparrow}}}

\newcommand{\Tup}{T^{\scriptscriptstyle{\uparrow}}}
\newcommand{\Tdown}{T}

\newtheorem{theo}{Theorem}

\newtheorem{prop}[theo]{Proposition}
\newtheorem{rmk}[theo]{Remark}
\newtheorem{cor}[theo]{Corollary}
\newtheorem{lem}[theo]{Lemma}
\newtheorem{conj}{Conjecture}

\renewcommand{\thesubsection}{\alph{subsection})}
\renewcommand{\thesubsubsection}{\roman{subsubsection})}
\begin{document}

\renewcommand{\contentsname}{Contents}
\renewcommand{\refname}{\textbf{References}}
\renewcommand{\abstractname}{Abstract}

\begin{center}
\huge{The geometry of a critical percolation cluster on the UIPT}\\
\bigskip
\bigskip
\bigskip

\large{Matthias Gorny\footnote{FMJH, Univ. Paris-Sud, Universit\'e Paris-Saclay. \emph{email:} {matthias.gorny@math.u-psud.fr}
}, \'Edouard Maurel-Segala\footnote{Univ. Paris-Sud, Universit\'e Paris Saclay. \emph{email:} {edouard.maurel-segala@math.u-psud.fr}}, Arvind Singh\footnote{CNRS, Univ. Paris-Sud, Universit\'e Paris-Saclay. \emph{email:} {arvind.singh@math.u-psud.fr}. Work partially supported by ANR MALIN. 
}} 
 
\bigskip \bigskip \bigskip 
\end{center}

\begin{abstract}
We consider a critical Bernoulli site percolation on the uniform infinite planar triangulation. We study the tail distributions of the peeling time, perimeter, and volume of the hull of a critical cluster. The exponents obtained here differs by a factor $2$ from those computed previously by Angel and Curien~\cite{AC15} in the case of critical site percolation on the uniform infinite \emph{half-plane} triangulation. 
\end{abstract}
\bigskip \bigskip



\bigskip \bigskip \bigskip


\section{Introduction}
\label{Introduction}

The Uniform Infinite Planar Triangulation (UIPT) provides a simple, yet rich model of random planar geometry. Since its introduction by Angel and Schramm~\cite{AS03}, it has been the focus of intensive research which fostered progresses in the understanding of the geometric properties of generic random metric spaces, which is of interest to both mathematicians and phycisists. In this paper, we focus on the  model of site percolation on the UIPT, first considered by Angel in~\cite{Ang03}. We study the size of a typical cluster at criticality, computing the exponents associated with its peeling time, the volume and the perimeter of its hull. 
\medskip

\textbf{The UIPT.} A \emph{planar map} is a proper embedding of a finite connected graph on the sphere $\mathbb{S}_2$ considered up to an orientation preserving homeomorphism of the space. In order to prevent undesirable symmetries, it is usual to root the maps by distinguishing a special oriented edge $e = (x,y)$ called the \emph{root edge}. We shall always do so and we call \emph{root vertex} the origin vertex $x$ of $e$. A planar map is a \emph{triangulation} if all its faces (the connected components of the complementary
of the image of the embedding) are triangles. More generally, we call \emph{triangulation with a boundary} a map where all but one of its faces are triangles. If this special face -- called outer face -- is simple and composed of $n$ vertices, we say that we have a \emph{triangulation of the $n$-gon}.  Depending on the types of graphs allowed, one can consider different families of triangulations. Following the classification introduced in~\cite{AS03}, we define 
\begin{enumerate}
\item[]\textbf{Type I graphs.} Double edges and loops are allowed.
\item[]\textbf{Type II graphs.} Double edges are allowed but loops are forbidden.  
\item[]\textbf{Type III graphs.} Double edges and loops are both forbidden.
\end{enumerate}
In this paper, we restrict our study to the case of type II triangulations. However, the results obtained here also apply to type I and type III triangulations with minor and mostly straightforward  modifications of the proofs. It is possible that these results may also extend to other types of planar maps such as quadrangulations but adapting our arguments to these other cases seems more delicate. 
\medskip

The set of rooted planar triangulation with $n$ faces is finite so we can consider the uniform measure to pick at random an element $(T_n,e_n)$ where $T_n$ is a triangulation with $n$ faces rooted at $e_n$. The geometrical properties of this object are directly related to the combinatorics associated with these families of graphs. Since the pioneer work of Tutte in the 60's, explicit formulas for counting these objects are available and provide insight on the structure of a finite random triangulation.  
\medskip

It is natural to inquire about the limit of a random triangulation when its size increases to infinity. The approach considered here is that of the \emph{local weak convergence} of graphs\footnote{There is another way to define the limit of a random triangulation, called \emph{macroscopic limit}, by re-scaling the triangulation (seen as a random metric space) in such way that it converges to a compact random metric space called the \emph{brownian map} \emph{c.f.}~\cite{LG13,Mier13}. This results have far reaching consequences and the limiting object is the subject of intensive ongoing research but we shall not be concerned with it here.} introduced by Benjamini and Schramm~\cite{BS01}. More precisely, define the distance $\dist$ on the set of rooted maps by
\[
\dist\big((T,e_T) , (T',e_{T'})\big) \defeq\inf\,\big\{\,(1+R)^{-1} : B_T(e_T,R)\neq B_{T'}(e_{T'},R)\,\big\},
\]
where $B_T(e_T,R)$ denotes the ball of radius $R$ in $T$, \ie the finite rooted map obtained by deleting all vertices in $T$ at distance greater than $R$ from the root $e_T$ with respect to the graph distance. The set of all finite maps can be completed with respect to this distance and the new objects in this completion are called \emph{infinite maps}. In a seminal paper~\cite{AS03}, Angel and Schramm established that there exists a random variable $(T_\infty,e_\infty)$ supported on infinite rooted planar triangulation which is the limit in law of uniform finite rooted triangulations $(T_n,e_n)$ for the convergence induced by $\dist$:
\[
B_{T_n}(e_{n},R)\underset{n\to+\infty}{\overset{\law}\longrightarrow} B_{T_\infty}(e_{\infty},R)\quad\mbox{for all $R>0$}.
\]
This limiting object $T_\infty$ is a proper one-ended planar graph called the Uniform Infinite Planar Triangulation (UIPT).\medskip

\textbf{Percolation on the UIPT.} Somewhat remarkably, the study of percolation and related models of statistical physics is often simpler on random planar maps than on an euclidean lattice, thanks to a spatial Markov property which we will recall in the next section. In particular, explicit formulas for critical parameters and related quantities can often be computed explicitely. In this paper, we consider a percolation on the vertices of the UIPT constructed as follow. Fix a percolation parameter $p\in (0,1)$. Given a realization of $T_\infty$, we color each site of the triangulation independently, in red with probability~$p$ or in blue with probability $1-p$. Let $\C$ denote the connected component of red sites that contains the root vertex of the UIPT. In~\cite{Ang03}, Angel proved that the critical parameter of this model is 
\[
p_c\defeq\sup\,\Big\{\,p\geqslant 0 : \Pg\big( |\C| = \infty \big) = 0\,\Big\} = \frac{1}{2}
\] 
and that, furthermore, there is no percolation at criticality. Therefore we have the dichotomy:
\begin{enumerate}
\item[$\bullet$] If $p>1/2$, there exists an infinite connected component of red site a.s.
\item[$\bullet$] If $p\leqslant 1/2$, all the red connected components are finite a.s.
\end{enumerate}
Let us point out that the a.s.~above is with respect to the annealed law, \ie it takes into account both the randomness from the map and from the coloring of the sites. In subsequent works~\cite{MN14,AC15,CK15}, these results were generalized to other kind of percolation models such as bond percolation as well as other classes of random maps such as quadragulations and maps on the half plane. 
\medskip

We are interested in the geometry of the red cluster $\C$ when the percolation is critical. Thus, from now on, we fix $p=p_c=1/2$. For the sake of simplicity and to avoid dealing with degenerated cases, we also make the harmless assumption that the root vertex is colored red (otherwise $\C=\emptyset$). Since the UIPT is one-ended and since $\C$ is finite, its complement has exactly one infinite connected component $\mathcal{D}$. The hull $\Hl$ of the cluster $\C$ is defined as the complement of~$\mathcal{D}$. Alternatively, $\Hl$ corresponds to ``filling in'' the hole in $\C$ by adding to it all the sites that it disconnects from infinity. We denote by $|\Hl|$ the volume of the hull, that is the number of sites in~$\Hl$. The boundary $\D$ of the hull is defined as the set of edges which connect a (red) vertex of $\Hl$ to a (blue) vertex of $\mathcal{D}$. See figure~\ref{fig:HullPerco} for an illustration.\medskip

\begin{figure}[t]
\centering
\begin{tikzpicture}[scale=0.26]
%
\fill[color=red!20,rounded corners](20,16.3) to (17.2,16.3) to (16.2,21.3) to (13,22.3) to (9,20.3) to (4.7,17.2) to (3.7,12) to (4.8,6.8) to (7.8,3.8) to (11,3.8) to (14,5.7) to (17.4,4.7) to (15.3,9) to (12,12.4) to (8.3,14) to (9.2,16.7) to (13,15.7) to (17,15.7) to (18.7,13) to (19.7,9)to(20.3,9) to (19.3,13)to(20.2,15.7) to (23.2,14.7)to (23.2,15.3)to(20,16.3);
%
\draw[fleche=0.6:black] (12,12)--(12,8);
%
\draw[thick] (12,12)--(15,9)--(17,5)--(14,6)--(15,9);
\draw[thick] (14,6)--(11,4)--(8,4)--(5,7)--(4,12)--(8,14)--(12,12);
\draw[thick] (8,14)--(9,17)--(4,12)--(5,17)--(9,17)--(9,20)--(5,17);
\draw[thick] (9,17) --(13,16)--(17,16)--(19,13)--(20,16)--(17,16) --(16,21)--(13,22)--(9,20);
\draw[thick] (20,9)--(19,13)--(20,16)--(23,15);
%
\draw (12,8)--(14,6);\draw (12,8)--(15,9);\draw (12,8)--(11,4);\draw (12,8)--(10,6);\draw (12,8)--(10,8);\draw (12,8)--(10,11);\draw (12,8)--(10,8);\draw (10,11)--(12,12);\draw (10,11)--(8,14);\draw (10,11)--(7,11);\draw (10,11)--(7.5,8.5);\draw (10,11)--(10,8);\draw (10,8)--(10,6);\draw (10,8)--(7.5,8.5);
\draw (7,11)--(8,14);\draw (7,11)--(4,12);\draw (7,11)--(5,7);\draw (7,11)--(7.5,8.5);\draw (7.5,8.5)--(5,7);\draw (7.5,8.5)--(10,6);\draw (7.5,8.5)--(8,4);\draw (10,6)--(8,4);\draw (10,6)--(11,4);\draw (11.5,19)--(13,22);\draw (11.5,19)--(9,20);\draw (11.5,19)--(9,17);\draw (11.5,19)--(13,16);\draw (15,18)--(13,16);\draw (15,18)--(17,16);\draw (15,18)--(16,21);\draw (15.5,12.5)--(12,14);\draw (15.5,12.5)--(16,10.5);\draw (15.5,12.5)--(17,13.5);\draw (16,10.5)--(17,13.5);\draw (11.5,19)to[out=20,in=180](13.5,19.5) to[out=0,in=80] (15,18);\draw (15,18)to[bend left](11.5,19);\draw(11.5,19)--(16,21);\draw (11.5,19)--(13.5,18.75)--(15,18);
%
\draw (26,14)--(24,18)--(21,20)--(19,21)--(17,25) --(13.5,25.5)--(10,24)--(7,23)--(4,20)--(0,18)--(2,15)--(1,11)--(3,9)--(2,6)--(6,4)--(8,2)--(12,2)--(14,0) --(17,1)--(20,4)--(23,6)--(27,10)--(23.5,12.75)--cycle;
\draw (20,9) to[out=60,in=-170] (22.75,11.75)to[out=10,in=145](27,10);\draw (20,9)to[out=-45,in=-160] (24,7.75) to[out=20,in=-140] (27,10);\draw(20,9)--(23,11)--(27,10)--(24,8.5)--(20,9);\draw(24,8.5)--(23,11);\draw (20,9)to[out=60,in=-170](23.5,12.75);
%
\draw(26,14)--(23,15)--(24,18)--(20,16)--(21,20)--(17,16)--(19,21)--(16,21)
--(17,25)--(13,22)--(13.5,25.5)--(10,24)--(13,22)--(7,23)--(9,20)--(4,20);
\draw (5,17)--(0,18);\draw (5,17)--(4,20);\draw (5,17)--(2,15);\draw (4,12)--(2,15);\draw (4,12)--(1,11);\draw (4,12)--(3,9);\draw (5,7)--(3,9);\draw (5,7)--(2,6);\draw (5,7)--(6,4); \draw (6,4)--(8,4)--(8,2)--(11,4)-- (12,2)--(14,6)--(14,0)--(17,5); \draw (17,5)--(17,1);\draw (17,5)--(20,4)--(18,9);\draw (20,4)--(20,9);\draw (17,5)--(18,9);\draw (12,14)--(9,17);\draw (12,14)--(8,14);\draw (16,10.5)--(18,9);\draw (16,10.5)--(15,9);\draw (16,10.5)--(12,14);\draw (12,14)--(12,12);\draw (12,14)--(15,9);\draw (12,14)--(17,13.5);\draw (12,14)--(17,16);\draw (12,14)--(13,16);\draw (17,13.5)--(17,16);\draw (17,13.5)--(19,13);\draw (17,13.5)--(18,9);\draw (17,13.5)--(20,9);\draw (18,9)--(15,9);\draw (18,9)--(20,9);\draw (20,9)--(23,6);\draw (23.5,12.75)--(19,13);\draw (23.5,12.75)--(20,16);\draw (23.5,12.75)--(23,15);\draw (23.5,12.75)--(26,14);\draw (0,18) to[bend right] (5,17) to[bend right] (0,18);
%
\draw[dotted,color=red,rounded corners](20,16.3) to (17.2,16.3) to (16.2,21.3) to (13,22.3) to (9,20.3) to (4.7,17.2) to (3.7,12) to (4.8,6.8) to (7.8,3.8) to (11,3.8) to (14,5.7) to (17.4,4.7) to (15.3,9) to (12,12.4) to (8.3,14) to (9.2,16.7) to (13,15.7) to (17,15.7) to (18.7,13) to (19.7,9)to(20.3,9) to (19.3,13)to(20.2,15.7) to (23.1,14.7) to (23.2,15.3)to(20,16.3);
%
\foreach \i in {(4,12),(5,7),(5,17),(8,4),(8,14),(9,17),(9,20),(10,8),(11,4),(12,12),(13,16),(13,22),(14,6),(15,9),(17,5),(16,21),(17,16),(19,13),(20,9),(20,16),(23,15),(15.5,12.5),(13.5,18.75)}{\draw \i node[color=red!10] {$\bullet$} node[color=red] {$\circ$};}
\foreach \i in {(0,18),(1,11),(2,6),(2,15),(3,9),(4,20),(6,4),(7,11),(7,23),(8,2),(7.5,8.5),(10,6),(10,11),(10,24),(12,2),(12,8),(11.5,19),(12,14),(14,0),(13.5,25.5),(15,18),(17,13.5),(17,1),(18,9),(17,25),(19,21),(20,4),(21,20),(23.5,12.75),(23,6),(27,10),(24,18),(26,14),(16,10.5),(2.5,17.5),(23,11),(24,8.5)} {\draw \i node[color=blue] {$\bullet$} node[color=blue] {$\circ$};}
\end{tikzpicture}
\caption{An example of the cluster $\C$ (the red vertices connected with edges colored in red) and its hull $\Hl$ (the region in light red). \textit{Here $|\Hl|=29$ and $|\D|=57$}.}
\label{fig:HullPerco}
\end{figure}
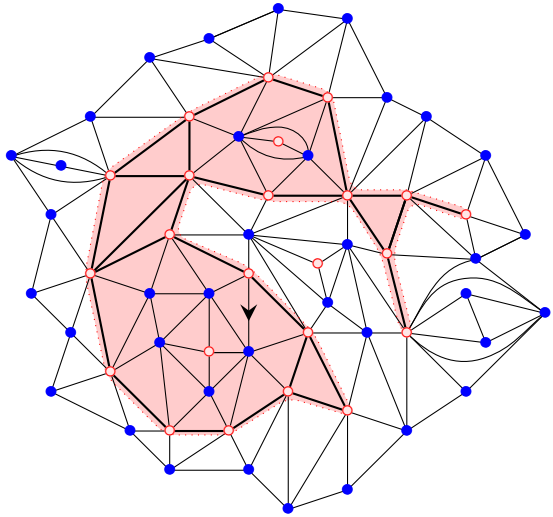

Given two non-negative functions $f$ and $g$, we use the notation $f \asymp g$ if there exist constants $c,C >0$ such that $ c \,g(x) \leqslant f(x) \leqslant C\, g(x)$ for all $x$ large enough. We also use the notation $f \lesssim g$ (respectively $f \gtrsim g$) when we only request the upper (respectively lower) bound to hold true. Our main result estimates the size of the hull $\Hl$ and of its boundary $\D$. 

\begin{theo}
\label{mainTheo}
Consider a critical site percolation on the UIPT. We have
\begin{equation}\label{resVolume}
\Pg(|\Hl|>n)  \asymp \frac{1}{n^{1/8}}
\end{equation}
and
\begin{equation}\label{resPerimeter}
\Pg(|\D|>n)  \asymp \frac{1}{n^{1/6}}.
\end{equation}
\end{theo}
\medskip
Let us make a few comments about this result. Another model related to the UIPT is the so-called \emph{Uniform infinite half-plane triangulation} (UIHPT). It is the random map obtained by considering first the local limit, when the number of interior vertices goes to infinity, of a uniform finite triangulation of the $m$-gon and then by letting $m$ itself go to infinity. This construction yields an infinite planar triangulation with an infinite boundary (rooted on the boundary) which can be embedded in the upper half-plane, hence its name. In~\cite{AC15}, Angel and Curien studied different percolation models on half-plane maps. In particular they showed that, for the UIHPT, we have the following estimate for the size of a critical percolation cluster rooted on the boundary:
\begin{equation}\label{ACvol}
\Pg(|\Hl|>n)  = \frac{1}{n^{1/4 + o(1)}}
\end{equation}
and
\begin{equation}\label{ACperim}
\Pg(|\D|>n)  \asymp \frac{1}{n^{1/3}}.
\end{equation}
Thus, the exponents for the half-plane triangulation differ by a factor $2$ from those of Theorem~\ref{mainTheo} for the full plane triangulation. Yet, we do not have a convincing heuristics as to why this should be so. Let us also mention that the arguments given in this paper may also be used to sharpen~\eqref{ACvol}, removing the sub-polynomial correction term. In fact, it is intuitively clear that all the estimates above should hold up to an equivalence sign instead of a $\asymp$ sign but our approach cannot give such precise results. 
\medskip

Another closely related work is that of Curien and Kortchemsky~\cite{CK15} where, among other quantities, the authors compute  typical length of a critical percolation interface in the UIPT. More precisely, assuming that the end vertices of the root edge are respectively  blue and red, they established that the length $\gamma$ of the interface between the corresponding blue and red clusters satisfies:
\[
\Pg(|\gamma| > n)\sim \frac{c}{n^{1/3}}
\]
for some explicit positive constant $c$. Thus, the length $\gamma$ is of the same order as that of $|\D|$ on the UIHPT but it is much smaller than the perimeter given by~\eqref{resPerimeter}. The reason is that $\gamma$ corresponds to the minimum of the perimeters of two adjacent clusters which are correlated and  unlikely to be large simultaneously. 
\medskip

In \cite{AC15},~\cite{CK15} and in this paper, the authors always consider the hull $\Hl$ of the percolation cluster instead of looking directly at $\C$. This limitation stems from the fact that some key estimates on the size of a percolation cluster on finite Boltzmann maps are missing. However, the arguments developed here are fairly robust and would also allow to estimate the real volume provided those estimates were available. See remark~\ref{remarkRealVolume} at the end of section~\ref{SectionResolution} for additional details.
\medskip
 
The approach used in this paper is based on the peeling process which is a well known standard tool when dealing with random maps with a spatial Markov property. Yet, using it to study percolation on full maps instead of half maps as in~\cite{AC15} leads to several additional difficulties. The most stringent one being that we must work with two-dimensional processes instead of a one-dimensional random walk as in the half-plane case. We believe that, in addition to the results stated in Theorem~\ref{mainTheo}, the main contribution of the paper are in the methods and tools employed to study these processes and which may prove useful for studying other quantities related to the UIPT.

\medskip

\textbf{Organization of the paper.} In the next section, we recall the peeling procedure on the UIPT. This allows us to rewrite the estimates of Theorem~\ref{mainTheo} concerning the volume of the hull $|\Hl|$ in term of functionals of a two-dimensional Markov Chain. We show in section~\ref{SectionPerimeter} that the distribution of the size of the perimeter of the hull $|\D|$ is related to distribution of the time needed to discover the red cluster $\C$. From this point on, the initial problem boils down to estimating fluctuations of a particular two-dimensional process. Yet, this still turns out to be quite tricky to do directly. To overcome this difficulty, we apply various transformations to the original processes  in section~\ref{SectionSimplification}, effectively reducing the problem to that of studying the joint fluctuations of two independent random walks, one being conditioned to stay positive. Finally, the proof of the main theorem is carried out in section~\ref{SectionResolution}. The last section, in appendix, collects several technical estimates concerning random walks which are used throughout the paper.  

\section{Peeling of a percolation cluster}
\label{SectionEncoding}

The \emph{peeling process} of a uniform random planar map was introduced by Watabiki \cite{Wat95} and subsequently formalized by Angel \cite{Ang03,Ang05}. The key idea is to make use of the spatial Markov property of the map to reveal its faces one at a time. This approach which enable to construct the graph together with the percolation ``on the fly'' has become the \emph{de facto} standard approach for studying dynamics on random maps, \emph{c.f.}~\cite{Ang03,Ang05, AC15, MN14, CLG16}. We refer the reader to~\cite{Ang03,Cur16,CLG16} for the proofs and additional details  concerning the results stated in this section. \medskip

The peeling procedure is particularly simple in the case of triangulations. It goes as follow.
\begin{itemize}
\item We start by revealing only the root edge of the UIPT together with its two adjacent colored vertices. By splitting this edge in two, we can think of this graph as a rooted planar triangulation of the $2$-gon (without interior vertices).
\item At each step, we have a colored rooted finite planar triangulation of the $n$-gon with $n\geqslant 2$. In order to perform a new step, we choose an edge $e$ on the boundary of the already revealed region and we reveal the outer face that is adjacent to it. There are two cases to consider:
\begin{enumerate}
\item The third vertex of the newly revealed face is a new vertex. In this case, we simply reveal its color and we are left, again, with a colored rooted finite planar triangulation with a boundary of size $n+1$.   
\item The third vertex, say $v$, of the newly revealed face is located on the boundary. Then, adding this new face separates the undiscovered portion of the UIPT into two disjoints connected components with only one being infinite. Thus, we also reveal the finite triangulation (together with its coloring) to recover a colored finite planar triangulation with a boundary of size $n-d \geqslant 2$ where $d$ is the distance on the original boundary between the peeling edge $e$ and the third vertex $v$.       
\end{enumerate}  
\end{itemize}

\begin{figure}[t]
\centering
\begin{tabular}{cccc}

\begin{tikzpicture}[scale=1]
\foreach \k in {1,2,...,11} {\coordinate (A\k) at ({cos(360*(\k+7)/11)},{sin(360*(\k+7)/11)});}
\coordinate (D7) at ({cos(360*3/11)+2*sin(180/11)*cos(360*3/11+180/11-60)},{sin(360*3/11)+2*sin(180/11)*sin(360*3/11+180/11-60)});
\fill[color=gray!20]  (A1) -- (A2)--(A3)--(A4)--(A5)--(A6)--(A7)--(A8)-- (A9)--(A10)--(A11)--cycle;
\draw  (A7) -- (D7)--(A6);
\draw  (A7)--(A8)-- (A9)--(A10)--(A11)--(A1) -- (A2)--(A3)--(A4)--(A5)--(A6);
\draw[thick]  (A7) to node[below]{$e$} (A6);
\foreach \k in {1,2,...,11} {\draw (A\k) node[color=white]{$\bullet$}node{$\circ$};}
\draw  (D7) node[color=white]{$\bullet$}node{$\circ$};
\end{tikzpicture}

&

\begin{tikzpicture}[scale=1]
%
\foreach \k in {1,2,...,11} {\coordinate (A\k) at ({cos(360*(\k+7)/11)},{sin(360*(\k+7)/11)});}
%
\fill[color=gray!20]  (A1) -- (A2)--(A3)--(A4)--(A5)--(A6)--(A7)--(A8)-- (A9)--(A10)--(A11)--cycle;
%
\coordinate (B5) at ({2*cos(360*12/11)},{2*sin(360*(5+7)/11)});
\coordinate (C5) at ({1.5*cos(360*12/11)},{1.5*sin(360*(5+7)/11)});
\fill[color=gray!100]  (A6) to[out=90,in=130] (C5) to[out=-50,in=0] (A3) -- (A4) -- (A5)--(A6);
\draw (A7) to[out=90,in=120] (B5);
\draw (B5) to[out=-60,in=0] (A3);
\draw (A6) to[out=90,in=130] (C5);
\draw (C5) to[out=-50,in=0] (A3);
%
\draw  (A7)--(A8)-- (A9)--(A10)--(A11)--(A1) -- (A2)--(A3)--(A4)--(A5)--(A6);
\draw[thick]  (A7) to node[above]{$e$} (A6);
%
\foreach \k in {1,2,...,11} {\draw (A\k) node[color=white]{$\bullet$}node{$\circ$};}
%
\draw[>=stealth,<->] ({cos(360*(6+7)/11)},{sin(360*(6+7)/11)-0.1}) to node[pos=0.5,left]{$d$} ({cos(360*(3+7)/11)-0.05},{sin(360*(3+7)/11)+0.1});
\end{tikzpicture}
\end{tabular}
\caption{The two different cases when peeling a triangulation. The light gray parts are the faces which have already been discovered. The dark gray part is the finite region (free Boltzmann triangulations) discovered when revealing the outer face adjacent to $e$. Here, we didn't reveal the colors of the vertices.}
\label{fig:Peeling}
\end{figure}
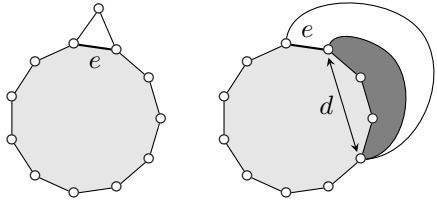

See figure~\ref{fig:Peeling} for an illustration of the different cases. Repeating this procedure \emph{ad infinitum} reveals the whole UIPT. The important fact here is that this peeling process is Markovian in the sense that the law of the newly revealed regions depends only on the geometry of the currently discovered graph through the size of its outer boundary. Moreover, the transition probabilities are explicit and remarkably simple. For the type II triangulations considered here, we define
\begin{equation}\label{defrwpk}
p_k\defeq
\left\{\begin{array}{cl}
\displaystyle{0}&\mbox{for}\quad k=0 \quad\mbox{or} \quad k\geqslant 2,\\[0.25cm]
\displaystyle{\frac{2}{3}}&\mbox{for}\quad k=1,\\[0.4cm]
\displaystyle{\frac{2(-2k-2)!}{4^{-k}(-k-1)!(-k+1)!}}&\mbox{for}\quad k\leqslant -1,
\end{array}\right.
\end{equation}
We have $\sum_{k\in \Z} p_k = 1$ thus $(p_k)_{k\in \Z}$ is a probability distribution on $\Z$ which we interpret as the step distribution of a right-continuous random walk, \ie a random walk that can perform arbitrary large downward jumps but can only jumps upward by $1$ (and, in this case, cannot stay still since $p_0 = 0$). We also set
\begin{equation}\label{defhtransform}
h(k)\defeq
\left\{\begin{array}{cl}
\displaystyle{\frac{\Gamma(k+\frac{1}{2})}{\Gamma(k)}}&\mbox{for}\quad k \geqslant 1,\\[0.3cm]
\displaystyle{0}&\mbox{for}\quad k \leqslant 0.
\end{array}\right.
\end{equation}
This function is sub-additive, non-decreasing and such that $h(k) \sim \sqrt{k}$ when $k$ goes to $+\infty$. We have the following description of the law of the peeling process. At each step, the probability transition $p_{n,m}$ to change the size of the outer boundary of the discovered triangulation from $n$ to $m$ is  given by the formula
\begin{equation}\label{lawmcpeeling}
p_{n,m} = \frac{h(m-1)}{h(n-1)} \,p_{m-n} \quad\mbox{for all $m,n \geqslant 2$.}
\end{equation} 
Moreover, when the size of the boundary is reduced by $n - m >0$ sites (a downward jump), then the finite triangulation revealed during the peeling process is distributed as a free colored Boltzmann triangulation of the $(n - m + 1)$-gon. On the other hand, when the size of the boundary increase by one, then the newly discovered vertex is colored independently in red or in blue with probability $1/2$. This description characterizes the law of the peeling process.  
\medskip

The function $h$ defined by~\eqref{defhtransform} is harmonic with respect to the random walk with step distribution $(p_k)_{k\in \Z}$ given by~\eqref{defrwpk}. This insures that~\eqref{lawmcpeeling} defines a proper transition kernel. Moreover, the particular form of Formula~\eqref{lawmcpeeling} reveals a remarkable property of this Markov chain: the function $h$ is the Doob's h-transform of the right-continuous random walk with step distribution $(p_k)_{k\in \Z}$, conditioned to stay positive\footnote{When peeling the UIHPT instead of the UIPT, there is an infinite outer boundary at all time so the increments of the boundary during the peeling are exactly those of the random walk with increment distributed according to $(p_k)_{k\in \Z}$,  making the analysis much simpler.}. Furthermore, one can verify that the sequence $(p_{-k})_{k\geqslant 1}$ is decreasing and we have 
\begin{equation}\label{asymppk}
\sum_{k\in \Z} k\, p_k = 0
\qquad \mbox{and}\qquad p_{-k} \underset{k\to+\infty}{\sim} \frac{1}{2\sqrt{\pi}k^{5/2}}.
\end{equation}
This shows that this random walk is centered and lies in the domain of normal attraction of a completely asymmetric distribution of index $3/2$ called the Airy law. These facts will play a major role in the rest of the paper.
\medskip

One of the strength of the peeling procedure is that the peeled edge may be chosen arbitrarily at each step. In the case of a percolated map, it is natural to choose it in such way that the red and blue vertices on the boundary of the revealed region remain separated at all time (\emph{i.e.} all the red vertices, if any, are adjacents). If all the sites on the boundary have the same color, then any arbitrary edge may be selected for peeling. One the other hand, when both colors are present on the boundary and are separated, then there are exactly two edges on the red/blue interface. By convention, we  choose the edge which, going counter-clockwise around the boundary, goes from blue to red. This strategy insures that the colors on the boundary remain separated on the next step.\medskip

Let us now fix some notations. For any $n\in \N$, we define
\begin{align*}
\Rg_n &\defeq \mbox{number of red sites on the boundary at step $n$},\\
\Bg_n &\defeq \mbox{number of blue sites on the boundary at step $n$},\\
\Sg_n &\defeq \Rg_n + \Bg_n = \mbox{size of the boundary at step $n$}.
\end{align*}
As we already mentioned, the length $\Sg$ of the peeling interface  is a Markov chain on $\{2,3,\ldots\}$ with transition probabilities $p_{n,m}$. Furthermore, it follows from the previous description that the pair $(\Rg,\Bg)$ also forms a Markov chain such that, for any $n\in \N$,
\begin{equation}\label{RBdef}
(\Rg_{n+1},\Bg_{n+1})  = f\big( (\Rg_n,\Bg_n) + (\Sg_{n+1} - \Sg_{n})(\eta_{n+1},1-\eta_{n+1})\big),
\end{equation}
where 
\begin{equation}\label{deffreflection}
f(r,b) \defeq  (r,b)\ind{r,b\geqslant 0} + (r+b,0)\ind{b < 0} + (0,b+r)\ind{r < 0}
\end{equation}
acts by reflecting $r$ and $b$ inside the first quadrant, and where $\eta_n$ indicates the color of the boundary sites which were ``concerned'' by the $n$-th step, \ie
\[
\eta_n \defeq \mathds{1}_{\left\{
\begin{array}{l}
\mbox{Either a new red vertex is discovered at the $n$-th step or the}\\
\mbox{discovered face reattaches itself to the boundary going counter-}\\
\mbox{clockise, \ie the vertices swallowed are on the left of the peeling}\\
\mbox{edge.}
\end{array}\right\}.
}
\]

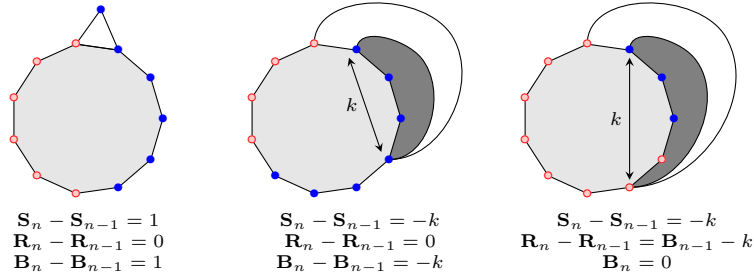
\begin{figure}[t]
\centering
\scriptsize{\begin{tabular}{cccc}

\begin{tikzpicture}[scale=1]
\foreach \k in {1,2,...,11} {\coordinate (A\k) at ({cos(360*(\k+7)/11)},{sin(360*(\k+7)/11)});}
\coordinate (D7) at ({cos(360*3/11)+2*sin(180/11)*cos(360*3/11+180/11-60)},{sin(360*3/11)+2*sin(180/11)*sin(360*3/11+180/11-60)});
\fill[color=gray!20]  (A1) -- (A2)--(A3)--(A4)--(A5)--(A6)--(A7)--(A8)-- (A9)--(A10)--(A11)--cycle;
\draw  (A7) -- (D7)--(A6)--cycle;
\draw  (A1) -- (A2)--(A3)--(A4)--(A5)--(A6)--(A7)--(A8)-- (A9)--(A10)--(A11)--cycle;
\foreach \k in {7,8,...,11,1} {\draw (A\k) node[color=red!20] {$\bullet$} node[color=red] {$\circ$};}
\foreach \k in {2,...,6} {\draw[color=blue]  (A\k) node {$\bullet$};}
\draw[color=blue]  (D7) node {$\bullet$};
\end{tikzpicture}

&~&

\begin{tikzpicture}[scale=1]
%
\foreach \k in {1,2,...,11} {\coordinate (A\k) at ({cos(360*(\k+7)/11)},{sin(360*(\k+7)/11)});}
%
\fill[color=gray!20]  (A1) -- (A2)--(A3)--(A4)--(A5)--(A6)--(A7)--(A8)-- (A9)--(A10)--(A11)--cycle;
%
\coordinate (B5) at ({2*cos(360*12/11)},{2*sin(360*(5+7)/11)});
\coordinate (C5) at ({1.5*cos(360*12/11)},{1.5*sin(360*(5+7)/11)});
\fill[color=gray!100]  (A6) to[out=90,in=130] (C5) to[out=-50,in=0] (A3) -- (A4) -- (A5)--(A6);
\draw (A7) to[out=90,in=120] (B5);
\draw (B5) to[out=-60,in=0] (A3);
\draw (A6) to[out=90,in=130] (C5);
\draw (C5) to[out=-50,in=0] (A3);
%
\draw  (A1) -- (A2)--(A3)--(A4)--(A5)--(A6)--(A7)--(A8)-- (A9)--(A10)--(A11)--cycle;
%
\foreach \k in {7,8,...,10} {\draw (A\k) node[color=red!20] {$\bullet$} node[color=red] {$\circ$};}
\foreach \k in {11,1,2,...,6} {\draw[color=blue]  (A\k) node {$\bullet$};}
%
\draw[>=stealth,<->] ({cos(360*(6+7)/11)-0.1},{sin(360*(6+7)/11)-0.1}) to node[pos=0.5,left]{$k$} ({cos(360*(3+7)/11)-0.1},{sin(360*(3+7)/11)+0.1});
\end{tikzpicture}

&

\begin{tikzpicture}[scale=1]
\foreach \k in {1,2,...,11} {\coordinate (A\k) at ({cos(360*(\k+7)/11)},{sin(360*(\k+7)/11)});}
\fill[color=gray!20]  (A1) -- (A2)--(A3)--(A4)--(A5)--(A6)--(A7)--(A8)-- (A9)--(A10)--(A11)--cycle;
\coordinate (B5) at ({2*cos(360*12/11)},{2*sin(360*(5+7)/11)});
\coordinate (C5) at ({1.5*cos(360*12/11)},{1.5*sin(360*(5+7)/11)});
\fill[color=gray!100]  (A6) to[out=90,in=130] (C5) to[out=-50,in=0] (A2) -- (A3)--(A4) -- (A5)--(A6);
\draw (A7) to[out=90,in=120] (B5);
\draw (B5) to[out=-60,in=0] (A2);
\draw (A6) to[out=90,in=130] (C5);
\draw (C5) to[out=-50,in=0] (A2);
\draw  (A1) -- (A2)--(A3)--(A4)--(A5)--(A6)--(A7)--(A8)-- (A9)--(A10)--(A11)--cycle;
\foreach \k in {7,8,...,11,1,2,3} {\draw (A\k) node[color=red!20] {$\bullet$} node[color=red] {$\circ$};}
\foreach \k in {4,5,6} {\draw[color=blue]  (A\k) node {$\bullet$};}
\draw[>=stealth,<->] ({cos(360*(6+7)/11)},{sin(360*(6+7)/11)-0.1}) to node[pos=0.5,left]{$k$} ({cos(360*(2+7)/11)},{sin(360*(2+7)/11)+0.1});
\end{tikzpicture}

\\
$\Sg_n-\Sg_{n-1}=1$ &~&$\Sg_{n}-\Sg_{n-1}=-k$&$\Sg_{n}-\Sg_{n-1}=-k$\\
$\Rg_n-\Rg_{n-1}=0$ &~&$\Rg_n-\Rg_{n-1}=0$&$\Rg_n-\Rg_{n-1}=\Bg_{n-1}-k$\\
$\Bg_n-\Bg_{n-1}=1$ &~&$\Bg_n-\Bg_{n-1}=-k$ &$\Bg_n=0$
\end{tabular}}
\caption{Different cases when peeling a percolation cluster on the UIPT at the $n$-th step (case where $\eta_n=0$). The light gray parts are the faces discovered before this step. The dark gray parts are the finite regions (free Boltzmann triangulations) discovered at this step.}
\label{fig:PeelingPerco}
\end{figure}

We call a step with $\eta = 1$ (resp. $\eta = 0$) a red step (resp. blue step). Beware however that a red (resp. blue) step needs not change the number of red  (resp. blue) vertices on the boundary if there are none and it can also reduce the number of vertices of the opposite color.
\medskip

Since we consider an i.i.d.~critical percolation of the UIPT with $p_c = 1/2$, the sequence $(\eta_i)_{i\geqslant 1}$ is i.i.d.~Bernoulli with parameter $1/2$ and is independent of the Markov chain $\Sg$.  
\medskip

We introduce the peeling time
\[
\thetag \defeq \inf\,\{\,n \geqslant 1 : \Rg_n = 0\,\}
\]
which corresponds to the time the red cluster of the origin is discovered. Since there is no infinite component at criticality, the peeling time $\thetag$ is almost surely finite. In order to study the volume of the red cluster, it is important to control precisely the tail distribution of $\thetag$. We will prove the following estimate.

\begin{figure}[t]
\centering
\newcommand{\dgreen}{green!60!black}
\begin{tikzpicture}[scale=0.27]
%
\fill[color=gray!50] (12,12)--(10,11)--(7,11)--(7.5,8.5)--(10,6)--(12,8)--(12,12);
\fill[color=gray!50] (8,14)--(4,12)--(5,17)--(9,20)--(13,22)--(16,21)--(17,16)--(13,16)--(9,17)--(8,14);
\fill[color=gray!50] (12,12)--(15,9)--(14,6)--(17,5)--(20,4)--(18,9)--(17,13.5)--(12,14)--(12,12);
\fill[color=gray!50] (19,13)--(20,16)--(17,16)--cycle;
\fill[color=gray!50] (27,10)--(24,8.5)--(23,11)--cycle;
%
\draw (12,12)--(15,9)--(17,5)--(14,6)--(15,9);
\draw (14,6)--(11,4)--(8,4)--(5,7)--(4,12)--(8,14)--(12,12);
\draw (8,14)--(9,17)--(4,12)--(5,17)--(9,17)--(9,20)--(5,17);
\draw (9,17) --(13,16)--(17,16)--(19,13)--(20,16)--(17,16) --(16,21)--(13,22)--(9,20);
\draw (20,9)--(19,13)--(20,16)--(23,15);
%
\draw (12,8)--(14,6);\draw (12,8)--(15,9);\draw (12,8)--(11,4);\draw (12,8)--(10,6);\draw (12,8)--(10,8);\draw (12,8)--(10,11);\draw (12,8)--(10,8);\draw (10,11)--(12,12);\draw (10,11)--(8,14);\draw (10,11)--(7,11);\draw (10,11)--(7.5,8.5);\draw (10,11)--(10,8);\draw (10,8)--(10,6);\draw (10,8)--(7.5,8.5);
\draw (7,11)--(8,14);\draw (7,11)--(4,12);\draw (7,11)--(5,7);\draw (7,11)--(7.5,8.5);\draw (7.5,8.5)--(5,7);\draw (7.5,8.5)--(10,6);\draw (7.5,8.5)--(8,4);\draw (10,6)--(8,4);\draw (10,6)--(11,4);\draw (11.5,19)--(13,22);\draw (11.5,19)--(9,20);\draw (11.5,19)--(9,17);\draw (11.5,19)--(13,16);\draw (15,18)--(13,16);\draw (15,18)--(17,16);\draw (15,18)--(16,21);\draw (15.5,12.5)--(12,14);\draw (15.5,12.5)--(16,10.5);\draw (15.5,12.5)--(17,13.5);\draw (16,10.5)--(17,13.5);\draw (11.5,19)to[out=20,in=180](13.5,19.5) to[out=0,in=80] (15,18);\draw (15,18)to[bend left](11.5,19);\draw(11.5,19)--(16,21);\draw (11.5,19)--(13.5,18.75)--(15,18);
%
\draw (26,14)--(24,18)--(21,20)--(19,21)--(17,25) --(13.5,25.5)--(10,24)--(7,23)--(4,20)--(0,18)--(2,15)--(1,11)--(3,9)--(2,6)--(6,4)--(8,2)--(12,2)--(14,0) --(17,1)--(20,4)--(23,6)--(27,10)--(23.5,12.75)--cycle;
\draw (20,9) to[out=60,in=-170] (22.75,11.75)to[out=10,in=145](27,10);\draw (20,9)to[out=-45,in=-160] (24,7.75) to[out=20,in=-140] (27,10);\draw(20,9)--(23,11)--(27,10)--(24,8.5)--(20,9);\draw(24,8.5)--(23,11);\draw (20,9)to[out=60,in=-170](23.5,12.75);
%
\draw(26,14)--(23,15)--(24,18)--(20,16)--(21,20)--(17,16)--(19,21)--(16,21)
--(17,25)--(13,22)--(13.5,25.5)--(10,24)--(13,22)--(7,23)--(9,20)--(4,20);
\draw (5,17)--(0,18);\draw (5,17)--(4,20);\draw (5,17)--(2,15);\draw (4,12)--(2,15);\draw (4,12)--(1,11);\draw (4,12)--(3,9);\draw (5,7)--(3,9);\draw (5,7)--(2,6);\draw (5,7)--(6,4); \draw (6,4)--(8,4)--(8,2)--(11,4)-- (12,2)--(14,6)--(14,0)--(17,5); \draw (17,5)--(17,1);\draw (17,5)--(20,4)--(18,9);\draw (20,4)--(20,9);\draw (17,5)--(18,9);\draw (12,14)--(9,17);\draw (12,14)--(8,14);\draw (16,10.5)--(18,9);\draw (16,10.5)--(15,9);\draw (16,10.5)--(12,14);\draw (12,14)--(12,12);\draw (12,14)--(15,9);\draw (12,14)--(17,13.5);\draw (12,14)--(17,16);\draw (12,14)--(13,16);\draw (17,13.5)--(17,16);\draw (17,13.5)--(19,13);\draw (17,13.5)--(18,9);\draw (17,13.5)--(20,9);\draw (18,9)--(15,9);\draw (18,9)--(20,9);\draw (20,9)--(23,6);\draw (23.5,12.75)--(19,13);\draw (23.5,12.75)--(20,16);\draw (23.5,12.75)--(23,15);\draw (23.5,12.75)--(26,14);\draw (0,18) to[bend right] (5,17) to[bend right] (0,18);
%
\foreach \i in {(4,12),(5,7),(5,17),(8,4),(8,14),(9,17),(9,20),(10,8),(11,4),(12,12),(13,16),(13,22),(14,6),(15,9),(17,5),(16,21),(17,16),(19,13),(20,9),(20,16),(23,15),(15.5,12.5),(13.5,18.75)}{\draw \i node[color=red!10] {$\bullet$} node[color=red] {$\circ$};}
\foreach \i in {(0,18),(1,11),(2,6),(2,15),(3,9),(4,20),(6,4),(7,11),(7,23),(8,2),(7.5,8.5),(10,6),(10,11),(10,24),(12,2),(12,8),(11.5,19),(12,14),(14,0),(13.5,25.5),(15,18),(17,13.5),(17,1),(18,9),(17,25),(19,21),(20,4),(21,20),(23.5,12.75),(23,6),(27,10),(24,18),(26,14),(16,10.5),(2.5,17.5),(23,11),(24,8.5)} {\draw \i node[color=blue] {$\bullet$} node[color=blue] {$\circ$};}
\draw[color=\dgreen,thick,dashed] (12,10) to[out=0,in=0] (10.5,5);\draw[color=\dgreen,fleche=0.5:\dgreen,thick,dashed] (10.5,5) to[out=180,in=-100] (5.5,11.5);\draw[color=\dgreen,thick,dashed] (5.5,11.5) to[out=80,in=180] (9,12.5) to[out=20,in=180] (10.5,15.5); \draw[color=\dgreen,fleche=0.55:\dgreen,thick,dashed] (10.5,15.5) to[out=0,in=160] (16.5,14.75);\draw[color=\dgreen,thick,dashed](16.5,14.75) to[out=-20,in=90] (18.5,9) to[out=-90,in=-135] (21.5,7.5)to[out=45,in=-45] (22,11) to[out=135,in=-170] (23.5,14); \draw[color=\dgreen,fleche=0.35:\dgreen,thick,dashed] (23.5,14) to[out=10,in=-30] (23.5,16.5) to[out=150,in=-60] (16.5,23);\draw[color=\dgreen,fleche=0.75:\dgreen,thick,dashed] (16.5,23) to[out=120,in=45] (4.5,18.5); \draw[color=\dgreen,fleche=0.45:\dgreen,thick,dashed] (4.5,18.5) to[out=-135,in=90] (3,13) to[out=-90,in=120] (3.5,10.5) to[out=-60,in=90] (3.5,6.5) to[out=-90,in=-180] (5.5,5.5) to[out=0,in=150] (8,3) to[out=-30,in=-135]  (13,4);\draw[color=\dgreen,fleche=0.99:\dgreen,thick,dashed] (13,4) to[out=45,in=135] (18.5,2.5);
%
\draw[fleche=0.6:black] (12,12)--(12,8);
\draw[color=\dgreen] (13,9.67+0.2) node{\scriptsize{1}} (6.34+0.3,12.34-0.4) node{\scriptsize{10}} (18.34-0.325,10.5+0.5) node{\scriptsize{20}} (22.34,14.6) node{\scriptsize{30}}(14.25,24) node{\scriptsize{40}} (12.1,4) node{\scriptsize{60}} 
;
\draw[color=gray](10,12.34) node{\scriptsize{12}} (24.25,11.15) node{\scriptsize{26}} (19.34,17.34) node{\scriptsize{35}} (3.67,14.37) node{\scriptsize{50}}(18,3.34+0.5) node{\scriptsize{64}};
\end{tikzpicture}
\caption{The exploration of the percolation interface of the example in figure~\ref{fig:HullPerco} until time $\thetag$. The interface is in green. The
gray parts are the finite regions (free Boltzmann triangulations) discovered during the peeling process (\textit{they are revealed respectively at steps $12$, $26$, $35$, $50$ and $\thetag=64$}).}
\label{fig:ExPeeling}
\end{figure}
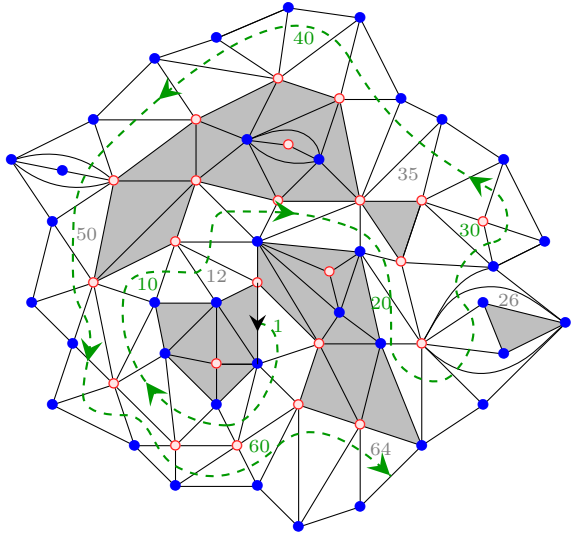

\begin{theo}\label{theoExploTime}
 Consider a critical percolation on the UIPT, we have
\[
\Pg(\thetag \geqslant n) \asymp \frac{1}{n^{1/6}}.
\]
\end{theo}

Let us point out that, just as for Theorem~\ref{mainTheo}, the exponent $1/6$ associated with the peeling of a critical cluster of the full plane map differs again by a factor $2$ from the exponent $1/3$ obtained in~\cite{AC15} for the peeling time of a critical percolation cluster of the UIHPT. Let us also remark that, a realization of the UIPT being given, the
peeling time depends on the peeling strategy adopted. However, the law of
$\thetag$ itself does not depend on the choices of the peeled edges provided that we always select, whenever possible, an edge belonging to the interface of the red/blue cluster.
\medskip 

We also define quantities associated with the volume of the triangulation:
\begin{align*}
\Vg_n &\defeq\mbox{number of sites discovered in the triangulation after the $n$-th peeling step,}\\
\Vgr_n &\defeq \sum_{i=1}^n (\Vg_{i}-\Vg_{i-1})\eta_i =\mbox{volume discovered by the red steps,}\\
\Vgb_n &\defeq \sum_{i=1}^n (\Vg_{i}-\Vg_{i-1})(1-\eta_i) =\mbox{volume discovered by the blue steps.}
\end{align*}
Notice that, whenever the peeling reveals a new site, the volume increases by one. On the other hand, when the newly discovered face reattaches itself to the boundary at distance $d-1\geqslant 1$ from the peeled edge, then the number of sites added to the volume is distributed as the number of inner sites of a free Boltzmann triangulation of the $d$-gon. Therefore, conditionally on $(\Rg,\Bg)$, the increments $(\Vg_{i}-\Vg_{i-1})_{i\geqslant 1}$ form a sequence of independent random variables such that, for any $i\geqslant 1$, $n\in \N$ and $d\geqslant 2$,
\begin{multline}\label{formulaVolume}
\Pg(\Vg_{i}-\Vg_{i-1} = n \,|\, \Sg_{i} - \Sg_{i-1} = 1-d) 
=
\Pg\left(
\begin{array}{c}
\hbox{\scriptsize a free Boltzmann triangulation of}\\[-0.1cm]
\hbox{\scriptsize the $d$-gon has $n$ inner vertices}
\end{array}
\right)\\
=
2\,\frac{\displaystyle{(2d-3)d(d-1)(2d+3n-4)!}}{\displaystyle{n!(2d+2n-2)!}}\,\left(\dfrac{4}{27}\right)^{n}\left(\dfrac{4}{9}\right)^{d-1}
\end{multline}
with the convention $\Vg_{i}-\Vg_{i-1}=1$ if $\Sg_{i} - \Sg_{i-1} = 1$. Using \eqref{formulaVolume}, we can explicitly compute the expectation (\emph{c.f.}~\cite{CLG16}):
\begin{equation}\label{formulaVolumeExpectation}
\Eg\left[\Vg_{i}-\Vg_{i-1}  \,|\, \Sg_{i} - \Sg_{i-1} = 1- d
\right] = \frac{1}{3}(d-1)(2d-3) \qquad\hbox{for $d\geq 2$.}
\end{equation}  
We are interested in $|\Hl|$, the number of sites of the hull of the red cluster. However, we do not have direct access to this quantity from the peeling process. Nevertheless, it is clear that the total volume $\Vg_{\thetag}$ of all sites discovered by the peeling process up to time $\thetag$ is larger than $|\Hl|$. One the other hand, it is easy to see that every site discovered by red steps of the peeling process up to time $\thetag-1$ necessarily belongs to $\Hl$. As a consequence, we have the upper and lower bounds
\begin{equation*}
\Vgr_{\thetag-1} \leqslant |\Hl| \leqslant \Vg_\thetag. 
\end{equation*}
The estimate~\eqref{resVolume} of Theorem~\ref{mainTheo} will follow directly from the next result whose proof consumes a significant part of the paper.
\begin{prop} \label{propVol} Consider a critical percolation on the UIPT, we have
\[
\Pg\big(\Vgr_{\thetag-1} \geqslant n\big) \asymp \Pg\big(\Vg_{\thetag} \geqslant n\big)\asymp
\frac{1}{n^{1/8}}.
\]
\end{prop}

\section{The perimeter \texorpdfstring{$\D$}{}}
\label{SectionPerimeter}

The aim of this section is to relate $\D$, the perimeter of the hull of the cluster, to the peeling time $\thetag$. It will enable us to focus only on $\thetag$ in latter sections. The following proposition, together with Theorem~\ref{theoExploTime}, directly implies~\eqref{resPerimeter} of Theorem~\ref{mainTheo}.

\begin{prop}\label{PropPerim} Consider the peeling process, starting from an initial red/blue separated boundary with at least one red vertex. Recall that $\thetag$ denotes the peeling time of the red cluster and $\D$ the set of edges between the hull $\Hl$ and its complement. There exists $c>0$ such that, for all $n$ large enough,
\begin{equation}\label{boundperimeter}
c\, \Pg\big(\thetag >  n \big) \leqslant \Pg\big(|\D| > n \big) \leqslant  2\, \Pg\big( \thetag > n/2-1 \big).
\end{equation}
\label{PropPerimeter}
\end{prop}

\subsection{Construction via a random walk}
\label{sectionDefHtransfrom}

In the previous section, we constructed the peeling process as a Markov chain defined on some abstract probability space. However, as we already mentioned, the law of the peeling process is related to that of a particular random walk $\Sd$ conditioned to stay positive. In order to exploit this property, it is convenient to introduce another probability measure $\P$ such that our original probability $\Pg$ may be formally written as 
\begin{equation}\label{idlaws}
\Pg\big(\,\cdot\, \big) = \P \big(\,\cdot\,|\, \Sd\geqslant 2\big).
\end{equation}
We will call $\Pg$  the \textit{peeling law} whereas  $\P$ will be called the \textit{random walk law}. 
\medskip

\textbf{Construction.} Fix a probability space $(\Omega,\mathcal{F},\P)$. The corresponding expectation is denoted by $\E$. 
When working with Markov processes, we will sometime use subscripts such as $\P_x$ and $\E_x$ to empathize the initial condition of processes (omitting to specify to which processes they apply when it is obvious). Let $\Sd = (\Sd_0,\Sd_1,\ldots)$ denote a random walk whose increments, under $\P$, are i.i.d.~and distributed according to the law $(p_k)_{k\in \Z}$ given by~\eqref{defrwpk}. We denote by $s_0$ the stating point of $\Sd$. We also consider a sequence $(\eta_i)_{i\geqslant 1}$ of i.i.d.~Bernoulli random variables with parameter $1/2$. We define the process $\Sg$ by
\[
\Sg_n \defeq  \Sd_n,\quad \mbox{for all $n\in \N$}.
\]
Using two notations for the same object may seems redundant but will prove useful later on. We will use the notation $\Sg$ when working under the peeling law $\Pg$ and the notation $\Sd$ when working under the random walk law $\P$. We also define the pair of processes $(\Rg,\Bg)$, as in~\eqref{RBdef}, by the recurrence relation
\begin{equation}\label{defRBwithS}
\left\{
\begin{array}{rcl}
(\Rg_{0},\Bg_{0})  &\defeq& (r_0,b_0), \\
(\Rg_{n+1},\Bg_{n+1})  &\defeq& f\big( (\Rg_n,\Bg_n) + (\Sd_{n+1} - \Sd_{n})(\eta_{n+1},1-\eta_{n+1})\big).
\end{array}
\right.
\end{equation}
where $f$ is the function given by~\eqref{deffreflection}.  We shall always enforce the relation $s_0 = b_0 + r_0$ so that $\Sg_n = \Rg_n + \Bg_n$ for all $n\in \N$. 
\medskip

According to the classical theory of Doob's h-transform, the formula~\eqref{lawmcpeeling} shows that the peeling law $\Pg$, introduced in the previous section, can be constructed from $\P$ via a change of measure with the Radon-Nikodym derivative given by the harmonic function $h$. In other word, suppose that we start the peeling process from an initially red/blue separated outer boundary of size $s_0 = r_0 + b_0$ with $b_0$ blue vertices and $r_0$ red vertices (the case $s_0 = 2$ corresponds to the peeling starting from a single edge but, due to its Markovian nature, it makes sense to start it from arbitrary outer boundary). Then, by starting the process $(\Sg,\Rg,\Bg)$ from $(r_0+b_0,r_0,b_0)$, we have that, for any non-negative test function $F$,
\begin{multline}
\label{htransformeq}
\Eg_{(r_0,b_0)}[F\big((\Sg_k,\Rg_k,\Bg_k)_{k\leqslant n}\big)] \\
\begin{aligned}
&=\E_{(r_0,b_0)}\left[\frac{h(\Sd_n - 1)}{h(\Sd_0 - 1)}F\big((\Sd_k,\Rg_k,\Bg_k)_{k\leqslant n}\big)\ind{\Sd_0,\ldots,\Sd_n \geqslant 2}\right]\\
&=\lim_{m\to\infty} \E_{(r_0,b_0)}\big[F\big((\Sd_k,\Rg_k,\Bg_k)_{k\leqslant n}\big)\, \big| \, \Sd_0,\ldots,\Sd_m \geqslant 2\big].
\end{aligned}
\end{multline}
The second equality states that, in a weak limit sense, the law of $\Sg$ under $\Pg$ is indeed the law of the random walk $\Sd$ ``conditioned to stay inside $[2,+\infty)$ forever''. Moreover, since $\Sd$ is oscillating, this law also coincides with the weak limit for the law of the random walk conditioned to reach arbitrary high heights before going below $2$, see for instance~\cite{BD94} for details. By convention, we shall now write 
\begin{equation}\label{defcondiinfini}
\P\big( \,\cdot\, \big| \, \Sd \geqslant 2 \big) \defeq \lim_{m\to\infty} \P\big( \,\cdot\,   \big| \, \Sd_0,\ldots,\Sd_m \geqslant 2\big)
\end{equation}
This notation, together with~\eqref{htransformeq}, makes~\eqref{idlaws} rigorous. From now on, we will use interchangeably $\Pg$ or $\P \big(\,\cdot\,|\, \Sd\geqslant 2\big)$, usually favouring $\Pg$ to shorten formula but using $\P \big(\,\cdot\,|\, \Sd\geqslant 2\big)$ when we make use of the specific form of the probability measure.\medskip

We need two lemmas before we can provide the proof of Proposition~\ref{boundperimeter}.

\begin{lem} For any initial starting point $s_0\geqslant 2$, the sequence of increments $(\Sg_{n+1} - \Sg_n)_{n\in \N}$ under $\Pg_{s_0}$ stochastically dominates the sequence $(\Sd_{n+1} - \Sd_n)_{n\in \N}$ of increments of the non-conditioned random walk under $\P_{s_0}$. 
\label{lemStocDomS}
\end{lem}

\begin{proof} This result is well known. It follows directly from the explicit formula~\eqref{lawmcpeeling} for the transition kernel of $\Sg$ under $\Pg$ and from the fact that $h$ is non-decreasing.
\end{proof}

The next lemma is rather intuitive. It states, informally, that while peeling a red cluster, we should expect to see more red vertices than blue ones on the boundary. 

\begin{lem}\label{OnionLemma} Consider the peeling process starting from an initial outer boundary with $r_0$ red vertices and $b_0$ blue vertices such that $r_0 \geqslant b_0$. Fix $k \geqslant n\geqslant 0$. Then, under $\Pg_{(b_0,r_0)}$, conditionally on $\Sg$ and $\{\thetag > n\}$, the random variable $\Rg_k$ stochastically dominates $\Bg_k$.
\end{lem}

\begin{proof}
We first prove the result when $k=n$, by induction on $n$. We must show that
\[
\forall x \geqslant 0\qquad\Pg\Big(\Rg_n \geqslant x,\thetag>n\,\Big|\,\Sg\Big)\geqslant
\Pg\Big(\Bg_n\geqslant x,\thetag>n\,\Big|\,\Sg\Big)
\]
(in this proof, the processes always start from $r_0$ and $b_0$ so we just write $\Pg$ instead for $\Pg_{(r_0,b_0)}$). The result holds for $n=0$ since $r_0 \geqslant b_0$. To prove the induction step, we decompose
\[
\{\thetag>n+1\}=\bigg\{\begin{array}{c}
\thetag>n,\,\eta_{n+1}=1,\\
\Rg_n+(\Sg_{n+1}-\Sg_n) > 0
\end{array}\bigg\}\,\bigcup\,
\Big\{\thetag>n,\,\eta_{n+1}=0 \Big\}.
\]
Therefore, we can write
\begin{align*}
\Pg\Big(\!\Rg_{n+1}\geqslant x,\,\thetag>n+1\,\Big|\,\Sg\!\Big) &= \Pg\Big(\!\Rg_n\!+\!(\Sg_{n+1}-\Sg_n)\geqslant x,\,\thetag>n,\,\eta_{n+1}=1\,\Big|\,\Sg\!\Big)\\
&\qquad + \Pg\Big(\Rg_n\geqslant x,\,\thetag>n,\,\eta_{n+1}=0\,\Big|\,\Sg\Big)\ind{\Sg_{n+1} \geqslant x}\\
&=\frac{1}{2}\Pg\Big(\Rg_n+(\Sg_{n+1}-\Sg_n) \geqslant x,\,\thetag>n\,\Big|\,\Sg\Big)\\
&\qquad +\frac{1}{2}\Pg\Big(\Rg_n \geqslant x,\,\thetag>n\,\Big|\,\Sg\Big)\ind{\Sg_{n+1} \geqslant x},
\end{align*}
where we used the fact that the sequence $(\eta_n)_{n\geqslant 1}$ is Bernoulli of parameter $1/2$ and independent of $\Sg$ for the last equality. Similarly, we can also write
\begin{align*}
&\Pg\Big(\Bg_{n+1}\geqslant x,\,\thetag>n+1\,\Big|\,\Sg\Big)\\
&= \Pg\Big(\Bg_n+(\Sg_{n+1}-\Sg_n)\geqslant x,\,\thetag>n,\,\eta_{n+1}=0\,\Big|\,\Sg\Big)\\
&\qquad + \Pg\Big(\Bg_n \geqslant x,\, \Rg_n + (\Sg_{n+1}-\Sg_n)>0 ,\,\thetag>n,\,\eta_{n+1}=1\,\Big|\,\Sg\Big)\ind{\Sg_{n+1} \geqslant x}\\
&\leqslant \frac{1}{2}\!\left(\Pg\Big(\Bg_n+(\Sg_{n+1}-\Sg_n) \geqslant x,\,\thetag>n\Big) \!+\! \Pg\Big(\Bg_n \geqslant x,\,\thetag>n\,\Big|\,\Sg\Big)\ind{\Sg_{n+1} \geqslant x}\!\right).
\end{align*}
Using the induction hypothesis, we conclude that, as required, 
\begin{align*}
\Pg\Big(\Bg_{n+1}\geqslant x,\,\thetag>n+1\,\Big|\,\Sg\Big) &\leqslant \frac{1}{2}\,\Pg\Big(\Rg_n+(\Sg_{n+1}-\Sg_n)\geqslant x,\,\thetag>n\,\Big|\,\Sg\Big)\\
& \qquad + \frac{1}{2}\,\Pg\Big(\Rg_n \geqslant x,\,\thetag>n\,\Big|\,\Sg\Big)\ind{\Sg_{n+1} \geqslant x}\\
&=\Pg\Big(\Rg_{n+1}\geqslant x,\,\thetag>n+1\,\Big|\,\Sg\Big).
\end{align*}
We now prove that the result still holds for any $k\geqslant n$, for $n$ fixed, by induction on $k$. Indeed, suppose that
\begin{equation}\label{onionInduc2}
\forall x\geqslant 0\qquad \Pg\Big(\Rg_{k}\geqslant x,\,\thetag>n\,\Big|\,\Sg\Big) \geqslant \Pg\Big(\Bg_{k}\geqslant x,\,\thetag>n\,\Big|\,\Sg\Big).
\end{equation}
Using similar argument as before, we can write
\begin{multline*}
\Pg\Big(\Rg_{k+1}\geqslant x,\,\thetag>n\,\Big|\,\Sg\Big) = \frac{1}{2}\Pg\Big(\Rg_{k}\geqslant x,\,\thetag>n \,\Big|\,\Sg\Big) \\
+ \frac{1}{2}\Pg\Big(\Rg_{k}  \geqslant x - (\Sg_{n+1} -\Sg_n),\,\thetag>n\, \Big|\,\Sg\Big)
\end{multline*}
but also 
\begin{multline*}
\Pg\Big(\Bg_{k+1}\geqslant x,\,\thetag>n\,\Big|\,\Sg\Big) = \frac{1}{2}\Pg\Big(\Bg_{k}\geqslant x,\,\thetag>n \,\Big|\,\Sg\Big) \\
+  \frac{1}{2}\Pg\Big(\Bg_{k} \geqslant x - (\Sg_{n+1} -\Sg_n),\,\thetag>n\, \Big|\,\Sg\Big)
\end{multline*}
from which we deduce that the~\eqref{onionInduc2} also holds for $k+1$. 
\end{proof}

\subsection{Proof of proposition~\ref{PropPerimeter}}

We first note that, during the exploration of the red cluster of the origin, each step of the peeling process discovers exactly one new edge that connects a blue vertex to a red vertex, except for the special times when the entire border of the unexplored region is red. Let $\Deltag$ be the last time before $\thetag$ such that this happens, \ie such that $\Bg_{\Deltag} =0$. At this step, the entire border is red and, at the next step, we discover a blue vertex which necessarily belongs to $\Hl^c$ (otherwise it would be absorbed by the cluster at a latter time and it would be a time where the whole boundary is red). Therefore the peeling at step $\Deltag+1$ creates two edges counted in $\D$. Thus, from step $\Deltag+1$, we continue to peel the cluster, using the usual condition that we always choose the one edge which goes from blue to red when rotating around the boundary counter-clockwise. After the time $\Deltag$, there are always both blue and red vertices on the boundary so that, at each step, we discover one new edge in $\D$ until time $\thetag$. Finally, at the last step $\thetag$, the whole boundary become blue and we must add to $\D$ an additional number $\Eg$ of edges that belong to the free Boltzmann triangulation discovered at this step, see figure~\ref{fig:Perimeter}. As a consequence, we obtain
\begin{equation}\label{eqPerim}
|\D|=1+ \thetag - \Deltag + \Eg
\end{equation}
(to make this equality rigorous, we use the convention $\Deltag = 1$ if we never encounter a fully red boundary during the peeling of the red cluster).

\begin{figure}[ht]
\centering
\newcommand{\dgreen}{green!60!black}
\begin{tikzpicture}[scale=0.28]
%
\fill[color=gray!50] (12,12)--(10,11)--(7,11)--(7.5,8.5)--(10,6)--(12,8)--(12,12);
\fill[color=gray!50] (8,14)--(4,12)--(5,17)--(9,20)--(13,22)--(16,21)--(17,16)--(13,16)--(9,17)--(8,14);
\fill[color=gray!50] (12,12)--(15,9)--(14,6)--(17,5)--(20,4)--(18,9)--(17,13.5)--(12,14)--(12,12);
\fill[color=gray!50] (19,13)--(20,16)--(17,16)--cycle;
\fill[color=gray!50] (27,10)--(24,8.5)--(23,11)--cycle;
%
\draw (12,12)--(15,9)--(17,5)--(14,6)--(15,9);
\draw (14,6)--(11,4)--(8,4)--(5,7)--(4,12)--(8,14)--(12,12);
\draw (8,14)--(9,17)--(4,12)--(5,17)--(9,17)--(9,20)--(5,17);
\draw (9,17) --(13,16)--(17,16)--(19,13)--(20,16)--(17,16) --(16,21)--(13,22)--(9,20);
\draw (20,9)--(19,13)--(20,16)--(23,15);
%
\draw (12,8)--(14,6);\draw (12,8)--(15,9);\draw (12,8)--(11,4);\draw (12,8)--(10,6);\draw (12,8)--(10,8);\draw (12,8)--(10,11);\draw (12,8)--(10,8);\draw (10,11)--(12,12);\draw (10,11)--(8,14);\draw (10,11)--(7,11);\draw (10,11)--(7.5,8.5);\draw (10,11)--(10,8);\draw (10,8)--(10,6);\draw (10,8)--(7.5,8.5);
\draw (7,11)--(8,14);\draw (7,11)--(4,12);\draw (7,11)--(5,7);\draw (7,11)--(7.5,8.5);\draw (7.5,8.5)--(5,7);\draw (7.5,8.5)--(10,6);\draw (7.5,8.5)--(8,4);\draw (10,6)--(8,4);\draw (10,6)--(11,4);\draw (11.5,19)--(13,22);\draw (11.5,19)--(9,20);\draw (11.5,19)--(9,17);\draw (11.5,19)--(13,16);\draw (15,18)--(13,16);\draw (15,18)--(17,16);\draw (15,18)--(16,21);\draw (15.5,12.5)--(12,14);\draw (15.5,12.5)--(16,10.5);\draw (15.5,12.5)--(17,13.5);\draw (16,10.5)--(17,13.5);\draw (11.5,19)to[out=20,in=180](13.5,19.5) to[out=0,in=80] (15,18);\draw (15,18)to[bend left](11.5,19);\draw(11.5,19)--(16,21);\draw (11.5,19)--(13.5,18.75)--(15,18);
%
\draw (26,14)--(24,18)--(21,20)--(19,21)--(17,25) --(13.5,25.5)--(10,24)--(7,23)--(4,20)--(0,18)--(2,15)--(1,11)--(3,9)--(2,6)--(6,4)--(8,2)--(12,2)--(14,0) --(17,1)--(20,4)--(23,6)--(27,10)--(23.5,12.75)--cycle;
\draw (20,9) to[out=60,in=-170] (22.75,11.75)to[out=10,in=145](27,10);\draw (20,9)to[out=-45,in=-160] (24,7.75) to[out=20,in=-140] (27,10);\draw(20,9)--(23,11)--(27,10)--(24,8.5)--(20,9);\draw(24,8.5)--(23,11);\draw (20,9)to[out=60,in=-170](23.5,12.75);
%
\draw(26,14)--(23,15)--(24,18)--(20,16)--(21,20)--(17,16)--(19,21)--(16,21)
--(17,25)--(13,22)--(13.5,25.5)--(10,24)--(13,22)--(7,23)--(9,20)--(4,20);
\draw (5,17)--(0,18);\draw (5,17)--(4,20);\draw (5,17)--(2,15);\draw (4,12)--(2,15);\draw (4,12)--(1,11);\draw (4,12)--(3,9);\draw (5,7)--(3,9);\draw (5,7)--(2,6);\draw (5,7)--(6,4); \draw (6,4)--(8,4)--(8,2)--(11,4)-- (12,2)--(14,6)--(14,0)--(17,5); \draw (17,5)--(17,1);\draw (17,5)--(20,4)--(18,9);\draw (20,4)--(20,9);\draw (17,5)--(18,9);\draw (12,14)--(9,17);\draw (12,14)--(8,14);\draw (16,10.5)--(18,9);\draw (16,10.5)--(15,9);\draw (16,10.5)--(12,14);\draw (12,14)--(12,12);\draw (12,14)--(15,9);\draw (12,14)--(17,13.5);\draw (12,14)--(17,16);\draw (12,14)--(13,16);\draw (17,13.5)--(17,16);\draw (17,13.5)--(19,13);\draw (17,13.5)--(18,9);\draw (17,13.5)--(20,9);\draw (18,9)--(15,9);\draw (18,9)--(20,9);\draw (20,9)--(23,6);\draw (23.5,12.75)--(19,13);\draw (23.5,12.75)--(20,16);\draw (23.5,12.75)--(23,15);\draw (23.5,12.75)--(26,14);\draw (0,18) to[bend right] (5,17) to[bend right] (0,18);
%
\draw[color=violet,very thick] (17,5)--(18,9)--(15,9)--(16,10.5);\draw[color=violet,very thick] (15,9)--(12,14);
%
\foreach \i in {(4,12),(5,7),(5,17),(8,4),(8,14),(9,17),(9,20),(10,8),(11,4),(12,12),(13,16),(13,22),(14,6),(15,9),(17,5),(16,21),(17,16),(19,13),(20,9),(20,16),(23,15),(15.5,12.5),(13.5,18.75)}{\draw \i node[color=red!10] {$\bullet$} node[color=red] {$\circ$};}
\foreach \i in {(0,18),(1,11),(2,6),(2,15),(3,9),(4,20),(6,4),(7,11),(7,23),(8,2),(7.5,8.5),(10,6),(10,11),(10,24),(12,2),(12,8),(11.5,19),(12,14),(14,0),(13.5,25.5),(15,18),(17,13.5),(17,1),(18,9),(17,25),(19,21),(20,4),(21,20),(23.5,12.75),(23,6),(27,10),(24,18),(26,14),(16,10.5),(2.5,17.5),(23,11),(24,8.5)} {\draw \i node[color=blue] {$\bullet$} node[color=blue] {$\circ$};}
%
\draw[color=\dgreen,thick,dashed] (12,10) to[out=0,in=0] (10.5,5);\draw[color=\dgreen,fleche=0.5:\dgreen,thick,dashed] (10.5,5) to[out=180,in=-100] (5.5,11.5);
\draw[color=\dgreen,thick,dashed] (5.5,11.5) to[out=80,in=180] (9,12.5) to[out=0,in=-120] (10,13);
\draw[color=brown,thick,dashed](10,13) to[out=60,in=-160] (12.5,15);
\draw[color=brown,fleche=0.5:brown,thick,dashed] (12.5,15) to[out=20,in=140] (17,14.75);\draw[color=brown,thick,dashed](16.5,14.75) to[out=-20,in=90] (18.5,9) to[out=-90,in=-135] (21.5,7.5)to[out=45,in=-45] (22,11) to[out=135,in=-170] (23.5,14); \draw[color=brown,fleche=0.35:brown,thick,dashed] (23.5,14) to[out=10,in=-30] (23.5,16.5) to[out=150,in=-60] (16.5,23);\draw[color=brown,fleche=0.75:brown,thick,dashed] (16.5,23) to[out=120,in=45] (4.5,18.5) ; \draw[color=brown,fleche=0.45:brown,thick,dashed] (4.5,18.5) to[out=-135,in=90] (3,13) to[out=-90,in=120] (3.5,10.5) to[out=-60,in=90] (3.5,6.5) to[out=-90,in=-180] (5.5,5.5) to[out=0,in=150] (8,3) to[out=-30,in=-135]  (13,4);\draw[color=brown,fleche=0.99:brown,thick,dashed] (13,4) to[out=45,in=135] (18.5,2.5);
%
\draw[fleche=0.6:black] (12,12)--(12,8);
\end{tikzpicture}
\caption{The exploration of the percolation interface of the example in figure~\ref{fig:HullPerco} until time $\thetag$. The interface until time $\Deltag$ is in green. The interface during time $\thetag-\Deltag$ is in brown. We add $\Eg$ edges (in violet) that belong to the free Boltzmann triangulation discovered at step $\thetag$. \textit{Here, $\Deltag=12$, $\thetag=64$ and $\Eg=4$}.}
\label{fig:Perimeter}
\end{figure}
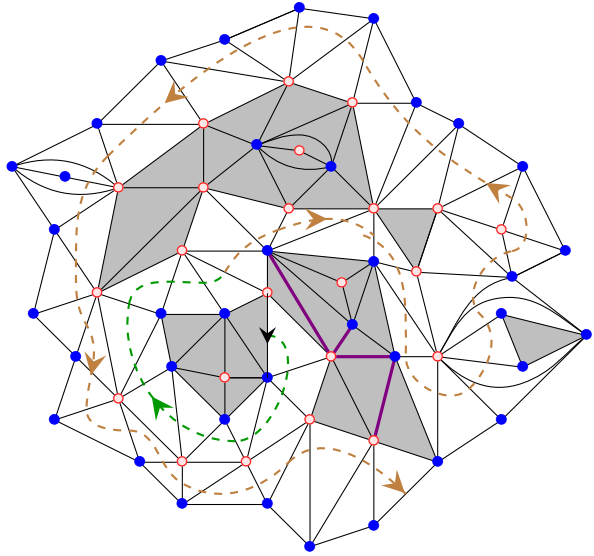

\subsubsection{Proof of the upper bound}

Evaluating the tail distribution of the number $\Eg$ of edges would be burdensome. Instead we rely on a clever symmetry trick we learned from Angel and Curien~\cite{AC15}. Observe that, after time $\Deltag$, we could have chosen the other convention for the peeling process (\ie choosing at each step the edge going from blue to red when rotating clockwise around the boundary). This would have yielded another peeling time $\smash{\widetilde{\thetag}}$ and another number of edges added at the last step $\smash{\widetilde{\Eg}}$. Yet, the complete symmetry of these procedures (we know that after this time peeling to the right or to the left we will never make the blue vertices disappear) implies that $\thetag$ and $\smash{\widetilde{\thetag}}$ have the same law. Moreover these two peeling strategies completely surround the cluster which, in turn, implies that 
\begin{equation}
|\D|\leqslant 1+ \thetag - \Deltag + 1+ \widetilde{\thetag} - \Deltag  \leqslant 2 + \thetag + \widetilde{\thetag}.
\end{equation}
The upper bound of the proposition directly follows from a crude union bound on the right hand side of the inequality above.

\subsubsection{Proof of the lower bound}

The lower bound is more delicate and its proof is decomposed in several steps. First, we can assume without loss of generality that the initial boundary condition is $r_0 = b_0 = 1$ since this configuration can be reached with positive probability while peeling the red cluster after a finite number of step. Thus the Markov property of the peeling process shows that changing the initial condition affects only the multiplicative constant $c$ in the lower bound of~\eqref{boundperimeter}. Now, let us construct an event on which $|\D|$ is larger than $n$ and which has the required probability. Fix $a>0$ (which we will chosen sufficiently large later on) and define
\begin{equation*}
\mathcal{E} \defeq \mathcal{E}_1 \cap \mathcal{E}_2 \cap \mathcal{E}_3
\end{equation*}
with
\begin{align*}
\mathcal{E}_{1} &\defeq  \big\{ \thetag > n,\, \Rg_n \geqslant 5 a n^{2/3} \big\}, \\
\mathcal{E}_{2} &\defeq  \big\{ \hbox{the first time $\tau \geqslant n$ such that $\Bg_\tau \geqslant \Rg_\tau$ also satisfies $\Rg_\tau\geqslant a n^{2/3}$}\big\}, \\
\mathcal{E}_{3} &\defeq  \big\{ \hbox{After time $\tau$, $\Rg$ hits $0$ before $\Bg$ but not before time $\tau + n$} \big\}.
\end{align*}
On the event $\mathcal{E}$, we have $\thetag \geqslant \tau + n$. On the other hand, the last time $\Deltag$ when the outer boundary is fully red must occur prior to time $\tau$. Therefore, on this event, we have $\thetag - \Deltag \geqslant n$. In view of~\eqref{eqPerim}, this means that
\begin{equation*}
\Pg(\mathcal{E}) \leqslant \Pg(|\D| > n).
\end{equation*} 
We lower bound the probability of each event composing $\mathcal{E}$ separately.
\medskip

\textbf{\textit{Event $\mathcal{E}_1$.}} Denote $\eta=(\eta_i)_{i\geqslant 1}$. Recalling the notation~\eqref{defcondiinfini}, we can write
\begin{align*}
\Pg\big(\thetag > n,\, \Sg_n \geqslant 10 a n^{2/3} \big) &=\P\big(\thetag > n,\, \Sd_n \geqslant 10 a n^{2/3} \,\big|\, \Sd \geqslant 2 \big)\\
& = \E\Big[ \P\big(\thetag > n,\, \Sd_n \geqslant 10 a n^{2/3} \,\big|\, \Sd \geqslant 2 ,\,\eta \big) \Big]
\end{align*}
Conditionally on $\eta$, the events $\{\Sd_n \geqslant 10 a n^{2/3}\}$ and $\{\thetag >n\}$ are increasing with respect to the increments of the random walk $\Sd$. Indeed, larger increments  always mean larger values for $\Rd$ and $\Bd$ (even on the event that one of them become negative) and thus $\Rd$ is less likely to vanish. Thus, according to Harris inequality recalled in Proposition~\ref{HarrisProp} of the appendix (more precisely  formulation \eqref{harriseqcondi}), we have
\begin{align*}
\Pg\big(\thetag > n,\, \Sg_n \geqslant 10 a n^{2/3} \big) &\geqslant  \E\Big[ \P\big(\thetag > n\, \big|\, \Sd \geqslant 2 ,\,\eta \big) 
\,\P\big(\Sd_n \geqslant 10 a n^{2/3}\, \big|\,\eta\big)\Big]\\
&= \Pg\big(\thetag > n \big) \,\P\big(\Sd_n \geqslant 10 a n^{2/3}\big),
\end{align*}
where we used that $\Sd$ does not depend on $\eta$ for the last equality. The random walk $\Sd$ is centered and lies in the domain of normal attraction of a spectrally negative stable law of index $3/2$. Therefore, the probability $\P(\Sd_n \geqslant 10 a n^{2/3})$ converges to a strictly positive constant (depending only on $a$) when $n$ goes to infinity. On the other hand, since we can assume that we start with as many red as blue vertices on the initial boundary, Lemma~\ref{OnionLemma} shows that
\[
\Pg\big(\mathcal{E}_1 \big) = \Pg\big(\thetag > n,\, \Rg_n \geqslant  5 a n^{2/3} \big) \geqslant \frac{1}{2}\,\Pg\big(\thetag > n,\, \Sg_n \geqslant 10 a n^{2/3} \big).
\]
Putting these facts together yields the lower bound: there exists some constant $c>0$ (depending only on $a$) such that, for $n$ large,
\begin{equation}\label{boundE1}
\Pg\big(\mathcal{E}_1 \big) \geqslant c\, \Pg\big(\thetag > n \big).
\end{equation}

\textbf{\textit{Event $\mathcal{E}_2$.}} Let us now assume that the peeling process starts from some initial boundary condition $(r_0,b_0)$ with $r_0 \geqslant 5 a n^{2/3}$ and $b_0 \geqslant 0$. Define the stopping time $\tau \defeq \inf\,\{\,k\geqslant 0: \Bg_k \geqslant \Rg_k \,\}$ and observe that, necessarily, $\tau \leqslant \thetag$. We will prove that, for $n$ large, 
\begin{equation}\label{toproveE2}
\Pg_{(r_0,b_0)}\big( \Rg_\tau \geqslant a n^{2/3}\big) \geqslant c
\end{equation}
where $c>0$  does not depend on $b_0$ and $r_0 \geqslant 5 a n^{2/3}$. In view of the Markov property of the peeling process, this will entail that
\begin{equation}\label{boundE2}
\Pg\big(\mathcal{E}_2 \,\big|\, \mathcal{E}_1 \big) \geqslant c. 
\end{equation}

Set $m = \lfloor a n^{2/3} \rfloor$. We first notice that, on the event $A\defeq \{\tau>k,\, \Bg_k \geq 2m  \}$ we have $\Rg_k> 2m$. Thus, on this event, if $\Bg$ jumps at time $k+1$ then we have $\Rg_{k+1} = \Rg_{k} > 2m$ or $\tau > k+1$. This means that, on the event $A$, the event $\{\tau=k+1,\, \Rg_{k+1} \leq m\}$ happens if and only if $\Rg$ makes a downward jump such that $\Sg_{k+1} = \Rg_{k+1}+\Bg_{k+1} \leq m + \Bg_k$. Thus, we can write
\begin{align*}
&\Pg(\Rg_{\tau}\leq m,\,\tau=k+1\,|\,\Rg_k,\,\Bg_k,\,\tau>k)\ind{\tau>k,\,\Bg_k \geq 2m}\\
&=\Pg( \Rg_{k+1} \neq \Rg_{k},\, \Sg_{k+1} \leq \Bg_k + m)\ind{\tau>k,\,\Bg_k\geq 2m}\\
&=\frac{1}{2}\left(\sum_{y = 2}^{\Bg_k + m} p_{\Sg_k,y}\right)\ind{\tau>k,\,\Bg_k\geq 2m}
\end{align*}
where $(p_{x,y})_{x,y\geqslant 2}$ is the transition kernel of $\Sg$ under $\Pg$ given by formula~\eqref{lawmcpeeling}. Using similar arguments, we also find that
\begin{align*}
&\Pg(\Rg_{\tau} > m,\,\tau=k+1\,|\,\Rg_k,\,\Bg_k,\,\tau>k)\ind{\tau>k,\,\Bg_k \geq 2m}\\
&\geqslant \Pg( \Rg_{k+1} \neq \Rg_{k},\,2\Bg_k \geq \Sg_{k+1} > \Bg_k + m)\ind{\tau>k,\,\Bg_k \geq 2m}\\
&=\frac{1}{2}\left(\sum_{y=\Bg_k+m+1}^{2 \Bg_k}p_{\Sg_k,y}\right)\ind{\tau>k,\,\Bg_k \geq 2m}
\end{align*}
(the first line is not an equality since it is also possible for $\tau$ to occur while $B$ makes an  upward step). It is easy to check from \eqref{lawmcpeeling} that, for any $x>2$, the function $y \longmapsto p_{x,y}$  is non-decreasing for $y \in  \llbracket 2,x-1  \rrbracket$. As a consequence, if $\Rg_k > \Bg_k \geq 2m$, we have
\[
\sum_{y=2}^{\Bg_k+m }p_{\Sg_k,y}
\leqslant 3\sum_{y=\Bg_k + m +1}^{2\Bg_k} p_{\Sg_k,y}\]
since all the terms in the sum on the right hand side are bigger than those in the sum on the left hand side. and there are at most $3$ times as many terms on the left sum than on the right one. Combining this with the previous estimates and summing over $k$ yields
\[ \Pg_{(r_0,b_0)}\big(\Rg_\tau \leq m ,  \Bg_{\tau-1} \geq 2 m\big)
\leqslant 3\Pg_{(r_0,b_0)}\big(\Rg_\tau > m ,  \Bg_{\tau-1} \geq 2 m\big)\] and thus
\begin{equation}\label{event2eq1}
\Pg_{(r_0,b_0)}\big(\Rg_\tau> m  \big) \geqslant \frac{1}{4}\Pg_{(r_0,b_0)}\big( \Bg_{\tau-1} \geq 2 m\big)\geqslant \frac{1}{4}\Pg_{(r_0,b_0)}\big( \Bg_{\tau} > 2m\big)
\end{equation}
where, for the last inequality, we use the fact that $\Bg$ is right-continuous hence $\Bg_{\tau}\leq \Bg_{\tau-1}+1$. 
By definition, at time $\tau$, the blue vertices represent at least half of the total vertices on the boundary so that $\{ \Sg_\tau > 4 m \}\subset \{ \Bg_\tau > 2 m \}$. In view of~\eqref{event2eq1}, this gives the lower bound
\begin{align*}
\Pg_{(r_0,b_0)}\big(\Rg_\tau > a n^{2/3} \big)  = \Pg_{(r_0,b_0)}\big(\Rg_\tau > m \big) 
&\geqslant \frac{1}{4} \Pg_{s_0 = r_0 + b_0}\big( \Sg_k > 4 m \hbox{ for all $k$} \big)\\
&=  \frac{1}{4} \P_{s_0}\big( \gamma = +\infty \,\big|\, \Sd \geqslant 2\big),
\end{align*}
where $\gamma \defeq \min(k\geqslant 0,\, \Sd_k \leq 4 m)$. It remains to lower bound the right hand side of the previous equation uniformly in $s_0\geqslant  5 a n^{2/3} \geq 5 m$. Since $\Sd$ conditioned not to go below height $2$ is obtained by a h-transform and the function $h$ given by~\eqref{defhtransform} is increasing and regularly varying with index $1/2$, we get that
\begin{multline}\label{transienceMC}
\P_{s_0}\big( \gamma < +\infty \,\big|\, \Sd \geqslant 2\big) =\E_{s_0}\left[ \ind{\gamma < +\infty, \Sd_\gamma \geqslant 2} \frac{h(\Sd_\gamma - 1)}{h(s_0 - 1)}\right]\\
\leqslant \,\frac{h(4 m - 1)}{h(5 m - 1)} \,\underset{n\to\infty}{\longrightarrow}\, \sqrt{\frac{4}{5}} \,<\, 1.
\end{multline}
Putting everything together, we conclude that~\eqref{toproveE2} holds.
\medskip

\textbf{\textit{Event $\mathcal{E}_3$.}} On the event $\mathcal{E}_1 \cap \mathcal{E}_2$, we have, by construction, $\Bg_\tau \geq \Rg_\tau \geqslant a n ^{2/3}$. We now show that, 
\begin{equation}\label{toproveE3}
\Pg_{(r_0,b_0)}\big( \hbox{$\Rg$ hits $0$ before $\Bg$ but not before time $n$} \big) > \frac{1}{4},
\end{equation}
provided that $n$ is large enough and $b_0 \geqslant r_0\geqslant an^{2/3}$. Then the Markov property of the peeling process, applied at the stopping time $\tau$, will imply that 
\begin{equation*}
\Pg\big(\mathcal{E}_3 \,\big|\, \mathcal{E}_1,\, \mathcal{E}_2 \big) > \frac{1}{4}.
\end{equation*}
This, combined with the previous bounds~\eqref{boundE1} and~\eqref{boundE2}, will complete the proof of the proposition. 
\medskip

Notice first that, by symmetry and since $r_0 \leqslant b_0$, the probability that $\Rg$ hits zero before $\Bg$ is at least $1/2$. Thus, we have
\begin{multline}\label{lastPerimProp1}
\Pg_{(r_0,b_0)}\big( \hbox{$\Rg$ hits $0$ before $\Bg$ but not before time $n$} \big)  \\
\begin{aligned}
& \geqslant \Pg_{(r_0,b_0)}\big( \hbox{$\Rg$ hits $0$ before $\Bg$} \big) - \Pg_{(r_0,b_0)}\big( \hbox{$\Rg$ or $\Bg$ hit $0$ before time $n$} \big)\\
& \geqslant \Pg_{(r_0,b_0)}\big( \hbox{neither $\Rg$ nor $\Bg$ hit $0$ before time $n$} \big) - \frac{1}{2}.
\end{aligned}
\end{multline}
Recall that we can construct $(\Rg,\Bg)$  via~\eqref{defRBwithS}. In particular, as long as neither $\Rg$ nor $\Bg$ hit zero, the function $f$ given by~\eqref{deffreflection} acts as the identity function. Therefore, if we define 
\begin{equation*}
\Rd_k \defeq r_0 + \sum_{i=1}^{k} \eta_i(\Sd_{i} - \Sd_{i-1})\qquad\mbox{and}\qquad \Bd_k \defeq b_0 + \sum_{i=1}^{k} (1-\eta_i)(\Sd_{i} - \Sd_{i-1})
\end{equation*}
for any $k\geqslant 1$, then the process $(\Rg,\Bg)$ and the random walk $(\Rd,\Bd)$ coincide up to the first time either $\Rd$ or $\Bd$ enters $(-\infty,0]$. This yields
\begin{multline}\label{lastPerimProp2}
\Pg_{(r_0,b_0)}\big( \hbox{$\Rg$ or $\Bg$ hits $0$ before time $n$} \big)\\
\begin{aligned}
&= \P_{(r_0,b_0)}\big( \hbox{$\Rd$ or $\Bd$ enters $(-\infty,0]$ before time $n$} \,\big|\, \Sd\geqslant 2\big)\\
&\leqslant 2\, \P_{(r_0,b_0)}\big(\hbox{$\Rd$ enters $(-\infty,0]$ before time $n$} \,\big|\, \Sd\geqslant 2\big)\\
&\leqslant 2\, \P_{r_0}\big(\hbox{$\Rd$ enters $(-\infty,0]$ before time $n$} \big)\\
&\leqslant 2\, \P_{r_0}\big(\hbox{$\Sd$ enters $(-\infty,0]$ before time $n$} \big),\\ 
 \end{aligned}
\end{multline}
where we used that $b_0 \geqslant r_0$ for the first inequality, Lemma~\ref{lemStocDomS} for the second inequality and the fact that $\Rd$ under $\P$ is a time-delayed version of $\Sd$ (because $\eta$ and $\Sd$ are independent) for the last inequality. Finally, in view of item \ref{StableRW-CV} of Proposition \ref{propStableRW} in the appendix, we find that
\begin{equation*}
\P_{r_0}\big(\hbox{$\Sd$ enters $(-\infty,0]$ before time $n$} \big) \leqslant \P_{0}\!\bigg(\!\inf_{i\leqslant n}\Sd_i < - a n^{2/3}\! \bigg) \underset{n\to\infty}{\longrightarrow}\! \P\big(\inf_{t\leqslant 1}\mathcal{X}_t < -a\big),
\end{equation*}
where $\mathcal{X}$ is a spectrally negative strictly stable process of index $3/2$. In particular, the limit above goes to $0$ as $a$ increases to infinity. Therefore we can find $a$ such that, for all $n$ large enough, 
\[\P_{r_0}\big(\hbox{$\Sd$  enters $(-\infty,0]$ before time $n$} \big) \leqslant 1/8.\]
Combining this inequality with~\eqref{lastPerimProp1} and~\eqref{lastPerimProp2}, we conclude that~\eqref{toproveE3} holds. This ends the proof of proposition~\ref{PropPerimeter}.

\section{From the peeling process to random walks}
\label{SectionSimplification}

In order to complete the proof of the results stated in the introduction, we must establish Theorem~\ref{theoExploTime} and Proposition~\ref{propVol}. This amounts to proving the following three tail estimates concerning the peeling process:
\begin{align}
\label{asympTheta}\Pg(\thetag \geqslant n) &\asymp n^{-1/6},\\
\label{asympVred}\Pg\big(\Vgr_{\thetag-1} \geqslant n\big) &\lesssim n^{-1/8},\\
\label{asympV}\Pg\big(\Vg_{\thetag} \geqslant n\big)&\gtrsim n^{-1/8},
\end{align}
where we use the notation $f \lesssim g$ (respectively $f\gtrsim g$) if $f$ and $g$ are two non-negative functions and if there exists a constant $c >0$ such that $f(x) \leqslant c \,g(x)$ (respectively $f(x)\geqslant c \,g(x)$) for all $x$ large enough.
\medskip

The goal of this section is to show that~\eqref{asympTheta}-\eqref{asympV} concerning the peeling quantities $\Rg,\Bg,\Theta,\Vg$ are equivalent to  similar estimates, but concerning processes that are more amenable to analysis, constructed from a pair of independent continuous-time random walks with one of them conditioned to stay non-negative. 

\subsection{Definitions}

We introduce here all the processes and notations that will be used throughout this section and the next one. We work on the probability space $(\Omega,\mathcal{F},\P)$ defined in Section~\ref{SectionPerimeter}. Recall that, under $\P$, the process $\Sd$ is a random walk, starting from~$s_0$, with increments distributed according to $(p_k)_{k\in \Z}$ given by~\eqref{defrwpk}. Recall also that the sequence $(\eta_i)_{i\geqslant 1}$ is i.i.d.~with Bernoulli distribution of parameter $1/2$ and is independent of $\Sd$. 
\medskip

We construct a two-dimensional random walk $(\Rd,\Bd)$ by the relation
\begin{equation*}
\left\{
\begin{array}{rcl}
(\Rd_{0},\Bd_{0})  &\defeq& (r_0,b_0),\\
(\Rd_{n+1},\Bd_{n+1})  &\defeq&  (\Rd_n,\Bd_n) + (\Sd_{n+1} - \Sd_{n})(\eta_{n+1},1-\eta_{n+1}),\quad n\in \N.
\end{array}
\right.
\end{equation*}
Again, we enforce $s_0 = r_0 + b_0$ so that $\Sd_n = \Rd_n + \Bd_n$ for any $n\in \N$. The counterpart of the peeling time $\thetag$ for $(\Rd,\Bd)$ is defined by
\begin{equation*}
\Thetad \defeq \inf\,\big\{\, n \in \N : \Rd_n + \min(\Bdu_n , 0) \leqslant 0\,\big\},
\end{equation*}
where $\Bdu_n \defeq \inf_{k\leqslant n}\Bd_k$ is the running infimum of $\Bd$. Notice that the process $(\Rd,\Bd)$ differs from the peeling process $(\Rg,\Bg)$, defined by~\eqref{defRBwithS}, only by the absence of the reflection function $f$ in the recurrence equation. In fact, they are essentially equivalent since we can reconstruct the peeling process $(\Rg,\Bg)$ up to the peeling time of the red cluster\footnote{in fact, we can recover it at all time but the formula becomes more complicated and we shall not need it anyway.} by reflecting the walks $(\Rd,\Bd)$ against the running infimum of $B$. Indeed, prior to the peeling time, only $\Bd$ hits new negative records, at which time we subtract the undershoot and add it to $\Rd$. In other words, it holds that, deterministically, 
\begin{equation*}
\big( \Rd_n+\min(\Bdu_n,0), \Bd_n-\min(\Bdu_n,0) \big)_{n < \Thetad}
=
\big( \Rg_n,\Bg_n \big)_{n < \Thetag}
\end{equation*}
and, in particular, $\Thetad = \Thetag$. See figure~\ref{fig:transRW} for an illustration of this transformation. 
\medskip

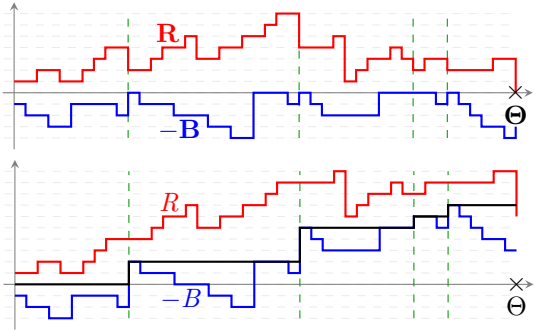
\begin{figure}[ht]
\centering
\begin{tabular}{c}
\begin{tikzpicture}[scale =0.15]
\foreach \k in {-4,-3,...,7} {\draw [color=gray!15,dashed] (-0.9,\k)-- (44,\k);}
\draw [>=stealth,->,color=gray!100] (0,-5) -- (0,8) ;
\draw [>=stealth,->,color=gray!100] (-1,0) -- (45.5,0) ;
\draw [dashed,color=green!60!black] (10,-4) -- (10,7) ;
\draw [dashed,color=green!60!black] (25,-4) -- (25,7) ;
\draw [dashed,color=green!60!black] (35,-4) -- (35,7) ;
\draw [dashed,color=green!60!black] (38,-4) -- (38,7) ;
\draw [thick,samples=1000,color=red] (0,1) -- (2,1) -- (2,2) -- (4,2) -- (4,1) -- (6,1) -- (6,2) -- (7,2) -- (7,3)-- (8,3) -- (8,4) -- (10,4) -- (10,2) -- (12,4-2) -- (12,5-2) -- (13,5-2) -- (13,6-2) -- (15,6-2) -- (15,7-2) -- (16,7-2) -- (16,5-2) -- (18,5-2) -- (18,6-2) -- (20,6-2) -- (20,7-2) -- (22,7-2) -- (22,8-2) -- (23,8-2) -- (23,9-2)--(25,9-2)--(25,9-5)
-- (28,9-5) -- (28,10-5) -- (29,10-5) -- (29,6-5) -- (30,6-5) --(30,7-5)--(31,7-5) -- (31,8-5) -- (33,8-5) -- (33,9-5) -- (34,9-5) -- (34,8-5) --(35,8-5)--(35,8-6)-- (36,8-6) -- (36,9-6)--(38,9-6)--(38,9-7)
-- (42,9-7) -- (42,10-7) -- (44,10-7)-- (44,6-6);
\draw [thick,samples=1000,color=blue] (0,-1) --  (1,-1) -- (1,-2) -- (3,-2) -- (3,-3) -- (5,-3) -- (5,-1) -- (9,-1) -- (9,-2) -- (10,-2) -- (10,0) 
-- (11,2-2) -- (11,1-2) -- (14,1-2) -- (14,0-2) -- (17,0-2) -- (17,-1-2) -- (19,-1-2) -- (19,-2-2) -- (21,-2-2) -- (21,2-2) -- (24,2-2) -- (24,1-2) -- (25,1-2) -- (25,5-5)
-- (26,5-5) -- (26,4-5) -- (27,4-5) -- (27,3-5) -- (32,3-5) -- (32,5-5) -- (35,5-5) -- (35,6-6)
-- (37,6-6) -- (37,5-6) -- (38,5-6) -- (38,7-7) -- (39,7-7) -- (39,6-7) -- (40,6-7) -- (40,5-7) -- (41,5-7) -- (41,4-7) -- (43,4-7) -- (43,3-7) --(44,3-7)--(44,3-6);
\draw (44,0) node {$\times$} ;
\draw (44,-0.3) node[below] {$\Thetag$} ;
\draw [thick,color=red] (13.5,7) node[below] {$\Rg$};
\draw [thick,color=blue] (14.5,-1.5) node[below] {$-\Bg$};
\end{tikzpicture}\\

\begin{tikzpicture}[scale =0.15]
\foreach \k in {-3,-2,...,10} {\draw [color=gray!15,dashed] (-0.9,\k)-- (44,\k);}
\draw [>=stealth,->,color=gray!100] (0,-4) -- (0,11) ;
\draw [>=stealth,->,color=gray!100] (-1,0) -- (45.5,0) ;
\draw [dashed,color=green!60!black] (10,-3) -- (10,10) ;
\draw [dashed,color=green!60!black] (25,-3) -- (25,10) ;
\draw [dashed,color=green!60!black] (35,-3) -- (35,10) ;
\draw [dashed,color=green!60!black] (38,-3) -- (38,10) ;
\draw [thick,samples=1000,color=red] (0,1) -- (2,1) -- (2,2) -- (4,2) -- (4,1) -- (6,1) -- (6,2) -- (7,2) -- (7,3)-- (8,3) -- (8,4) 
-- (12,4) -- (12,5) -- (13,5) -- (13,6) -- (15,6) -- (15,7) -- (16,7) -- (16,5) -- (18,5) -- (18,6) -- (20,6) -- (20,7) -- (22,7) -- (22,8) -- (23,8) -- (23,9)
-- (28,9) -- (28,10) -- (29,10) -- (29,6) -- (30,6) --(30,7)--(31,7) -- (31,8) -- (33,8) -- (33,9) -- (34,9) -- (34,8) 
-- (36,8) -- (36,9)
-- (42,9) -- (42,10) -- (44,10)-- (44,6);
\draw [thick,samples=1000,color=blue] (0,-1) --  (1,-1) -- (1,-2) -- (3,-2) -- (3,-3) -- (5,-3) -- (5,-1) -- (9,-1) -- (9,-2) -- (10,-2) -- (10,2) 
-- (11,2) -- (11,1) -- (14,1) -- (14,0) -- (17,0) -- (17,-1) -- (19,-1) -- (19,-2) -- (21,-2) -- (21,2) -- (24,2) -- (24,1) -- (25,1) -- (25,5)
-- (26,5) -- (26,4) -- (27,4) -- (27,3) -- (32,3) -- (32,5) -- (35,5) -- (35,6)
-- (37,6) -- (37,5) -- (38,5) -- (38,7) -- (39,7) -- (39,6) -- (40,6) -- (40,5) -- (41,5) -- (41,4) -- (43,4) -- (43,3) --(44,3);
\draw [samples=1000, thick] (0,0) --  (10,0) -- (10,2) --  (25,2) -- (25,5) -- (35,5) -- (35,6) --  (38,6) -- (38,7) -- (44,7) ;
\draw (44,0) node {$\times$} ;
\draw (44,-0.3) node[below] {$\Thetad$} ;
\draw [thick,color=red] (13.5,9) node[below] {$\emph{\Rd}$};
\draw [thick,color=blue] (14.5,0.5) node[below] {$-\emph{\Bd}$};
\end{tikzpicture}
\end{tabular}
\caption{The top figure represents a realization of $\Rg$ (in red) and $-\Bg$ (in blue) until time $\Thetag$. The down figure represents the random walks $\Rd$ and $\Bd$ derived from $(\Rg,\Bg)$ until time $\Thetad=\Thetag$. The random walk $\smash{\big(-\min(\Bdu_n,0)\big)_{n<\Thetad}}$ is the one in black. The green dashed lines are the times where a blue step swallows red sites (\ie at these times, both $\Bg$ and $\Rg$ jump : $\Bg$ returns to $0$ and $\Rg$ makes a negative jump).}
\label{fig:transRW}
\end{figure}

We also need to construct the volume process $(\Vg,\Vgr,\Vgb)$, defined in section~\ref{SectionEncoding}, on our probability space $(\Omega,\mathcal{F},\P)$. Let
\[\Big(\z{j}{i},\, i\geqslant 1,\, j\in \{-1\}\cup\{1,2,\ldots\}\Big)\]
be a family of independent random variables, independent of everything else,  such that, for any fixed $j\geqslant -1$, $\smash{\big(\z{j}{i}\big)_{i\geqslant 1}}$ is a sequence of i.i.d.~random variables. Besides $\z{j}{1}$ has the same distribution as the number of inner vertices inside a rooted free Boltzmann triangulation with a boundary of size $j+1$, (with the convention that $\z{-1}{i} = 1$). We define the random walk $(\Vd,\Vdr, \Vdb)$ by, for any $n\geqslant 1$,
\begin{align*}
\Vd_{n} &\defeq \sum_{k=1}^{n} \z{\Sd_{k-1} - \Sd_{k}}{k},\\
\Vdr_{n} &\defeq \sum_{k=1}^{n} \eta_k\z{\Sd_{k-1} - \Sd_{k}}{k} = \sum_{k=1}^{n} \eta_k\z{\Rd_{k-1} - \Rd_{k}}{k},\\
\Vdb_{n} &\defeq \sum_{k=1}^{n} (1-\eta_k)\z{\Sd_{k-1} - \Sd_{k}}{k} = \sum_{k=1}^{n} (1-\eta_k)\z{\Bd_{k-1} - \Bd_{k}}{k}.
\end{align*}
By construction, under $\Pg = \P(\,\cdot\,|\,\Sd \geqslant 2)$, the process $(\Vd,\Vdr,\Vdb)$ has the same law as the volume of the peeling process defined in section~\ref{SectionEncoding}. Rigorously speaking, since we start the processes $\Vd$,$\Vdr$ and $\Vdb$ from $0$, the volume defined in this way does not take into account the vertices located on the initial boundary. However, this only changes the value of the volume by a finite offset (at most $2$ when the peeling starts from the root edge). Therefore it does not change the tail asymptotic and we can safely ignore this technicality. \medskip

As a consequence, we can now work with $\Rd$, $\Bd$, $\Thetad$, $\Vd$, $\Vdr$, $\Vdb$ in place of $\Rg$, $\Bg$, $\Thetag$, $\Vg$, $\Vgr$, $\Vgb$. Let us stress again that, under $\P$, these processes are just classical random walks. 
\medskip

We make one last change to our problem by embedding our processes into continuous time using a classical ``Poissonization'' procedure. Let $\NP=(\NP_t)_{t\geqslant 0}$ denote a Poisson process with unit intensity, independent of everything else. We define the continuous-time processes $\Rc,\Bc,\Vc$\ldots from their discrete-time counterparts via the change of time induced by $\NP$, \ie
\[
\Bc_t \defeq \Bd_{N_t}, \quad \Rc_t \defeq \Bd_{N_t}, \quad \Vc_t \defeq \Vd_{N_t}, \quad\ldots
\]
More generally, we use an italic/calligraphic font for quantities related to continuous-time and normal font for discrete-time (the bold font indicates quantities directly related to the UIPT). The peeling time for $(\Rc,\Bc)$ is defined by
\begin{equation*}
\Thetac \defeq \inf\,\big\{\,t>0 : \Rc_t+\min(\Bcu_t,0) \leqslant 0\,\big\}.
\end{equation*}
or, equivalently, $N_\Thetac = \Thetad$. The main advantage of moving to continuous-time is that it decorrelates the red and blue processes. Indeed, under $\P$, the processes $\Rc$ and $\Bc$ are now independent and have the same law (which is also the law of $\Sd$, timed-changed by a Poisson process of intensity $1/2$). However, $\Rc$ and $\Bc$ are still not independent under the conditional measure $\Pg = \P(\,\cdot\,|\, \Sd\geqslant 2)$. 
\medskip

The aim of this section is to prove that we can recover some independence by moving the conditioning from $\Sd$ to $\Rd$. More precisely, just as in Section~\ref{SectionPerimeter}, we can define the measure $\smash{\P\big(\,\cdot \,\big|\, \Rd > 0\big)}$ (resp.~$\smash{\P\big(\,\cdot \,\big|\, \Rc > 0\big)}$) via a Doob's h-transform. Since the random walk $\Rd$ (resp.~$\Rc$) is oscillating, this conditional law corresponds to the weak limit for the sequence of measures obtained by conditioning $\Rd$ (resp.~$\Rc$) to stay positive up to arbitrarily large time. On the other hand, up to a time change, the processes $\Rc,\Rd$ and $\Sd$ all have the same law. Therefore, they admit the same harmonic functions. This implies that the Doob's transforms for $\Rc$ and $\Rd$ are equal\footnote{Another way to see this is by recalling that the conditional law can also be constructed by taking the limit for the processes required to hits arbitrarily high heights before entering the negative half line \emph{c.f.}~\cite{BD94}. Since these events depend only on the trace of the processes, they are the same for $\Rd$ and $\Rc$.} thus
\begin{equation}\label{egalcondi}
\P\big(\,\cdot \,\big|\, \Rc > 0\big) = \P\big(\,\cdot \,\big|\, \Rd > 0\big).
\end{equation}
Moreover, this conditional measure corresponds to the Doob's h-transform of $\P$ with the same function $h$ given by~\eqref{defhtransform} as before. We point out that~\eqref{egalcondi} states that the two operations ``conditioning'' and ``moving to continuous time'' commutes: we can perform them is any order. In particular, it follows that, even though~$\Rc$ depends on the Poisson process $\NP$, the law of $\NP$ remains unchanged under $\P\big(\,\cdot \,\big|\, \Rc > 0\big)$. \medskip

\subsection{Reduction of the problem}

We now show that we can perform the following three operations without changing the asymptotics of~\eqref{asympTheta}-\eqref{asympV}. 
\begin{enumerate}
\item[(i)] We can replace the probability $\Pg = \P(\,\cdot\,|\,\Sd\geqslant 2)$ by $\P(\,\cdot\,|\,\Rd>0)$.
\item[(ii)] We can move the origin from $(r_0,b_0)$ to $(1,0)$. This position does not make sense for the initial peeling process $(\Rg,\Bg)$ under $\Pg$ but it is well defined for $(\Rd,\Bd)$ under $\P(\,\cdot\,|\,\Rd>0)$.  
\item[(iii)] We can also move from discrete to continuous time.
\end{enumerate}
The precise statement of our result is the following.
 
\begin{prop}\label{ReformulationProp} 
Suppose we know that 
\begin{equation}\label{hypregvar}
\P_{(1,0)}\big(\thetac>n\,|\,\Rc>0\big) \asymp n^{-a}
\end{equation}
for some exponent $a>0$. Consider the peeling process starting from an initial boundary of size $s_0 = r_0 + b_0 \geqslant 2$ such that $r_0\geqslant 1$. Then, it holds that
\begin{align*}
\Pg_{(r_0,b_0)}\big(\Thetag >n\big) &\asymp \P_{(1,0)}\big(\thetac>n\,|\,\Rc>0\big),\\
\Pg_{(r_0,b_0)}\big(\Vg_\Thetag > v\big) &\lesssim \P_{(1,0)}\big(\Vc_{\thetac}>v\,\big|\,\Rc>0\big),\\
\Pg_{(r_0,b_0)}\big(\Vgr_{\Thetag-1} >v\big) &\gtrsim \P_{(1,0)}\big(\Vcr_{\thetac^-} > v\,\big|\,\Rc>0\big),
\end{align*}
with the convention that $\Vcr_{\thetac^-}$ denotes the left limit at $\thetac$ (\ie the value of $\Vcr$ at the previous step).
\end{prop}

\begin{proof}
We will split the proof in 3 parts.
\medskip

\textbf{\textit{Change of conditioning: lower bounds.}} We show that
\begin{align}
\label{minchange1}\Pg_{(r_0,b_0)}\big(\Thetag > n \big) & \gtrsim \P_{(1,0)}\big(\Thetad > n \,\big|\, \Rd>0 \big),\\
\label{minchange2}\Pg_{(r_0,b_0)}\big(\Vgr_{\Thetag-1} > v\big) & \gtrsim \P_{(1,0)}\big(\Vdr_{\Thetad-1} > v \,\big|\, \Rd>0\big).
\end{align}
We only need to consider the case $(r_0,b_0) = (2,0)$. Indeed, starting from any $(r_0,b_0)$ such that $r_0\geqslant 1$, we can always reach the boundary state $(2,0)$ with probability larger than some $\varepsilon = \varepsilon(r_0,b_0) >0$ in $2$ steps while still peeling the red cluster. Then, applying the Markov property at that instant, it follows that
\begin{equation*}
\Pg_{(r_0,b_0)}\big(\Thetag > n \big) \geqslant  \varepsilon \Pg_{(2,0)}\big(\Thetag > n \big) \, \hbox{ and }\,
\Pg_{(r_0,b_0)}\big(\Vgr_{\Thetag-1} > v\big) \geqslant \varepsilon \Pg_{2,0}\big(\Vgr_{\Thetag-1} > v\big).
\end{equation*}
Let $A$ be either the whole probability space $\Omega$ or the event  $\{ \Vdr_{n-1 }\leqslant v < \Vdr_n \}$ for $n\geqslant 1$. We write
\begin{align*}
\Pg_{(2,0)}\big(A, \Thetag > n \big) & = \E_{(2,0)}\left[\indic{A}\ind{\Thetad > n}\ind{\Sdu_n \geqslant 2}\frac{h(\Sd_n - 1)}{h(1)}\right]\\
& \geqslant \E_{(1,0)}\left[\indic{A}\ind{\Thetad > n}\ind{\Sdu_n \geqslant 1}\frac{h(\Sd_n)}{h(1)}\right],
\end{align*}
where the inequality above was obtained by remarking that, if a trajectory of $\Rd$ satisfies $\{ \Thetad > n\}$ and $A$, then it is also the case for the same trajectory shifted upward by $1$. On the other hand, by definition of $\Thetad$, we have 
$\{ \Thetad > n\} \subset \{\Sdu_n \geqslant 1\}$. Therefore, we get that
\begin{align}
\nonumber\Pg_{(2,0)}\big(A,&\, \Thetag > n \big)  \geqslant \E_{(1,0)}\left[\indic{A}\ind{\Thetad > n}\frac{h(\Sd_n)}{h(1)}\right]\\
\nonumber& \geqslant \E_{(1,0)}\left[\indic{A}\ind{\Thetad > n}\ind{\Rd_n > 0}\ind{\Bd_n \geqslant 0}\frac{h(\Rd_n + \Bd_n)}{h(1)}\right]\\
\nonumber&  \geqslant \E_{(1,0)}\left[\indic{A}\ind{\Thetad > n}\ind{\Rd_n > 0}\ind{\Bd_n \geqslant 0}\frac{h(\Rd_n)}{h(1)}\right]\\
\label{changemin1}&  =  \E_{(1,0)}\left[\indic{A}\ind{\Rd_n > 0} \P_{(1,0)}\big(\Thetad > n, \Bd_n \geqslant 0\, \big|\, \eta,\Rd,\Vdr \big)\frac{h(\Rd_n)}{h(1)}\right],
\end{align}
where we used that $h$ is non-decreasing for the third inequality. Conditionally on $\smash{(\eta,\Rd,\Vdr)}$, the increments of $\Sd$ are deterministic along the red steps but are still i.i.d.~along the blue steps. Moreover, $\eta,\Rd,\Vdr$ being fixed, the events $\{\Bd_n \geqslant 0\}$ and
\[\{\Thetad>n\}=\{\forall k\leqslant n, \,R_k+\min(\Bu_k,0)>0\}\]
are increasing with respect to the increment of $\Sd$ along the blue steps. Thus Harris inequality (\emph{c.f.}~Proposition~\ref{HarrisProp} in the appendix) implies that they are positively correlated under the conditional law $\P_{{(r_0,b_0)}}(\,\cdot\,|\,\eta,\,\Rd,\,\Vdr\big)$, \ie
\begin{multline}\label{changemin2}
\P_{(1,0)}\big(\Thetad > n, \Bd_n \geqslant 0\, \big|\, \eta,\Rd,\Vdr \big)  \\
\geqslant  \P_{(1,0)}\big(\Thetad > n \,\big|\, \eta,\Rd,\Vdr \big) \P_{(1,0)}\big(\Bd_n \geqslant 0 \,\big|\, \eta,\Rd,\Vdr \big)
\end{multline}
Conditionally on $(\eta,\Rd,\Vdr)$, the process $\Bd$, considered only at the instants $i$ where $\eta_i = 0$ (when it jumps), is a random walk with the same law as $\Sd$. Therefore we have
\begin{equation}
\label{changemin3}  \P_{(1,0)}\big(\Bd_n \geqslant 0\, \big|\, \eta,\Rd,\Vdr \big) \geqslant \min_{k}\, \P_{0}\big(\Sd_k \geqslant 0 \big) = c > 0
\end{equation}
since $\Sd$ is centered and lies in the domain of attraction of a spectrally negative stable law of index $3/2$. Combining~\eqref{changemin1},\eqref{changemin2} and~\eqref{changemin3}, we deduce that
\begin{align*}
\Pg_{(2,0)}\big(A, \Thetag > n \big) & \geqslant c\,\E_{(1,0)}\left[\indic{A}\ind{\Rd_n > 0} \ind{\Thetad > n}\frac{h(\Rd_n)}{h(1)}\right]\\
&= c \,\P_{(1,0)}\big(A,\, \Thetad > n \, | \, \Rd > 0\big).
\end{align*}
Taking $A = \Omega$, we conclude that~\eqref{minchange1} holds. For the lower bound on the volume generated by the red steps, we choose $A = \{ \Vdr_{n-1 }\leqslant v < \Vdr_n \}$ and then, by summing over $n$, we conclude that 
\begin{align*}
\Pg_{(2,0)}\big(\Vgr_{\Thetag-1} > v\big) &= \sum_{n\geqslant 1} \Pg_{(2,0)}\big( \Vgr_{n-1 }\leqslant v < \Vgr_n, \, \Thetag > n\big)\\
& \geqslant c\,  \sum_{n\geqslant 1}  \P_{(1,0)}\big( \Vdr_{n-1 }\leqslant v < \Vdr_n, \, \Thetad > n \, | \, \Rd > 0 \big)\\
& = c\, \P_{(1,0)}\big( \Vdr_{\Thetad-1} > v \, | \, \Rd > 0 \big),
\end{align*}
which completes the proof of~\eqref{minchange2}.
\bigskip

\textbf{\textit{Change of conditioning: upper bounds.}} We prove that
\begin{align}
\label{maxchange1}\Pg_{{(r_0,b_0)}}\big(\Thetag > n \big) & \lesssim \P_{(1,0)}\big(\Thetad > n \,\big|\, \Rd>0 \big),\\
\label{maxchange2}\Pg_{{(r_0,b_0)}}\big(\Vg_{\Thetag} > v\big) & \lesssim \P_{(1,0)}\big(\Vd_{\Thetad} > v \,\big|\, \Rd>0\big).
\end{align}
We use a similar strategy as in the lower bounds. Let now A be either $\Omega$ or the event $A = \{ \Vd_{n}\leqslant v < \Vd_{n+1} \}$. Notice that $A$ depends now on the randomness up to time $n+1$ so that we write
\begin{multline*}
\Pg_{(r_0,b_0)}\big(A,\,\Thetag > n \big) =  \E_{(r_0,b_0)}\left[\indic{A}\ind{\Thetad > n}\ind{\Sdu_{n+1} \geqslant 2}\frac{h(\Sd_{n+1} - 1)}{h(s_0-1)}\right]\\
\leqslant \E_{(r_0,b_0)}\left[\indic{A}\ind{\Thetad > n}\frac{h(\Rd_{n+1})}{h(s_0-1)}\right] + \E_{(r_0,b_0)}\left[\indic{A}\ind{\Thetad > n}\frac{h(\Bd_{n+1})}{h(s_0-1)}\right],
\end{multline*}
where we used the fact that the function $h$, defined by~\eqref{defhtransform}, is non-decreasing and satisfies the sud-additive inequality $h(x+y) \leqslant h(x) + h(y)$ for all $(x,y)\in \Z^2$. Recalling Lemma~\ref{OnionLemma} and using again that $h$ is non-decreasing, we get
\begin{align*}
\E_{(r_0,b_0)}\left[\indic{A}\ind{\Thetad > n}\frac{h(\Bd_{n+1})}{h(s_0-1)}\right] & = \E_{(r_0,b_0)}\left[\indic{A}\,\E_{(r_0,b_0)}\Big[\ind{\Thetad > n}\frac{h(\Bd_{n+1})}{h(s_0-1)} \,\Big|\, \Sd,\, \big(\z{j}{i}\big)\Big]\right]\\
& \leqslant \E_{(r_0,b_0)}\left[\indic{A}\,\E_{(r_0,b_0)}\Big[\ind{\Thetad > n}\frac{h(\Rd_{n+1})}{h(s_0-1)} \,\Big|\, \Sd,\,\big(\z{j}{i}\big)\Big]\right]\\
& = \E_{(r_0,b_0)}\left[\indic{A}\ind{\Thetad > n}\frac{h(\Rd_{n+1})}{h(s_0-1)}\right].
\end{align*}
Therefore, we obtain
\begin{align*}
\Pg_{(r_0,b_0)}\big(A, \Thetag > n \big) &\leqslant 2\, \E_{(r_0,b_0)}\left[\indic{A}\ind{\Thetad > n}\frac{h(\Rd_{n+1})}{h(s_0-1)}\right]\\
& =  2\, \E_{(r_0,b_0)}\left[\indic{A}\ind{\Thetad > n}\ind{\Rdu_{n+1} > 0}\frac{h(\Rd_{n+1})}{h(s_0-1)}\right]\\
& = \frac{2 h(r_0)}{h(s_0-1)}\P_{(1,0)}\big(A,\, \Thetad > n \,\big|\, \Rd>0 \big),
\end{align*}
where we used that $h(\Rd_{n+1}) = h(\Rd_{n+1})\ind{\Rd_{n+1} > 0}$ and $\ind{\Thetad > n} = \ind{\Thetad > n}\ind{\Rdu_n > 0}$ for the middle equality. Choosing $A = \Omega$, we obtain
\begin{equation}\label{maxchange1mid}
\Pg_{(r_0,b_0)}\big(\Thetag > n \big) \leqslant c\, \P_{(r_0,b_0)}\big(\Thetad > n \,\big|\, \Rd>0 \big).
\end{equation}
Just as for the lower bounds, taking $A = \{ \Vd_{n}\leqslant v < \Vd_{n+1} \}$ and summing over all $n\in\N$ yields
\begin{equation}\label{maxchange2mid}
\Pg_{(r_0,b_0)}\big(\Vg_{\Thetag} > v \big) \leqslant c\, \P_{(r_0,b_0)}\big(\Vg_{\Thetag} > v \,\big|\, \Rd>0 \big).
\end{equation}
It remains to prove that changing the initial condition to $(1,0)$ changes only the probabilities on the right hand side of~\eqref{maxchange1mid} and~\eqref{maxchange2mid} by a constant multiplicative factor. The argument used to change of origin for the lower bounds still applies: for any $(r_0,b_0)$ such that $r_0\geqslant 1$, we can find $m\in \N$ such that \[\P_{(0,1)}\big(\Rd_m = r_0,\, \Bd_m= b_0,\, \Thetad > m\big) = \varepsilon > 0.\]
Using the Markov property at time $m$, we deduce that
\[
\P_{(1,0)}\big(\Thetad > n \,\big|\, \Rd>0 \big) \geqslant \varepsilon\, \P_{(r_0,b_0)}\big(\Thetad > n \,\big|\, \Rd>0 \big)
\]
and 
\[
\P_{(1,0)}\big(\Vg_{\Thetad} > v \,\big|\, \Rd>0 \big) \geqslant \varepsilon\, \P_{(r_0,b_0)}\big(\Vg_{\Thetad} > v \,\big|\, \Rd>0 \big),
\]
which completes the proof of the upper bounds~\eqref{maxchange1} and~\eqref{maxchange2}.
\bigskip

\textbf{\textit{Passage in continuous time.}} By definition, we have $\Vd_\thetad=\Vc_\thetac$ and $\Vdr_{\thetad - 1}=\Vcr_{\thetac^-}$. This directly shows that 
\begin{align*}
\Pg_{(r_0,b_0)}\big(\Vg_\Thetag > v\big) &\lesssim \P_{(1,0)}\big(\Vd_{\thetad}>v\,\big|\,\Rd>0\big)= \P_{(1,0)}\big(\Vc_{\thetac}>v\,\big|\,\Rc>0\big),\\
\Pg_{(r_0,b_0)}\big(\Vgr_{\Thetag-1} >v\big) &\gtrsim \P_{(1,0)}\big(\Vdr_{\thetad^-} > v\,\big|\,\Rd>0\big) = \P_{(1,0)}\big(\Vcr_{\thetac^-} > v\,\big|\,\Rc>0\big).
\end{align*}
It remains to deal with the peeling time.  Recall that $\Thetad = N_{\thetac}$. Moreover, under $\P(\,\cdot\,|\, \Rc>0) = \P(\,\cdot\,|\, \Rd>0)$, the process $(N_t)_{t\geq 0}$ is a Poisson process with unit intensity which is independent or $\thetad$. Therefore, we have
\begin{equation*}
\P_{(1,0)}\big(\thetac>n\,|\,\Rc>0\big) \geqslant  \P_{(1,0)}\big(\thetad>n,\, N_n \leqslant n  \,|\,\Rd>0\big)  \geqslant \frac{1}{2} \P_{(1,0)}\big(\thetad>n\,|\,\Rd>0\big)
\end{equation*}
which takes care of the lower bound. For the upper bound, we write that
\begin{equation*}
\P_{(1,0)}\big(\thetac>n\,|\,\Rc>0\big) \leqslant \P_{(1,0)}\big(\thetad>\frac{n}{2}\,|\,\Rd>0\big) + \P\big(N_{n} < \frac{n}{2}\,|\,\Rd>0\big)
\end{equation*}
The second term decreases exponentially fast in $n$. Hence, according to assumption \eqref{hypregvar}, it is negligible compared to $\P_{(1,0)}\big(\thetad>\frac{n}{2}\,|\,\Rd>0\big) \asymp \P_{(1,0)}\big(\thetad> n\,|\,\Rd>0\big)$.
\end{proof}

\section{Proof of the main results.}
\label{SectionResolution}

\newcommand{\Hm}{H} 
\newcommand{\T}{T} 
\renewcommand{\U}{U} 


This section is devoted to establish the following three estimates which, together with Proposition~\ref{PropPerim} and Proposition~\ref{ReformulationProp}, complete the proof of Theorem~\ref{mainTheo} and Theorem~\ref{theoExploTime}. 

\begin{prop}\label{lastprop} We have
\begin{align*}
\P_{(1,0)}\big(\thetac > t\,\big|\,\Rc>0\big)&\asymp t^{-1/6},\\
\P_{(1,0)}\big(\Vcr_{\thetac^-} > v\,\big|\,\Rc>0\big)&\gtrsim v^{-1/8},\\
\P_{(1,0)}\big(\Vc_{\thetac}>v\,\big|\,\Rc>0\big)&\lesssim v^{-1/8}.
\end{align*}
\end{prop}

\subsection{Coding \texorpdfstring{$\{\theta>n\}$, $\{\Vcr_{\theta^-} >v\}$, $\{\Vc_{\theta} >v\}$}{} via stopping times}

Recall that, under $\P_{(1,0)}$, the processes $\Rc$ and $\Bc$ are two independent continuous-time random walks, starting respectively from $1$ and $0$ and jumping with rate $1/2$ with step distribution prescribed by $(p_k)_{k\in \Z}$ given by~\eqref{defrwpk}. These walks are recurrent and right-continuous, \ie they can only make upward jumps of unit size. We can define the \emph{strict descending ladder times} of $\Bc$ by
\begin{align*}
\T_0 &\defeq 0,\\
\T_{n+1} &\defeq \inf \,\{\,t>\T_n:\Bc_t < \Bc_{\T_n}\,\}\quad\hbox{for $n\in\N$.}
\end{align*}
Under $\P_{(1,0)}$, these stopping times are all finite almost surely and the sequence $(\T_{n+1} - \T_n)_{n\in \N}$ is i.i.d.~and has the law of $\T_1$, the first hitting time of $(-\infty,-1]$ by~$\Bc$. We also define the associated \emph{strict descending ladder heights} by
\begin{equation*}
\Hm_n \defeq - \Bc_{\T_n}\quad\hbox{for $n\in\N$.}
\end{equation*}
On the other hand, under the probability measure $\P_{(1,0)}(\,\cdot\,|\,\Rc>0)$, the law of $\Bc$ is unchanged whereas $\Rc$ becomes a continuous-time Markov chain, jumping with rate $1/2$, and with transition kernel given by~\eqref{lawmcpeeling}. In particular, it is a transient process that diverges to infinity. This fact may be checked directly adapting the argument of~\eqref{transienceMC} or can also be seen using Tanaka's construction of a random walk conditioned to stay positive from a space/time reversal of its excursions between successive strict \emph{ascending} ladder times, see~\cite{Tan89} for details. Since $\Rc$ only make upward jumps of size $1$, it must therefore hit each height one last time. Thus, the last passage times at heights $\Hm_n$, $n\geqslant 1$, are well defined and almost surely finite:
\begin{equation*}
\U_n \defeq \sup\,\{\, t\geqslant 0 : \Rc_{t} = \Hm_n\,\} \quad\hbox{for $n\geqslant 1$},
\end{equation*} 
with the convention that $\U_0 \defeq 0$. Notice that, since $\Rc$ is a càdlàg process, we have $\Rc_{U_n} = \Hm_n + 1$ for all $n\geqslant 1$. Furthermore, the process $\Rc$ conditioned to stay positive and observed after the time of its last passage of a given height $h$ has the same law as the process conditioned to stay above $h$. Thus, under $\P_{(1,0)}(\,\cdot\,|\,\Rc>0)$ the sequence $(\U_{n+1} - \U_n)_{n\in \N}$ is i.i.d.~and has the same law as $U_1$.
\medskip

The reason we are interested in the sequences $(T_n)_{n\in \N}$ and $(U_n)_{n\in \N}$ is because they can be used to control the peeling time $\thetac$. See figure~\ref{fig:T,U} for an example.

\begin{prop}
\label{CodingStoppingTimes}
The following inclusions hold:
\begin{align*}
\big\{\forall i\leqslant n,\,  \T_i > \U_i\big\} \, &\subset \, \big\{\thetac > \T_n\big\},\\
\big\{ \T_n< \U_n\big\} \, &\subset\, \big\{\thetac < \U_n\big\},\\ 
\big\{\forall i\leqslant n,\,  \T_i > \U_i\big\} \, &\subset  \, \big\{\Vcr_{\thetac^-}\geqslant \Vcr_{\U_n}\big\},\\
\big\{ \T_n < \U_n \big\}\, &\subset\, \big\{\Vc_\thetac\leqslant \Vc_{\U_n}\big\}.
\end{align*}
\end{prop}

\begin{proof} By definition, $\thetac$ corresponds to the first time when $\Rc$ visit a site previously visited by $-\Bc$. Suppose that $\thetac$ occurs between times $\T_i$ and $\T_{i+1}$, for some $i\in \N$. This means that $\Rc$ enters $(-\infty, \Hm_{i-1}\rrbracket$ during this time interval so, in particular, $\U_{i-1} > \T_{i-1}$. This proves the first inclusion. The proof of the second inclusion is similar. The last two inclusions follow directly from the first two combined with the fact that $\Vcr$ and $\Vc$ are non-decreasing processes. 
\end{proof}

\begin{figure}[ht]
\centering
\begin{tikzpicture}[scale =0.3]
\foreach \k in {-2,-1,...,10} {\draw [color=gray!15,dashed] (-1,\k)-- (27,\k);}
\draw [>=stealth,->,gray!100] (0,-3) -- (0,11) ;
\draw [>=stealth,->,gray!100] (-1,0) -- (26,0) ;
\draw [dashed,color=blue!70] (5.8,-2) -- (5.8,-1) ;
\draw [color=blue!70] (5.8,-3) node {$T_1$} ;
\draw [dashed,color=blue!70] (14.4,-2) -- (14.4,1) ;
\draw [color=blue!70] (14.4,-3) node {$T_2$} ;
\draw [dashed,color=blue!70] (19,-2) -- (19,5) ;
\draw [color=blue!70] (19,-3) node {$T_3$} ;
\draw [dashed,color=blue!70] (20.7,-2) -- (20.7,5) ;
\draw [color=blue!70] (20.7,-3) node {$T_4$} ;
\draw [dashed,color=blue!70] (26,-2) -- (26,8) ;
\draw [color=blue!70] (26,-3) node {$T_5$} ;
\draw [dashed,color=red!70] (4.6,3) -- (4.6,10) ;
\draw [dashed,color=red!70] (4.6,3) -- (5.8,3) ;
\draw [color=red!70] (4.6,11) node {$U_1$};
\draw [dashed,color=red!70] (10.1,5) -- (10.1,10) ;
\draw [dashed,color=red!70] (10.1,5) -- (14.4,5) ;
\draw [color=red!70] (10.1,11) node {$U_2$};
\draw [dashed,color=red!70] (17,6) -- (17,10) ;
\draw [dashed,color=red!70] (17,6) -- (19,6) ;
\draw [color=red!70] (17,11) node {$U_3$};
\draw [dashed,color=red!70] (24.5,9) -- (24.5,10) ;
\draw [color=red!70] (24.5,11) node {$U_4$};
\draw [thick,samples=1000,color=red] (0,1) -- (1.5,1) -- (1.5,2) -- (2.9,2) -- (2.9,1) -- (3.4,1) -- (3.4,2) -- (4,2) -- (4,3)-- (4.6,3) -- (4.6,4) -- (6.8,4) -- (6.8,5) -- (7.4,5) -- (7.4,6) -- (9,6) -- (9,7) -- (9.4,7) -- (9.4,5) -- (10.1,5) -- (10.1,6) -- (11.6,6) -- (11.6,7) -- (11.9,7) -- (11.9,8) -- (13.1,8) -- (13.1,9) -- (15.8,9) -- (15.8,10) -- (16.4,10) -- (16.4,6) -- (17,6) --(17,7)--(17.4,7) -- (17.4,8) -- (17.9,8) -- (17.9,9) -- (18.1,9) -- (18.1,8) -- (19.3,8) -- (19.3,9)-- (22.6,9) -- (22.6,10) -- (23.7,10)-- (23.7,7) -- (24.1,7) -- (24.1,8) -- (24.5,8) -- (24.5,9) -- (25.4,9)--(25.4,10) -- (26,10);
\draw [thick,samples=1000,color=blue] (0,0) --  (1.3,0) -- (1.3,-1) -- (1.9,-1) -- (1.9,-2) -- (3.2,-2) -- (3.2,0) -- (5,0) -- (5,-1) -- (5.8,-1) -- (5.8,3) -- (6.7,3) -- (6.7,2) -- (8.5,2) -- (8.5,1) -- (9.6,1) -- (9.6,0) -- (10.3,0) -- (10.3,-1) -- (12,-1) -- (12,2) -- (13.5,2) -- (13.5,1) -- (14.4,1) -- (14.4,5) -- (15.2,5) -- (15.2,4) -- (15.5,4) -- (15.5,3) -- (17.7,3) -- (17.7,5) -- (19,5) -- (19,6) -- (19.7,6) -- (19.7,5) -- (20.7,5) -- (20.7,8) -- (21.5,8) -- (21.5,7) -- (21.8,7) -- (21.8,8) -- (22.4,8) -- (22.4,7) -- (23,7) -- (23,6) -- (24.9,6) -- (24.9,8) -- (26,8) -- (26,9) ;
\draw [samples=1000, thick] (0,0) --  (5.8,0) -- (5.8,3) --  (14.4,3) -- (14.4,5) -- (19,5) -- (19,6) --  (20.7,6) -- (20.7,8) -- (26,8) -- (26,9) ;
\draw [thick,color=red] (14,9) node[below] {$\Rc$};
\draw [thick,color=blue] (16.5,3) node[below] {$-\Bc$};
\draw (23.7,0) node {$\times$} ;
\draw (23.7,-0.3) node[below] {$\theta$} ;
\draw [dashed] (23.7,0) -- (23.7,7) ;
\end{tikzpicture}
\caption{In red, a realization of the walk $R$ conditioned to stay positive. In blue, a realisation of the walk $-B$ and, in black, the associated process $-\underline{B}$.}
\label{fig:T,U}
\end{figure}
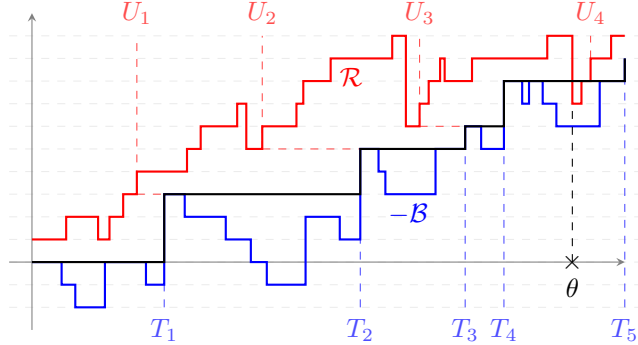

In the previous section, we showed that we can modify the problem by switching from the probability measure $\P(\,\cdot\,|\,\Rc + \Bc > 1)$ to $\P(\,\cdot\,|\,\Rc>0)$. While doing so, one might fear that we are losing the natural symmetry of the peeling procedure where blue and red vertices play the same role (since $p_c = 1/2$). In fact, it is rather the opposite that happens. Indeed, looking at the peeling process while the red cluster is still alive introduces a biasing between the red and blue processes that is, somehow, compensated under the new measure $\P(\,\cdot\,|\,\Rc>0)$. The following striking identity illustrates the above statement: 
\[
\hbox{the sequences $(\T_n)_{n\geqslant 0}$ and $(\U_n)_{n\geqslant 0}$ have the same law under $\P_{(1,0)}(\,\cdot\,|\,\Rc>0)$.}
\]
This equality is a consequence of the following more general result upon which the entire proof of Proposition~\ref{lastprop} is built.

\begin{prop} \label{Remarkableprop} The following identities hold under  $\P_{(1,0)}(\,\cdot\,|\,\Rc>0)$.
\begin{enumerate}
\item[(a)] The increments $(\T_{n+1} - \T_n,\,\Vcb_{\T_{n+1}} - \Vcb_{\T_n},\, \Vcr_{\U_{n+1}} - \Vcr_{\U_n},\, \U_{n+1} - \U_{n})_{n\in \N}$ form a sequence of i.i.d.~random variables.
\item[(b)] The triplets $(T_1,\Vcb_{T_1} + \Vcr_{U_1},U_1)$ and $(U_1,\Vcb_{T_1} + \Vcr_{U_1},T_1)$ have the same law.
\item[(c)] The $3$-dimensional random walks 
\[(T_n+U_n,\,\Vcr_{U_n}+\Vcb_{T_n},\,T_n-U_n)_{n\in \N}\]
and 
\[(T_n+U_n,\,\Vcr_{U_n}+\Vcb_{T_n},\,U_n-T_n)_{n\in \N}\]
have the same law.
\end{enumerate}
\end{prop}

\begin{proof} The increments of $T$ and $U$ are i.i.d. as we have already noticed. Item~$(a)$ is easily checked using the same arguments. Item~$(b)$ is a consequence of the right-continuous properties of $\Rc$ and $\Bc$ and is proved in Proposition~\ref{SkipFree-Symp} of the appendix. Finally, item~$(c)$ follows directly from the combination of items~$(a)$ and~$(b)$. 
\end{proof}

We define $\Lambda$, the first index when the random walk $T$ is below $U$, \ie
\begin{equation*}
\Lambda \defeq \inf \,\{\, k\geqslant 1 : \, \T_k < \U_k\,\}. 
\end{equation*}
Equivalently, $\Lambda$ is the first descending ladder time for the random walk $(T_n - U_n)_{n\in \N}$. Item $(c)$ in the previous proposition shows that the step distribution of this random walk is symmetric. Moreover, this law has no atoms since it is a sum of exponential random variables. Therefore, the celebrated Sparre Andersen formula implies that the law of $\Lambda$ is universal: it does not depend on the particular form of the step distribution $T_1 - U_1$. More precisely, according to the corollary of Theorem~1 in chapter~XII.7 of~\cite{Fel71}, the generating function of $\Lambda$ is given by the simple formula
\begin{equation}\label{SPsym}
\E_{(1,0)}\big[s^{\Lambda}\,\big|\,\Rc>0\big] = 1 - \sqrt{1-s}\quad\mbox{for all $s\in [-1,1]$}.
\end{equation}

We will make use of the following estimates concerning $T$, $U$, $\Lambda$ and $\Vc$.

\begin{prop}\label{propAsympTU}
There exist positive constants $c_1$ and $c_2$ such that
\begin{align}
\label{densityUminusT}\P_{(1,0)}\big( \Lambda = n \,\big|\,\Rc>0\big)& \, \underset{n\to\infty}{\sim} \, \frac{1}{2\sqrt{\pi}\,n^{3/2}}\\
\label{tailT1}\P_{(1,0)}(\T_1 > t\,\big|\,\Rc>0) = \P_{(1,0}(\U_1 > t\,\big|\,\Rc>0) & \, \underset{t\to\infty}{\sim} \, \frac{c_1}{t^{1/3}}\\
\label{tailUpT}\P_{(1,0)}(\T_1 + \U_1 > t\,\big|\,\Rc>0)  & \, \underset{t\to\infty}{\sim} \, \frac{c_2}{t^{1/3}}
\end{align}
There exists $c_3>0$ such that, for any $x,n\geqslant 1$,
\begin{equation}\label{upboundV}
 \P_{(1,0)}\Big( \Vcr_{U_n}+\Vcb_{T_n}> x\,\Big|\,\Rc>0\Big)\leqslant  \frac{c_3 n}{x^{1/4}}.
\end{equation}
\end{prop}

\begin{proof} The asymptotic~\eqref{densityUminusT} follows from the Taylor expansion of the generating function of $\Lambda$ given by~\eqref{SPsym}. The tail estimate~\eqref{tailT1} is a restatement of item~\ref{StableRW-LadderTime} of Proposition~\ref{propStableRW}. \medskip

The random variable $T_1$ and $U_1$ are clearly not independent. Yet, according to Proposition~\ref{SkipFree-Symp} of the appendix, the random variable $T_1 + U_1$ under $\P_{(1,0)}(\,\cdot\,|\,\Rc>0)$ has the same law as $\smash{\widetilde{T}^{\scriptscriptstyle{\uparrow}}_{1} + \cdots  + \widetilde{T}^{\scriptscriptstyle{\uparrow}}_{L+1}}$, where the $\smash{\widetilde{T}^{\scriptscriptstyle{\uparrow}}_i}$'s are i.i.d~copies distributed as the first strict \emph{ascending} ladder time of the random walk $\Bc$ (\ie the first time when $\Bc$ hits $1$) and where $L$ is an independent random variable whose law is obtained by size biasing a typical step of $\Bc$.  In particular, we have $\smash{\P(L > n) \sim c/n^{1/2}}$ for some $c>0$. On the other hand, according to item~\ref{StableRW-LadderTime} of Proposition~\ref{propStableRW}, we also have $\smash{\P(\widetilde{T}^{\scriptscriptstyle{\uparrow}}_i > t) \sim c'/t^{2/3}}$ for some $c'>0$. These two estimates easily imply~\eqref{tailUpT} by computing, for instance, the Laplace transform of $T_1+U_1$ in term of the Laplace transforms of $\widetilde{T}^{\scriptscriptstyle{\uparrow}}_1$ and $L$ and then using Tauberian theorems to relate the tail distributions with the asymptotic of the Laplace transforms near $0$.\medskip

It remains to prove~\eqref{upboundV}. According to Proposition~\ref{SkipFree-Symp} of the appendix, under $\P_{(1,0)}(\,\cdot\,|\,\Rc>0)$, the random variables $\Vcb_{\T_1 - }$ and $\Vcr_{\U_1 -}$ have the same law. In particular, $\Vcb_{\T_1} + 1$ stochastically dominates $\smash{\Vcr_{\U_1}} = \smash{\Vcr_{\U_1 -}+1}$. On the other hand, according to Proposition~\ref{propVolRW2} of the appendix, we also have  $\smash{\P_{(1,0)}(\Vcb_{\T_1} > x) \lesssim x^{-1/4}}$. Therefore, a crude union bounds yields
\[\P_{(1,0)}\big( \Vcr_{U_1}+\Vcb_{T_1} > x \,\big|\,\Rc>0\big) \leqslant 2\,\P_{(1,0)}\Big( \Vcb_{T_1} > \frac{x-1}{2} \Big) \lesssim \frac{1}{x^{1/4}}.\]
Now, Proposition~\ref{Remarkableprop} states that the sequence $(\Vcr_{\U_{i+1}} - \Vcr_{\U_{i}}  + \Vcb_{\T_{i+1}} - \Vcb_{\T_{i}})_{i\in\N}$ is i.i.d~so we can apply Lemma~\ref{Sum-positive-stable} to conclude that~\eqref{upboundV} holds. 
\end{proof}

We end this section with a remarkable result which, in  a way, extends the universality property~\eqref{SPsym} for the distribution of the ladder times for symetric walks with continuous distributions. This key ingredient will enable us to decorrelate the random walks $(T_n-U_n)_{n\geqslant 0}$ and $(T_n+U_n)_{n\geqslant 0}$. We learned it from Vysotsky in~\cite{VV14} even though it seems to have been noticed previously by different authors. The following statement is a rewritting of the corolary that follows Proposition~2 of~\cite{VV14}.

\begin{prop}[\textbf{\!\!\textup{\cite{VV14}}}] \label{Vysotsky} Let $(X_n,Y_n)_{n\in\NN}$ be a bivariate random walk on $\R^2$ starting from $0$ and such that the law of the steps satisfies the relation
\[(X_1,Y_1)\overset{\law}=(X_1,-Y_1).\]
Assume furthermore that the law of $Y_1$ has no atom. Define the descending ladder time $\zeta \defeq \inf\,\{\, k\geqslant 1\, :\, Y_k < 0\,\}$.  Then, for any measurable set $A$ and for any $n\geqslant 1$, the events $\{\zeta > n\}$ and $\{X_n \in A\}$ are independent.
\end{prop}

\begin{cor} \label{VysotskyCor} Under the assumptions of Proposition~\ref{Vysotsky}, the events $\{\zeta = n\}$ and $\{X_n \in A\}$ are also independent. 
\end{cor}

\begin{proof} Since $(X,Y)$ is a random walk, we notice that, for any $n\in \N$, the increment $X_{n+1} - X_n$ is independent of $(X_i,Y_i)_{i\leqslant n}$. Therefore, using the Markov property at time $n$ combined with Proposition~\ref{Vysotsky}, we find that, for any measurable sets $A$ and~$B$, the events $\{\zeta > n\}$ and $\{X_n \in A,\,X_{n+1} \in B\}$ are also independent. We conclude that 
\begin{align*}
\P\big( \zeta = n,\, X_n \in A\big) & = \P\big( \zeta > n-1,\, X_n \in A\big) - \P\big( \zeta > n,\, X_n \in A\big)\\
& = \P\big( \zeta> n-1 \big) \P\big(X_n \in A\big) - \P\big(\zeta > n\big) \P\big(X_n \in A\big)\\
& = \P\big( \zeta = n\big)\P\big(X_n \in A\big).\qedhere
\end{align*}

\end{proof}

\subsection{Proof of Proposition~\ref{lastprop}}

\begin{lem}[\textbf{Lower bound for the peeling time}]
\begin{equation}\label{lowerbound_explo}
\P_{(1,0)}(\thetac > t\,|\,\Rc>0)\gtrsim \frac{1}{t^{1/6}}.
\end{equation}
\end{lem}
\begin{proof}
Let $n\geqslant 1$. In view of Proposition~\ref{CodingStoppingTimes}, we can write
\begin{align}
\nonumber\P_{(1,0)}(\thetac > t\,|\,\Rc>0) & \geqslant \P_{(1,0)}\big( \Lambda > n \hbox{ and } \T_n > t \,\big|\, \Rc>0\big)\\
\nonumber&  \geqslant  \P_{(1,0)}\big( \Lambda > n \hbox{ and } \T_n+ \U_n > 2t \,\big|\, \Rc>0\big)\\
\label{loboundpeel1} &  =  \P_{(1,0)}\big( \Lambda > n \,\big|\,\Rc>0\,\big)\,\P\big( \T_n+ \U_n >2t\,\big|\,\Rc>0\big),
\end{align}
where we used that $U_n>T_n$ on the event $\{\Lambda>n\}$ for the second inequality
and Proposition~\ref{Vysotsky} with $(X_n,Y_n) = (\T_n + \U_n, \T_n - \U_n)$ for the last equality. We now choose $n=\lfloor t^{1/3}\rfloor$. Equivalence~\eqref{tailUpT} of Proposition~\ref{propAsympTU} asserts that $(\T_n + \U_n)/n^{3}$ converges in law towards a positive stable random variable of index $1/3$. Therefore
\begin{equation}
\label{loboundpeel2} \lim_{t\to\infty }\P_{(1,0)}\big( \T_n + \U_n >2t\,\big|\,\Rc>0\big) = c > 0.
\end{equation}
On the other hand, item~\eqref{densityUminusT} of Proposition~\ref{propAsympTU} insures that
\begin{equation}
\label{loboundpeel3}\P_{(1,0)}\big(\Lambda > n\,\big|\,\Rc>0\,\big) \sim \frac{1}{\sqrt{\pi n}} \sim \frac{1}{\sqrt{\pi} t^{1/6}}.
\end{equation}
The combination of~\eqref{loboundpeel1}, \eqref{loboundpeel2} and~\eqref{loboundpeel3} together yields~\eqref{lowerbound_explo}.
\end{proof}

\begin{lem}[\textbf{Upper bound for the peeling time}]
\begin{equation*}
\P_{(1,0)}(\thetac > t\,|\,\Rc>0)\lesssim \frac{1}{t^{1/6}}.
\end{equation*}
 \label{LemUpBoundPeel}
\end{lem}

\begin{proof} According to Proposition~\ref{CodingStoppingTimes}, we have $\thetac \leqslant U_\Lambda$ so that
\begin{equation}\label{upboundpeel1}
\P_{(1,0)}(\thetac > t\,|\,\Rc>0)\leqslant \P_{(1,0)}\big( \U_\Lambda > t \,\big|\,\Rc>0\big) \leqslant \P_{(1,0)}\big(\T_\Lambda + \U_\Lambda > t\,\big|\,\Rc>0\big).
\end{equation}
We decompose the right hand side according to the value of $\Lambda$. This gives
\begin{align*}
\P_{(1,0)}\big(\T_\Lambda + \U_\Lambda > t\,\big|\,\Rc>0\big) & \,\leqslant \,\sum_{k=1}^\infty  \P_{(1,0)}\big(\T_k + \U_k > t\hbox{ and } \Lambda = k\,\big|\,\Rc>0\big)\\
&  =  \sum_{k=1}^\infty  \P_{(1,0)}\big(\T_k + \U_k > t\,\big|\,\Rc>0\big) \P_{(1,0)}\big(\Lambda = k\,\big|\,\Rc>0\big),
\end{align*}
where we used Corolary~\ref{VysotskyCor} with $(X_n,Y_n) = (\T_n + \U_n, \T_n - \U_n)$ for the last equality. We can use Proposition~\ref{propAsympTU} and Lemma~\ref{Sum-positive-stable} to upper bound the probabilities appearing in the sum above. We conclude that there exists $c>0$, such that, for $t$ large enough,
\begin{align}
\nonumber\P_{(1,0)}\big(\T_\Lambda + \U_\Lambda > t\,\big|\,\Rc>0\big) & \,\leqslant \,\sum_{k=1}^\infty  \Big(\frac{c k}{t^{1/3}}\wedge 1\Big)\frac{1}{\sqrt{\pi}k^{3/2}}\\
\nonumber&  =  \frac{c}{\sqrt{\pi}t^{1/3}}\,\sum_{k=1}^{\lfloor t^{1/3}\rfloor}\frac{1}{k^{1/2}} + \sum_{k>\lfloor t^{1/3}\rfloor} \frac{1}{\sqrt{\pi}k^{3/2}}\\
&\lesssim t^{-1/3}\,(\lfloor t^{1/3}\rfloor)^{1/2} + (\lfloor t^{1/3}\rfloor)^{-1/2} \lesssim \frac{1}{t^{1/6}}.\label{upboundpeel2}
\end{align}
Combining the last line with~\eqref{upboundpeel1} completes the proof of Lemma~\ref{LemUpBoundPeel}.
\end{proof}

\begin{lem}[\textbf{Upper bound for the volume}]
\begin{equation}\label{upperbound_volume}
\P_{(1,0)}(\Vc_{\thetac} > v\,|\,\Rc>0)\lesssim \frac{1}{v^{1/8}}.
\end{equation}
\end{lem}

\begin{proof}
According to Proposition~\ref{CodingStoppingTimes}, we have $\Vc_\thetac \leqslant \Vc_{U_\Lambda}$. We decompose
\begin{equation*}
\Vc_{\U_\Lambda} = (\Vcr_{\U_\Lambda}+\Vcb_{\T_\Lambda})+(\Vcb_{\U_\Lambda}-\Vcb_{\T_\Lambda}).
\end{equation*}
which, by union bounds, yields
\begin{align*}
\P_{(1,0)}(\Vc_{\thetac} > v\,|\,\Rc>0) & \leqslant  \P_{(1,0)}(\Vcr_{\U_\Lambda}+\Vcb_{\T_\Lambda} > v/2\,|\,\Rc>0) \\
& \qquad\qquad+ \P_{(1,0)}(\Vcb_{\U_\Lambda}-\Vcb_{\T_\Lambda} > v/2\,|\,\Rc>0).
\end{align*}
We bound each term separately.
\medskip
 
\textbf{\textit{First term ${\Vcr_{\U_\Lambda}+\Vcb_{\T_\Lambda}}$.}} We use the same method as for the upper bound of the peeling time. We write
\begin{multline*}
\P_{(1,0)}\big(\Vcr_{\U_\Lambda}+\Vcb_{\T_\Lambda} > v/2\,\big|\,\Rc>0\big)\\
\begin{aligned}
& = \, \sum_{k=1}^\infty \P_{(1,0)}\big(\Vcr_{\U_k}+\Vcb_{\T_k} > v/2\hbox{ and }\Lambda = k\,\big|\,\Rc>0\big)\\
& = \sum_{k=1}^\infty \P_{(1,0)}\big(\Vcr_{\U_k}+\Vcb_{\T_k} > v/2\,\big|\,\Rc>0\big) \P_{(1,0)}\big(\Lambda = k\,\big|\,\Rc>0\big)
\end{aligned}
\end{multline*}
where we used again Corollary~\ref{VysotskyCor} for the last equality, but this time with the bivariate random walk $(X_n,Y_n) = (\Vcr_{\T_n} + \Vcb_{\U_n}, \T_n - \U_n)$ which, thanks to Proposition~\ref{Remarkableprop}, satisfies all the required assumptions. We can bound each term in the sum above with the help of Proposition~\ref{propAsympTU}. We conclude that, as expected,
\begin{align*}
\P_{(1,0)}\big(\Vcr_{\U_\Lambda}+\Vcb_{\T_\Lambda} > v/2\,\big|\,\Rc>0\big) & \,\leqslant \,\sum_{k=1}^\infty  \Big(\frac{c k}{v^{1/4}}\wedge 1\Big)\frac{1}{\sqrt{\pi}k^{3/2}}
 \lesssim \frac{1}{v^{1/8}}.
\end{align*}
\medskip

\textbf{\textit{Second term $\Vcb_{\U_\Lambda}-\Vcb_{\T_\Lambda}$.}} We first notice that, conditionally on $\Rc$, the random variable $T_\Lambda$ is a stopping time for the random walk $\Bc$. But, then, the time-shifted process $(\Vcb_{\T_\Lambda + t}-\Vcb_{\T_\Lambda})_{t\geqslant 0}$ is independent of $\Rc$ and of $(\Bc_t)_{t\leqslant T_\Lambda}$. From this, we deduce that
$\Vcb_{\U_\Lambda}-\Vcb_{\T_\Lambda}$ has the same distribution as $\Vcb_{D}$, where $D$ is random variable which is independent of $\Vcb$ and which has the same distribution as $\U_\Lambda - \T_\Lambda$. Therefore, we can write
\begin{multline}\label{upperbound_vol1}
\P_{(1,0)}\big(\Vcb_{\U_\Lambda}-\Vcb_{\T_\Lambda}>v/2\,\big|\,\Rc>0\big)\\
\begin{aligned}
&=\, \P_{(1,0)}\big(\Vcb_D>v/2\,\big|\,\Rc>0\big)\\
&\leqslant\, \P_{(1,0)}\big(D>v^{3/4}\,\big|\,R>0\big)+\P_{(1,0)}\big(\Vcb_D>v/2,\, D\leqslant v^{3/4}\,\big|\,\Rc>0\big).
\end{aligned}
\end{multline}
According to Lemma~\ref{propVolRW2} in the appendix (with $\alpha=3/2$ and $\delta=3/4$), we have $\P_{(1,0)}(\Vcb_t>v) \lesssim t v^{-3/4}$. Hence,
\begin{align*}
\P_{(1,0)}\big(\Vcb_D > v/2 \,,D\leqslant v^{3/4}\,\big|\,\Rc>0\big)&= \E_{(1,0)}\big[\ind{D\leqslant v^{3/4}}\,\P(\Vcb_D>v/2\,|\,D)\,\big|\,\Rc>0\big]\\
&\lesssim \E_{(1,0)}\big[\ind{D\leqslant v^{3/4}}\,D\,v^{-3/4}\,\big|\,\Rc>0\big]\\
&=\frac{1}{v^{3/4}}\,\int_{0}^{v^{3/4}}\P\big(D>x\big)\,dx.
\end{align*}
Next, we observe that $\smash{D \overset{\law}{=} \U_\Lambda - \T_\Lambda \leqslant \U_\Lambda + \T_\Lambda}$. Thus, according to \eqref{upboundpeel2}, we have $\P_{(1,0)}\big(D>v\,|\,\Rc>0\big)\lesssim v^{-1/6}$. As a consequence, we obtain
\begin{equation*}
\P_{(1,0)}(\Vcb_D>v,\,D\leqslant v^{3/4}\,|\,R>0) \lesssim \frac{1}{v^{3/4}}\,\int_{0}^{v^{3/4}}\frac{1}{x^{1/6}}\,dx  \lesssim  \frac{1}{v^{1/8}}.
\end{equation*}
Combining this estimate with~\eqref{upperbound_vol1} yields the matching upper bound for the second term
\begin{equation*}
\P_{(1,0)}\big(\Vcb_{\U_\Lambda}-\Vcb_{\T_\Lambda}>v/2\,\big|\,\Rc>0\big)\lesssim \frac{1}{(v^{3/4})^{1/6}}+\frac{1}{v^{1/8}}\lesssim \frac{1}{v^{1/8}}
\end{equation*}
which completes the proof of~\eqref{upperbound_volume}.
\end{proof}

\begin{lem}[\textbf{Lower bound for the volume}]
\begin{equation*}
\P_{(1,0)}\big(\Vcr_{\thetac^-} > v\,\big|\,\Rc>0\big)\gtrsim \frac{1}{v^{1/8}}.
\end{equation*}
\end{lem}

\begin{proof}
The approach for this lower bound differs from that used for the peeling time $\theta$. Here, we do not rely on Proposition~\ref{Vysotsky} but we construct instead an event insuring that $\Vcr_{\thetac^-} > v$ and which has the required probability. 
\medskip

Let $v>0$ and set $t \defeq v^{3/4}$. Let also $\gamma > 2$ be a constant whose exact value will be chosen later on. Consider the two events
\begin{align*}
\mathcal{E}_1 &\defeq\bigg\{\theta>t,\, \Bcu_t> -{t^{2/3}},\, \Bc_t\geqslant 0,\,\Rc_t> \gamma t^{2/3}\bigg\},\\
\mathcal{E}_2 & \defeq \bigg\{\Vcr_{2t}-\Vcr_{t}>v,\,\inf_{t\leqslant i\leqslant 2t}\Bc_i > -{t^{2/3}},\,\inf_{t\leqslant i\leqslant 2t}\Rc_i> \frac{\Rc_t}{2}\bigg\}.
\end{align*}

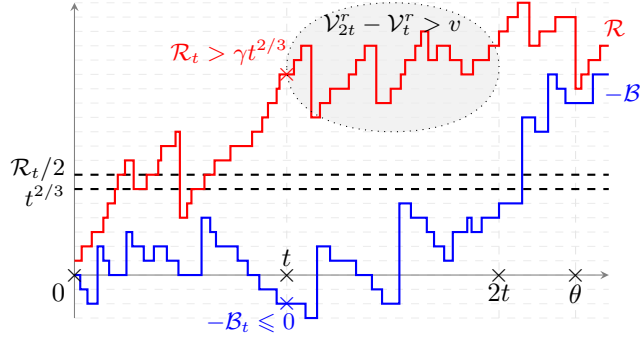
\begin{figure}[ht]
\centering
\begin{tikzpicture}[scale =0.19]
\fill[color=gray!10] (3.7,11.5) to[out=90,in=180] (11.05,16) to[out=0,in=90] (18.4,11.5) to[out=-90,in=0] (11.05,7) to[out=180,in=-90](3.7,11.5);
\foreach \k in {-6,-5,...,16} {\draw [color=gray!15,dashed] (-11,\k)-- (26,\k);}
\draw [>=stealth,->,color=gray!100] (-11,-3) -- (26,-3) ;
\draw [>=stealth,->,color=gray!100] (-11,-6) -- (-11,16) ;
\draw [color=gray!20,dashed] (3.7,-6) -- (3.7,16) ;
\draw [color=gray!20,dashed] (18.4,-6) -- (18.4,16) ;
\draw [color=gray!20,dashed] (23.7,-6) -- (23.7,16) ;
\draw [thick,dashed,color=black] (-11,3) -- (26,3) ;
\draw [color=black] (-11,2.8-0.2) node[left] {\small{$t^{2/3}$}};
\draw [thick,dashed,color=black] (-11,4) -- (26,4) ;
\draw [color=black] (-11,4.2+0.2) node[left] {\small{$\Rc_t/2$}};
\draw [thick,samples=1000,color=red]  (-11,-2)--(-10.5,-2)--(-10.5,-1)--(-9.8,-1)--(-9.8,0)--(-9,0)--(-9,1)--(-8.7,1)--(-8.7,2)--(-8.2,2)--(-8.2,3)--(-8,3)--(-8,4)--(-7.5,4)--(-7.5,5)--(-6.9,5)--(-6.9,3)--(-6,3)--(-6,4)--(-5.2,4)--(-5.2,5)--(-5,5)--(-5,6)--(-4,6)--(-4,7)--(-3.7,7)--(-3.7,1)--(-3.2,1)--(-3.2,2)--(-2.9,2)--(-2.9,3)--(-2,3)--(-2,4)--(-1.3,4)--(-1.3,5)--(-0.2,5)--(-0.2,6)--(1.3,6)--(1.3,7)--(2,7)--(2,8)--(2.4,8)--(2.4,9)--(3,9)--(3,10)--(3.2,10)--(3.2,11)--(4.2,11)-- (4.2,12) -- (4.7,12) -- (4.7,13) -- (5.4,13) -- (5.4,8) -- (6,8) -- (6,9) -- (6.7,9) -- (6.7,10) -- (8,10) -- (8,11) -- (8.5,11)--(8.5,12)-- (9.1,12)--(9.1,13)--(9.9,13)--(9.9,9)--(10.9,9)--(10.9,10) -- (11.6,10) -- (11.6,11) -- (11.9,11) -- (11.9,12) -- (12.6,12) -- (12.6,13) -- (13,13) -- (13,14) -- (13.5,14) -- (13.5,12) -- (14.1,12) --(14.1,13)--(15,13) -- (15,12) -- (15.8,12) -- (15.8,11) -- (16.6,11) -- (16.6,12) -- (17.7,12) -- (17.7,13)-- (18.4,13)--(18.4,14)--(19,14)--(19,15)--(19.7,15)--(19.7,16)--(20.5,16)-- (20.5,13)--(21.2,13)--(21.2,14)--(22.3,14)--(22.3,15)--(23.7,15)--(23.7,10)--
(24.1,10)--(24.1,11)--(24.6,11)--(24.6,12)--(25.4,12)--(25.4,13)--(26,13);
\draw [thick,samples=1000,color=blue] 
(-11,-3)--(-10.6,-3) -- (-10.6,-4) -- (-10.1,-4) -- (-10.1,-5) -- (-9.4,-5) -- (-9.4,-1) -- (-9,-1) -- (-9,-2) -- (-8.6,-2) -- (-8.6,-3) -- (-7.4,-3) -- (-7.4,0) -- (-7,0) -- (-7,-1) -- (-6.2,-1) -- (-6.2,-2) -- (-5.5,-2) -- (-5.5,-1) -- (-4.6,-1) -- (-4.6,-2) -- (-3.8,-2) -- (-3.8,-3) -- (-2.2,-3) -- (-2.2,1) -- (-1.6,1) -- (-1.6,0) -- (-1.2,0) -- (-1.2,-1) -- (0.3,-1) -- (0.3,-2) -- (1,-2) -- (1,-3) -- (1.8,-3) -- (1.8,-4) -- (3,-4) --(3,-5) -- (5,-5) -- (5,-6) -- (5.8,-6) -- (5.8,-1) -- (6.7,-1) -- (6.7,-2) -- (8.5,-2) -- (8.5,-3) -- (9.6,-3) -- (9.6,-4) -- (10.3,-4) -- (10.3,-5) -- (11.5,-5) -- (11.5,2) -- (12.3,2) -- (12.3,1) -- (13.4,1) -- (13.4,0) --(13.9,0)-- (13.9,-1)--(14.6,-1)--(14.6,-2)--(15.2,-2)--(15.2,0)--(16.2,0) -- (16.2,1) -- (16.5,1) -- (16.5,0) --(17.2,0)--(17.2,1)-- (18.4,1) -- (18.4,2) --(20,2) -- (20,8) -- (20.9,8) -- (20.9,7) -- (21.8,7) -- (21.8,11) -- (22.4,11) -- (22.4,10) -- (23,10) -- (23,9) -- (24.9,9) -- (24.9,11) -- (26,11) ;
\draw [thick,color=red] (25,13) node[above right] {\small{$\Rc$}};
\draw [thick,color=blue] (25,11) node[below right] {\small{$-\Bc$}};
\draw (-11,-3) node[below left] {$0$};
\draw (-11,-3) node {$\times$};
\draw (3.7,-3) node[above] {$t$};
\draw (3.7,-3) node {$\times$};
\draw (18.4,-3) node[below] {$2t$};
\draw (18.4,-3) node {$\times$};
\draw (3.7-4,11) node[color=red,above] {\small{$\Rc_{t}> \gamma t^{2/3}$}};
\draw (3.7,11) node[color=red] {$\times$};
\draw (3.7-2.5,-5) node[color=blue,below] {\small{$-\Bc_{t}\leqslant 0$}};
\draw (3.7,-5) node[color=blue] {$\times$};
\draw (23.7,-3) node[below] {$\theta$} ;
\draw (23.7,-3) node {$\times$} ;
\draw[dotted] (3.7,11.5) to[out=90,in=180] (11.05,16) to[out=0,in=90] (18.4,11.5) to[out=-90,in=0] (11.05,7) to[out=180,in=-90](3.7,11.5);
\draw (11,14.5) node {\small{$\Vc^r_{2t}-\Vc^r_t>v$}};
\end{tikzpicture}
\caption{An example of the event $\mathcal{E}_1\cap \mathcal{E}_2$. The volume generated by the red steps in the gray zone is bigger than $v$.}
\end{figure}

We observe that, on $\mathcal{E}_1 \cap \mathcal{E}_2$, we must have $\thetac > 2t$. Indeed, we first have $\thetac > t$ because of $\mathcal{E}_1$ but, then, $\mathcal{E}_2$ prevents $\Rc$ from going below
\[\Rc_t/2 > \frac{\gamma}{2}\,t^{2/3} > t^{2/3} > -\Bcu_{2t}\]
during the time interval $[t,2t]$. This insures that $\thetac > 2t$. Moreover, $\mathcal{E}_2$ also requires that the volume accumulated during the time interval $[t,2t]$ is greater than $v$. Therefore we conclude that $\mathcal{E}_1 \cap \mathcal{E}_2 \subset \{ \Vcr_{\theta^-} > v \}$ which yields the lower bound
\begin{equation*}
\P_{(1,0)}(\Vcr_{\theta^-} > v\,|\,\Rc>0)\geqslant\P_{(1,0)}(\mathcal E_1\,|\,\Rc>0)\,\P_{(1,0)}(\mathcal E_2\,|\,\mathcal E_1,\,\Rc>0).
\end{equation*}
We bound the probability of each event separately. 
\medskip

\textbf{\textit{Event $\mathcal{E}_1$.}} We prove that
\begin{equation}\label{lowboundVolE1}
\P_{(1,0)}(\mathcal E_1\,|\,\Rc>0) \gtrsim \frac{1}{v^{1/8}}.
\end{equation}
First, notice that, conditionally on $\Rc$, the events
$\{\theta>t\}$, $\{\Bcu_t> -{t^{2/3}}\}$ and  $\{\Bc_t\geqslant 0\}$ are all increasing with respect to the increments of $\Bc$. Thus, Harris inequality stated in Proposition~\ref{HarrisProp}\footnote{We use here a version for jump processes instead of discrete time processes but the adaptation is straightforward so the details are omitted} shows that
\begin{align*}
\P_{(1,0)}\big(\mathcal E_1\,|\,\Rc>0\big)& \geqslant \E_{(1,0)}\Big[\P_{(1,0)}(\theta>t,\,\Rc_t> \gamma t^{2/3}\,|\,\Rc)\,\P_{(1,0)}(\Bc_t\geqslant 0\,|\,\Rc)\\
&\qquad\qquad\qquad\qquad\qquad\qquad\qquad\P_{(1,0)}(\Bcu_t> -{t^{2/3}}\,|\,\Rc)\,\Big|\,\Rc>0\Big]\\
& = \P_{(1,0)}(\theta>t,\,\Rc_t> \gamma t^{2/3}\,|\,\Rc > 0)\,\P_{0}(\Bc_t\geqslant 0)\,\P_{0}(\Bcu_t> -{t^{2/3}}),
\end{align*}
where we used the fact that $\Bc$ and $\Rc$ are independent for the last equality. Now, since $\Bc$ lies in the normal domain of attraction of a stable law of index $3/2$, the probabilities $\P_{0}(\Bc_t\geqslant 0)$ and $\P_{0}(\Bcu_t> -{t^{2/3}})$ both converges to strictly positive constants. This means that
\[\P_{(1,0)}(\mathcal E_1\,|\,\Rc>0) \gtrsim  \P_{(1,0)}(\theta>t,\,\Rc_t> \gamma t^{2/3}\,|\,\Rc>0).\]
Applying again Harris inequality, but this time with respect to the increment of $\Rc$ conditioned to stay positive, we deduce that
\[\P_{(1,0)}(\mathcal E_1\,|\,\Rc>0)\gtrsim \P_{(1,0)}(\theta>t\,|\,\Rc>0)\,\P_{1}(\Rc_t> \gamma t^{2/3}).\]
According to~\eqref{lowerbound_explo}, we have $\smash{\P_{(1,0)}(\theta>t\,|\,\Rc>0) \gtrsim t^{-1/6} = v^{-1/8}}$. On the other hand,
$\smash{\P_{1}(\Rc_t> \gamma t^{2/3})}$ is bounded away from $0$, uniformly in $t$ because $\Rc$ is also in the normal domain of attraction of a stable law of index $3/2$ (the lower bound depends on $\gamma$ but is always strictly positive). Putting everything together, we conclude that~\eqref{lowboundVolE1} holds.
\medskip

\textbf{\textit{Event $\mathcal{E}_2$.}} We prove that, conditionally on $\mathcal{E}_1$, the event $\mathcal{E}_2$ is typical, \ie
\begin{equation}\label{lowboundVolE2}
\P_{(1,0)}(\mathcal E_2\,|\,\mathcal E_1,\Rc>0)  > c,
\end{equation}
where $c>0$ does not depend on $v$. Using Markov property of $(\Rc,\Bc)$ at time $t$ and recalling also that $\Rc$ and $\Bc$ are independent, we find that
\begin{align*}
\P_{(1,0)}(\mathcal E_2\,|\,\mathcal E_1,\Rc>0)&\geqslant\inf_{b\geqslant 0,\,r\geqslant \gamma t^{2/3}}\,\P_{(r,b)}\big(\Vcr_{t}>v,\,\Bcu_t> -{t^{2/3}},\,\Rcu_t>r/2\,\big|\,\Rc>0\big)\\
&=\inf_{r\geqslant \gamma t^{2/3}}\,\P_{r}\big(\Vcr_{t}>v,\,\Rcu_t>r/2\,\big|\,\Rc>0\big)\inf_{b\geqslant 0} \,\P_{b}\big(\Bcu_t> -{t^{2/3}}\big).
\end{align*}
We already noticed that
\[\inf_{t\geqslant 0}\inf_{b\geqslant 0} \P_{b}(\Bcu_t> -{t^{2/3}})  =  \inf_{t\geqslant 0} \P_{0}(\Bcu_t> -{t^{2/3}}) > 0\]
so that
\[\P_{(1,0)}\big(\mathcal E_2\,\big|\,\mathcal E_1,\Rc>0\big)\gtrsim \P_{r}\big(\Vcr_{t}>v,\,\Rcu_t>r/2\,\big|\,\Rc>0\big).\]
We rewrite the probability on the right hand side via the h-transform
\begin{align}
\nonumber\P_{r}(\Vcr_{t}>v,\,\Rcu_t>r/2\,|\,\Rc>0) & =  \E_r\bigg[\ind{\Vcr_{t}>v,\,\Rcu_t>r/2}\frac{h(\Rc_t)}{h(r)}\bigg]\\
\nonumber& \geqslant   \frac{1}{2}\P_r(\Vcr_{t}>v,\,\Rcu_t>r/2)\\
\label{lowboundVol_2}& \geqslant   \frac{1}{2}\big(\P_{r}(\Vcr_{t}>v)-\P_r(\Rcu_t\leqslant r/2)\big),
\end{align}
where we used that $h$ is non-decreasing and sub-additive to lower bound $h(\Rc_t)/h(r)$ by $1/2$. Now, we notice that $\P_{r}(\Vcr_{t}>v)$ does not depend on the starting point $r$ of $\Rc$ since the volume depend on $\Rc$ only through its increments, which are i.i.d.~under~$\P$. Moreover, according to Proposition \ref{annealedVolStep} of the appendix, the step distribution of $\Vcr$ under $\P$ is in the normal domain of attraction of a positive stable law of index $3/4$. Therefore, $\P_{r}(\Vcr_{t}>v) = \P_{r}(\Vcr_{t}> t^{4/3})$ remains bounded away from $0$ as $t$ increase. On the other hand, for $r\geqslant \gamma t^{2/3}$, we have, according to item \ref{StableRW-CV} of Proposition \ref{propStableRW} of the appendix, 
\[\P_r(\Rcu_t\leqslant r/2)\leqslant \P_0\Big(\Rcu_t\leqslant -\frac{\gamma}{2} t^{2/3}\Big) \underset{t\to\infty}{\longrightarrow} \P\left(\underline{\mathcal{X}}_t \leq -\frac{\gamma}{2}\right),\] 
where $\mathcal{X}$ is a spectrally negative strictly stable process of index $3/2$.  Thus, the limit above decreases to $0$ as $\gamma$ increases to infinity. This means that, choosing $\gamma$ large enough, we can lower bound~\eqref{lowboundVol_2} uniformly for all large $t$ and $r\geqslant \gamma t^{2/3}$. This completes the proof of~\eqref{lowboundVolE2}.
\end{proof}

\begin{rmk}\label{remarkRealVolume}
Now that the proof of Theorem \ref{mainTheo} is complete, we can point out the changes required to study the red cluster $\C$ instead of its hull $\Hl$. Clearly, not all quantities are well adapted to the peeling procedure. For example, the method seems unfit to estimate the diameter of $\C$. On the other hand, controlling the number $|\C|$ of red sites in the cluster of the origin should be possible provided that we could estimate the typical size of a cluster in a free Boltzmann triangulation.\medskip

More precisely, consider a colored free Boltzmann triangulation of the $(m + 1)$-gon whose boundary is composed only of red vertices. Denote by $W^{(m)}$ the size of the red cluster touching the outer boundary, with the convention 
$W^{(m)} = 1$ for $m< 1$. Suppose that the two following estimates hold for some $\beta \in (3/2,2]$:
\[\textbf{(1)}\hspace{1cm}\P\big(W^{(-\xi)}\geq x\big)\underset{x\to+\infty}{\sim}cx^{-\alpha/\beta}\]
where $\alpha = 3/2$ and $\xi$ is random variable with distribution given by \eqref{defrwpk} which is independent of $W$.  
\[\textbf{(2)}\hspace{1cm}\sup_{j\geqslant 1}\,\frac{\E[W^{(j)}]}{j^{\beta}}<+\infty.\]
Then, following the strategy used in this paper for studying $|\Hl|$, we can prove that
\[\Pg(|\C|>n)\asymp n^{-\alpha/(6 \beta)}.\]
While estimates \textbf{(1)} and \textbf{(2)} are not yet available, we learned from Nicolas Curien that they may hold for $\beta = 7/4$ which suggest the following conjecture:
\begin{conj} The cluster $\C$  for a critical site percolation on the UIPT satisfies \[\Pg(|\C|>n)\asymp n^{-1/7}.\]
\end{conj}
\end{rmk}

\addcontentsline{toc}{section}{Appendix}
\renewcommand{\thesubsection}{\Alph{subsection}}
\renewcommand{\thetheo}{\Alph{subsection}.\arabic{theo}}
\renewcommand{\thesubsubsection}{\roman{subsubsection})}
\setcounter{subsection}{0}
\setcounter{section}{0}
\section*{Appendix}

\subsection{Bivariate random walk}\label{BivRW}
\setcounter{theo}{0}

In all this section, we denote by $(\xi_n,\vv_n)_{n\geqslant 1}$ a sequence of i.i.d.~random variables taking values in $\Z\times \R^+$ (the random variables $\xi_n$ and $\vv_n$, $n\geq 1$, are not assumed to be independent). We also denote by $(N_t)_{t\geqslant 0}$ a Poisson process with unit intensity, independent of the previous sequence. We define the continuous-time bivariate random walk $(\Sc_t,\Vc_t)_{t\geq 0}$ starting from $(0,0)$ by
\begin{equation*}
\Sc_t\defeq \sum_{n\leqslant N_t}\xi_n\qquad\mbox{and}\qquad \Vc_t\defeq\sum_{n\leqslant N_t}\vv_n.
\end{equation*}
We make the following additional assumptions:
\begin{enumerate}
\item[\textbf{(a)}] The random walk $\Sc$ is centered and right-continuous, \ie
\begin{equation*}
\E[\xi_1]=0\qquad\mbox{and}\qquad \mathrm{supp}(\xi_1) \subset \{ \ldots,-3,-2,-1,0,1\}.
\end{equation*}
\item[\textbf{(b)}] The $\xi_i$'s are in the normal domain of attraction of a $\alpha$-stable law, for some $\alpha\in (1,2)$, \ie there exist a constant $c >0$ such that
\begin{equation}\label{asumpt2}
\P(\xi_1< -k)\sim c\,k^{-\alpha}\qquad\mbox{as $k\to+\infty$.}
\end{equation}
\item[\textbf{(c)}] The $\vv_i$'s are non negative and lie in the normal domain of attraction of a $\delta$-stable law for some $\delta \in (0,1)$, \ie 
there exist a constant $d >0$ such that
\begin{equation}\label{asumpt3}
\P(\vv_1> x )\sim d\, x^{-\delta}\qquad\mbox{as $x\to+\infty$.}
\end{equation}
\item[\textbf{(d)}] We have 
\begin{equation}\label{asumpt4}
\sup_{j\leqslant 1}\,\frac{\E[\vv_1\,|\,\xi_1=j]}{1+|j|^{\alpha/\delta}}<+\infty.
\end{equation}
\end{enumerate}
Let us keep in mind that we will be particularly interested in the case where $\alpha=3/2$ and $\delta=3/4$: we have then $\alpha/\delta=2$, which matches the bivariate random walk obtained by considering the random walk associated with the peeling process together with the volume generated by the free Boltzmann triangulations discovered during the peeling. Indeed, for this bivariate random walk, Assumption~\textbf{(a)} follows from~\eqref{defrwpk}, Assumption~\textbf{(b)} follows from~\eqref{asymppk}, and Assumption~\textbf{(d)} follows from~\eqref{formulaVolumeExpectation}. Finally, Assumption~\textbf{(c)} is a consequence of the explicit formula~\eqref{formulaVolume} and is proved in Proposition~\ref{annealedVolStep}.\medskip

We define the strict descending and ascending ladder times of $\Sc$ by
\[
\Tdown \defeq \inf\,\{\,t>  0: \Sc_t<0\,\} \qquad \hbox{and} \qquad \Tup \defeq \inf\,\{\,t>  0: \Sc_t > 0\,\}.
\]
Since $\Sc$ is oscillating, both ladder times are well defined a.s. We start by recalling classical results concerning fluctuations of random walks in the domain of attraction of a stable law. 
\begin{prop}
\begin{enumerate}
\item\label{StableRW-CV} The sequence of processes $(\Sc_{ nt}/n^{1/\alpha})_{t\in[0,1]}$ converges in law, in the Skorokhod space, for the $J_1$ topology, towards a spectrally negative strictly stable process $(\mathcal{X}_{t})_{t\in[0,1]}$ of index $\alpha$. In particular, for any $t\in (0,1]$ and $x>0$, we have
\begin{equation}\label{contminfun}
\lim_{n\to\infty} \,\P\left(\frac{1}{n^{1/\alpha}}\inf_{s\leq t}\Sc_{ns} <-x \right) = \P\left(\inf_{s\leq t}\mathcal{X}_s<-x\right) \in (0,1).
\end{equation}
Similarly $(\Vc_{nt}/n^{1/\delta})_{t\in[0,1]}$ converges in the Skorokhod space towards a positive strictly stable process $(\mathcal{Y}_{t})_{t\in[0,1]}$ of index $\delta$.
\item \label{StableRW-LadderTime} There exist constants $c_1, c_2>0$ such that, when $x$ goes to $+\infty$,
\[\P(\Tdown > x)\sim \frac{c_1}{x^{1-1/\alpha}}\qquad \mbox{and}\qquad \P(\Tup > x)\sim \frac{c_2}{x^{1/\alpha}}.\]
\end{enumerate}
\label{propStableRW}
\end{prop}

\begin{proof} 
The convergences in the Skorokod space stated in item \ref{StableRW-CV} is simply Donsker's theorem for Lévy processes (\emph{c.f.} for instance Theorem~16.14 of \cite{Kall02}). The limit~\eqref{contminfun} follows from the fact that the functional $f\mapsto \min_{s\leq t}f(s)$ is almost surely continuous with respect to the law on the trajectory of $\mathcal{X}$. Moreover, the support of any stable law with index $\alpha>1$ is necessarily the whole of $\R$ hence this limit is strictly between $0$ and $1$. The tail distribution of $\Tdown$ stated in item $2$ follow from Theorem~1 of Doney~\cite{Don82} (together with an easy change of time given by the Poisson process $N$). In turn, this estimate combined with Sparre-Andersen's formula (\emph{c.f.} Feller~\cite{Fel71} Chapter~XII.7.) insures that $\Tup$ is also in the domain of normal attraction of a positive stable law, this time with index $1/\alpha$. 
\end{proof}

\begin{lem} Let $\gamma\in (0,1)$. Suppose that $(Y_i)_{i\geqslant 1}$ is a sequence of i.i.d.~positive random variables such that $\P(Y_1>x) \lesssim x^{-\gamma}$ as $x$ goes to $+\infty$. Then, there exists a constant $c < \infty$ such that, for all $n\geqslant 1$ and for all $x>0$,
\[\P(Y_1+\dots+Y_n > x)\leqslant c n x^{-\gamma}.\]
 \label{Sum-positive-stable}
\end{lem}

\begin{proof} We use a coupling argument. Let $X$ denote a positive strictly stable random variable with index $\gamma$. We have that $\P(X>x) \sim c x^{-\gamma}$ for some $c > 0$. Thus, the assumption on the tail of the $Y_1$ shows that we can choose $a$ and $b$ large enough such that $Y_1$ is stochastically dominated by $a + b X$. Denoting by $(X_i)_{i\geqslant 1}$ a sequence of i.i.d.~random variable with the same law as $X$, we conclude that
\begin{align*}
\P(Y_1+\dots+Y_n>x)&\leqslant \P\big(na+b(X_1+\dots+X_n)>x\big)\\
&=\P\left(X>\frac{x-na}{bn^{1/\gamma}}\right)\leqslant Cnx^{-\gamma}.
\end{align*}
The above upper bound holds uniformly in $n\geq 1$ and $x>0$, provided that $C$ is chosen large enough. 
\end{proof}

The hypothesis that $\xi$ is skip free has many important consequences. One of them being an explicit path decomposition for an excursion of the random walk $\Sc$ which, in turn, leads to several remarkable identities. We gather here some of these results. They are classical but seem to be scattered throughout the existing literature.
\medskip

We need some additional notations. We define the size $L$ of the jump at time $\Tdown$ and its undershoot $H$ by
\begin{equation}\label{defHandL}
L \defeq \Sc_{{\Tdown}\scriptscriptstyle{-}} -\Sc_{\Tdown} \qquad\hbox{and}\qquad H \defeq -\Sc_{\Tdown}.
\end{equation}
Let now $(\Sh_t,\Vh_t)_{t\geqslant 0}$ denote a new Markov process, independent of $(\Sc,\Vc)$ and whose law is the same as that of $(\Sc  +1,\Vc)$ under the conditional measure $\P(\,\cdot\,|\, \Sc  + 1 > 0)$, \ie such that\footnote{As explained in section \ref{SectionEncoding}, the law of $(\Sh_t,\Vh_t)_{t\geqslant 0}$ can also be obtained from that of $(\Sc,\Vc)$ (starting from $(1,0)$) by a change of measure using the h-transform $h$ of $\Sc$, where $h$ is harmonic for $\Sc$ on $\llbracket 1,+\infty \llbracket$ and zero outside.}
\[\law(\Sh,\Vh) = \law \big( (\Sc + 1,\Vc) \, | \, \Sc +1 > 0\big).\]
In particular, this process starts from $(\Sh_0,\Vh_0) = (1,0)$ and $\Sh$ has the law of the random walk $\Sc$ starting from $1$ and conditioned to stay positive forever. Thus, it is a transient Markov process that diverges to $+\infty$. Moreover, since $\Sc$ is right-continuous, so is $\Sh$ and thus the last passage times at any heights are well defined. Therefore we can set
\[U \defeq \sup\,\{\,t>  0: \Sh_{t}=H\,\}.\]

\begin{prop}  \label{SkipFree-Symp}
Let $\smash{\big(\widetilde{\mathcal{V}}^{(i)},\widetilde{T}^{\scriptscriptstyle{\uparrow}}_i\big)}$, $i\geqslant 1$, be independent copies of $\smash{\big(\Vc_{\Tu},\Tu\big)}$.
\begin{enumerate}
\item The following identities in law hold:
\begin{align*}
\Tdown \,\overset{\law}{=}\, U &\,\overset{\law}{=}\, \widetilde{T}^{\scriptscriptstyle{\uparrow}}_1+\dots+ \widetilde{T}^{\scriptscriptstyle{\uparrow}}_{H},\\
\Tdown+U &\,\overset{\law}{=}\,  \widetilde{T}^{\scriptscriptstyle{\uparrow}}_1+\dots+ \widetilde{T}^{\scriptscriptstyle{\uparrow}}_{L + 1},\\
\Vc_{\Tdown\scriptscriptstyle{-}} \,\overset{\law}{=}\,\Vh_{U\scriptscriptstyle{-}} &\,\overset{\law}{=}\, \widetilde{\mathcal{V}}^{(1)}+\dots+ \widetilde{\mathcal{V}}^{(H)}\\
(\Tdown,U,\Vc_{\Tdown}+\Vh_U)&\,\overset{\law}{=}\,(U,\Tdown,\Vc_{\Tdown}+\Vh_U).
\end{align*}
\item The law of $L$ and $H$ are obtained from $\xi$ by size biasing: for $k\geq 1$, we have
\begin{equation*}
\P(L=k)=k\,\frac{\P(\xi_1= -k)}{\P(\xi_1=1)} \qquad \hbox{and} \qquad \P(H=k)=\frac{\P(\xi_1\leqslant -k)}{\P(\xi_1=1)}.
\end{equation*}
\end{enumerate}
\end{prop}

\begin{proof}
Item~1 uses arguments similar to that previously developed by Vysotsky (\emph{c.f.} Lemma~2 in~\cite{VV10}.) It is based on the observation that the law of a negative excursion of a right-continuous random walk is invariant by time reversal. 
\medskip

Define $\smash{\widehat{H}\defeq\Sc_{T\scriptscriptstyle{-}}}$ so that $L = \widehat{H} + H$. Define also $U'$ as the time it takes for the walk $\Sc$ to go back to $0$ after time $T$ \emph{i.e.}
\begin{equation*}
U'\defeq\inf\,\{\,t>0:\Sc_{T+t}=0\,\}.
\end{equation*}
Notice in particular that, since $\Sc$ is right-continuous, we have 
$\Sc_{(T+U')\scriptscriptstyle{-}}=-1$. Now, consider the new process $(\St,\Vt)$ obtained by reversing time on the interval $[0,T+U']$ in the following way:
\begin{align*}
\St_t & =
\left\{\begin{array}{ll}
-1-\Sc_{(T+U'-t)\scriptscriptstyle{-}}&\mbox{if $0\leq t<T+U'$,}\\
\Sc_t&\mbox{if $t\geqslant T+U'$.}
\end{array}\right.\\
\Vt_t &=
\left\{\begin{array}{ll}
\Vc_{(T+U')\scriptscriptstyle{-}}-\Vc_{(T+U'-t)\scriptscriptstyle{-}}&\mbox{if $0\leq t<T+U'$,}\\
\Vc_t&\mbox{if $t\geqslant T+U'$.}
\end{array}\right.
\end{align*}
See figure~\ref{fig:Reverting} for an illustration of this transformation. It is clear that the mapping $(\Sc,\Vc)\longmapsto(\St,\Vt)$ is an involution which preserves the measure on random walk paths. As a consequence
\[(\St,\Vt)\overset{\law}=({\Sc},{\Vc}).\]
Moreover, the transformation preserves the size $L$ of the jump below $0$ while exchanging the values of $T$ and $U'$. This implies that $\smash{(T,U',L) \overset{\law}{=} (U',T,L)}$. Using similar argument, it is easy to check that, more generally, the following  joint identity hold:
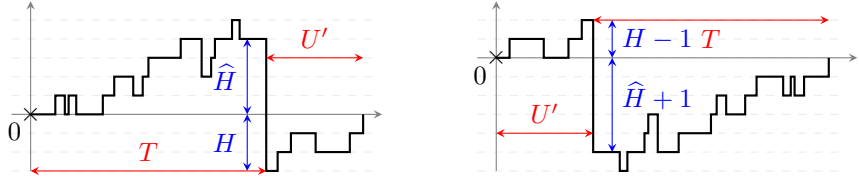
\begin{figure}[t]
\centering
\begin{tabular}{ccc}
\begin{tikzpicture}[scale =0.25]
\foreach \k in {-3,-2,...,5} {\draw [color=gray!15,dashed] (-1,\k)-- (18,\k);}
\draw [>=stealth,->,color=gray!100] (0,-3) -- (0,6) ;
\draw [>=stealth,->,color=gray!100] (-1,0) -- (18.5,0) ;
\draw [samples=1000,thick](0,0) -- (1.3,0)-- (1.3,1) -- (1.8,1)--(1.8,0)--(2,0)--(2,1)--(2.4,1)-- (2.4,0) -- (3.8,0) -- (3.8,1) -- (4.4,1) -- (4.4,2) -- (5.4,2) -- (5.4,1) -- (5.9,1) -- (5.9,2) -- (6.2,2) -- (6.2,3) -- (7.9,3) -- (7.9,4) -- (9,4) -- (9,2) -- (9.5,2) -- (9.5,3) -- (9.7,3) -- (9.7,4) -- (10.6,4) -- (10.6,5) -- (11,5) -- (11,4) -- (12.4,4) -- (12.4,-3) -- (13,-3) -- (13,-2) -- (13.7,-2) -- (13.7,-1) -- (15,-1) -- (15,-2) -- (16.8,-2) -- (16.8,-1) -- (17.5,-1) -- (17.5,0) ;
\draw [>=stealth,<->,color=blue] (11.4,4) -- (11.4,0) ;
\draw [thick,color=blue] (11.4,2) node[left] {$\widehat{H}$} ;
\draw [>=stealth,<->,color=blue] (11.4,0) -- (11.4,-3) ;
\draw [thick,color=blue] (11.4,-1.5) node[left] {$H$} ;
\draw [>=stealth,<->,color=red!100] (0,-3) -- (12.4,-3) ;
\draw [thick,color=red] (6.2,-3) node[above] {$T$} ;
\draw [>=stealth,<->,color=red] (12.4,3) -- (17.5,3) ;
\draw [thick,color=red] (14.95,3) node[above] {$U'$} ;
\draw [thick] (0,0) node {$\times$} ;
\draw (0,0) node[below left] {$0$} ;
\end{tikzpicture}

&
~~
&

\begin{tikzpicture}[scale =0.25]
\foreach \k in {-3,-2,...,5} {\draw [color=gray!15,dashed,rotate=180,xshift=0.5cm] (-1,\k)-- (18,\k);}
\draw [>=stealth,->,color=gray!100] (-17.5,-5) -- (-17.5,4) ;
\draw [>=stealth,->,color=gray!100] (-18.5,1) -- (1.5,1) ;
\draw [samples=1000,rotate=180,thick](0,-1) -- (0,0) -- (1.3,0)-- (1.3,1) -- (1.8,1)--(1.8,0)--(2,0)--(2,1)--(2.4,1)-- (2.4,0) -- (3.8,0) -- (3.8,1) -- (4.4,1) -- (4.4,2) -- (5.4,2) -- (5.4,1) -- (5.9,1) -- (5.9,2) -- (6.2,2) -- (6.2,3) -- (7.9,3) -- (7.9,4) -- (9,4) -- (9,2) -- (9.5,2) -- (9.5,3) -- (9.7,3) -- (9.7,4) -- (10.6,4) -- (10.6,5) -- (11,5) -- (11,4) -- (12.4,4) -- (12.4,-3) -- (13,-3) -- (13,-2) -- (13.7,-2) -- (13.7,-1) -- (15,-1) -- (15,-2) -- (16.8,-2) -- (16.8,-1) -- (17.5,-1) ;
\draw [>=stealth,<->,color=red,rotate=180] (12.4,3) -- (17.5,3) ;
\draw [thick,color=red] (-14.95,-3) node[above] {$U'$} ;
\draw [>=stealth,<->,color=red!100,rotate=180] (0,-3) -- (12.4,-3) ;
\draw [thick,color=red] (-6.2,3) node[below] {$T$} ;
\draw [>=stealth,<->,color=blue,rotate=180] (11.4,4) -- (11.4,-1) ;
\draw [thick,color=blue] (-11.4,2) node[right] {$H-1$} ;
\draw [>=stealth,<->,color=blue,rotate=180] (11.4,-1) -- (11.4,-3) ;
\draw [thick,color=blue] (-11.4,-1) node[right] {$\widehat{H}+1$} ;
\draw [thick] (-17.5,1) node {$\times$} ;
\draw (-17.5,1) node[below left] {$0$} ;
\end{tikzpicture}
\end{tabular}
\caption{The walk $\Sc$ and the time-reversed walk $\St$ on $[0,T+U']$.}
\label{fig:Reverting}
\end{figure}

\[\label{bigidentitylaw}
\left(
\begin{array}{c}
T\\
U'\\
H\\
1+\widehat{H}\\
L\\
\Vc_{T\scriptscriptstyle{-}}\\
{\Vc}_{(T+U')\scriptscriptstyle{-}}-\Vc_T\\
\Vc_T-\Vc_{T\scriptscriptstyle{-}}
\end{array}
\right)
\overset{\law}{=}
\left(
\begin{array}{c}
U'\\
T\\
1+\widehat{H}\\
H\\
L\\
{\Vc}_{(T+U')\scriptscriptstyle{-}}-\Vc_T\\
\Vc_{T\scriptscriptstyle{-}}\\
\Vc_T-\Vc_{T\scriptscriptstyle{-}}
\end{array}
\right).\]
On the other hand, conditionally on $(\Sc_t,\Vc_t)_{t\leqslant T}$, the process 
\begin{equation*}
(\Sc_{T+t}-\Sc_{T},\, \Vc_{T+t}-\Vc_T)_{0\leqslant t\leqslant U'}
\end{equation*}
has the same law as an independent copy of $(\Sc,\Vc)$ stopped at the first time when $\Sc$ reaches height $H$. Since this process is right-continuous, it can be decomposed into its $H$ excursions between new maxima. These excursions are i.i.d.~and their lengths are distributed as $\Tup$. This leads to the equality:
\begin{equation*}
U' \overset{\law}{=} \widetilde{T}^{\scriptscriptstyle{\uparrow}}_1+\dots+ \widetilde{T}^{\scriptscriptstyle{\uparrow}}_{H},
\end{equation*}
where $(\widetilde{T}^{\scriptscriptstyle{\uparrow}}_i)_{i\geqslant 1}$ are i.i.d.~copies of $\Tup$ which are independent of $H$. 
\medskip

Finally, we can apply Tanaka's construction of a random walk conditioned to stay positive via time reversal to relate the trajectory of $\Sc$ up to its hitting time of a given level $h$ with the trajectory of $\Sc^+$ up to its last passage time at the same height~$h$. More precisely, in our setting, adapting the argument of \cite{Tan89} shows that
\begin{equation*}
\big(\Sh_t,\,\Vh_t\big)_{t \leq U} \,\overset{\law}{=} \, \big( -\Sc_{(T+U' -t)\scriptscriptstyle{-}},\, \Vc_{(T+U')\scriptscriptstyle{-}} - \Vc_{(T+U' -t)\scriptscriptstyle{-}}  \big)_{t \leq U'}.
\end{equation*}
In particular we have $\smash{U\overset{\law}=U'}$ which completes the proof of the first identity of item~1. The second identity follows as well by recalling that
\[(T,\,1+\widehat{H},\,U',\,H)\overset{\law}=(U',\,H,\,T,\,1+\widehat{H}).\]
Conditionally on $(\Sc_t)_{t\leqslant T}$, the pairs $(U,\Vh_{U^-})$ and $\smash{(U',\Vc_{(T+U')\scriptscriptstyle{-}}-\Vc_T)}$ have the same law. By decomposing again along the $H$ excursions between new maxima, we conclude that $\Vh_{U^-}$  can be written as a sum of $H$ independent copies of $\smash{\Vc_{\Tu}}$. This proves the third identity of item~1. Finally, the last identity follows also from the previous construction by observing that $\Vc_U=\Vc_{U-}+\chi$ where $\chi$ is a random variable which is independent of all the other quantities and with law $\law(\vv_1 \,|\,\xi_1=1)$.
\medskip

Now let us prove item~2: the law of $H$ is well known (see for instance~\cite{Fel71} page~440-441). We can then deduce the law of $L$ from that of $H$ using the fact that $1+\widehat{H}$ and $H$ have the same law and conditioning on $\widehat{H}$:
\begin{align*}
\P(L=k)=\sum_{j=0}^{k-1}\P(L=k\,|\,\widehat{H}=j)\P(\widehat{H}=j)& = \sum_{j=0}^{k-1}\frac{\P(\xi_1 = -k)}{\P(\xi_1<-j)}\P(H = j+1)\\
&  = k \frac{\P(\xi_1 = -k)}{\P(\xi_1 = 1)}. \qedhere
\end{align*}
\end{proof}

\begin{prop} 
We have 
\[\P(\Vc_{T}>x)\lesssim x^{\delta(1-\alpha)/\alpha}.\]
\label{propVolRW2}
\end{prop}

\begin{proof} We remark that $\Vc_{T}$ does not depend on the time parametrization of our processes. It is convenient here to work in discrete time so we define the discrete time random walk $(\Sd_n,\Vd_n)_{n\in\N}$ whose increments are, as before, given by the sequence $(\xi_i,\vv_i)_{i\geqslant 1}$. Thus $(\Sc,\Vc)$ and $(\Sd,\Vd)$ are time-changed of each other. The corresponding strict descending and strict ascending ladder times of $\Sd$ are denoted respectively by $\Td$ and~$\smash{\Tud}$.
With these notations, we have $\Vc_T=\Vd_\Td$ and  $\smash{\Vc_{\Tu}=\Vd_{\Tud}}$. A straightforward adaptation of Proposition~\ref{SkipFree-Symp} shows that
\begin{equation}\label{timeequality}
\Vd_{\Td-1}\overset{\law}{=} \widetilde{\Vd}^{(1)}+\dots+ \widetilde{\Vd}^{(H)},
\end{equation}
where the $\widetilde{\Vd}^{(i)}$'s are i.i.d~copies of $\Vd_{\Tud}$ which are also independent of $H = -\Sd_{\Td}$. We study separately the tail of $\Vd_{\Tud}$ and that of the last jump $\Vd_{\Td} - \Vd_{\Td-1}$.
\medskip

\textbf{\textit{Tail of $\Vg_\Tud$.}}  Given a sequence $Y=(Y_i)_{i\geqslant 1}$, we define
\[t(Y)=\inf\,\{\,k\geqslant 1:Y_1+\dots+Y_k=1\,\}.\]
In particular, we have $\smash{t(\xi)=\Tud}$. Let also $\sigma_n$ denote the cyclical permutation of the $n$ first variable:
\[x=(x_i)_{i\geqslant 1}\longmapsto \sigma_n(x)=(x_2,\dots, x_n,x_1,x_{n+1},\dots).\]
A variant of the Ballot Theorem given in Lemma 6.1 p.122 of \cite{Pitman06} states that, since $\xi$ is a right-continuous path, we have
\[\bigsqcup_{j=1}^n\left\{t\left(\sigma_n^j(\xi)\right)=n\right\}=\{\Sd_{n}=1\}.\]
Thus, we find that
\[ _P(\Vd_\Tud>x)=\sum_{n\geqslant 1}\P(\Vd_n>x,\,\Tud=n)
=\sum_{n\geqslant 1}\frac{1}{n}\P(\Vd_n>x,\,\Sd_n=1).\]
At least half of the volume at time $n$ must have been collected, either before time $n/2$ or after that time, \emph{i.e}
\[\{\Vd_n>x,\,\Sd_n=1\}\subset \{\Vd_n-\Vd_{ \lfloor \frac{n}{2} \rfloor}>x/2,\,\Sd_n=1\}\bigcup\{\Vd_{ \lceil \frac{n}{2} \rceil}>x/2,\,\Sd_n=1\}.\]
The two events on the right hand side have the same probability since we can go from one to the other by applying the measure-preserving transformation $\smash{\sigma_n^{ \lceil n/2 \rceil}}$. As a consequence, we have
\[\P(\Vd_n>x,\,\Sd_n=1)\leqslant 2\,\P(\Vd_{\lceil \frac{n}{2} \rceil }>x/2,\,\Sd_n=1).\]
According to~\eqref{asumpt2}, the random walk $\Sd$ is in the domain of normal attraction of a spectrally negative stable law of index $\alpha = 3/2$. Thus, the local limit theorem stated in~\cite{IL71}, Theorem~4.2.1, shows that
\[\displaystyle{\sup_{k\geqslant -n}\, \P(\Sd_n=k)\lesssim n^{-1/\alpha}}.\]
As a consequence, we have
\begin{align*}
\P(\Vd_n>x,\,\Sd_n=1)&\leqslant2\E\left[\ind{\Vd_{  \lceil \frac{n}{2} \rceil}>x/2}\,\P(\Sd_n=1\,|\,\Sd_{ \lceil \frac{n}{2} \rceil})\right]\\
&\lesssim \P(\Vd_{ \lceil \frac{n}{2} \rceil}>x/2)\,n^{-1/\alpha} \lesssim \left(\frac{n}{x^{\delta}}\wedge 1\right) \,n^{-1/\alpha},
\end{align*}
where we used \eqref{asumpt3} and Lemma~\ref{Sum-positive-stable} for the last inequality. Therefore
\begin{align*}
\P(V_{\Tud}>x)&=\sum_{n\geqslant 1}\frac{1}{n}\P(\Vd_n>x,\,\Sd_n=1)\lesssim \sum_{n\geqslant 1}\frac{1}{n}\,\left(\frac{n}{x^{\delta}}\wedge 1\right) \,n^{-1/\alpha}\\
&\lesssim \sum_{n=1}^{\lfloor x^{\delta}\rfloor} \frac{1}{x^{\delta}n^{1/\alpha}}+\sum_{n>\lfloor x^{\delta}\rfloor} \frac{1}{n^{1+1/\alpha}}\lesssim x^{-\delta/\alpha}.
\end{align*}
\medskip

\textbf{\textit{Tail of $\Vd_{\Td-1}$.}} In view of \eqref{timeequality}, using again Lemma~\ref{Sum-positive-stable}, we find that
\begin{align*}
\P(\Vd_{\Td-1}>x)&\leqslant
\sum_{n\geqslant 1}\P(\widetilde{\Vd}^{(1)}+\dots+\widetilde{\Vd}^{(n)}>x)\P(H=n)\\
&\lesssim \P(H\geqslant x^{\delta/\alpha})+\sum_{n=1}^{\lfloor x^{\delta/\alpha}\rfloor } \frac{n}{x^{\delta/\alpha}}\,\P(H=n)\lesssim x^{(1-\alpha)\delta/\alpha},
\end{align*}
where we used the exact distribution of $H$ given in Proposition~\ref{SkipFree-Symp} combined with~\eqref{asumpt2} to obtain the last inequality. 
\medskip

\textbf{\textit{Tail of $\Vd_\Td-\Vd_{\Td-1}$.}} Recalling definition \eqref{defHandL}, we have $L = \Sd_{\Td} - \Sd_{\Td-1}$. By conditioning on the size of this jump, we can write
\[\P(\Vd_\Td-\Vd_{\Td-1}>x)= \sum_{j\geqslant 1}\P(\vv_1>x\,|\,\xi_1=-j)\P(L=j).\]
Using Markov's inequality and Assumption \eqref{asumpt4}, we find that
\begin{align*}
\P(\Vd_\Td-\Vd_{\Td-1}>x)&\lesssim
\P(L>x^{\delta/\alpha})+\sum_{j=1}^{\lfloor x^{\delta/\alpha}\rfloor }\frac{j^{\alpha/\delta}}{x}\,\P(L=j)\\
&\lesssim \P(L>x^{\delta/\alpha})+\sum_{n=1}^{\lfloor x^{\delta/\alpha}\rfloor }\frac{j^{\alpha/\delta-1}}{x}\,\P(L>j),
\end{align*}
where we used an Abel transform on the sum for the last line. Item 2 of Proposition~\ref{SkipFree-Symp} combined with~\eqref{asumpt2} imply that $\smash{\P(L > x) \lesssim x^{1-\alpha}}$ so we obtain that
\[\P(\Vd_\Td-\Vd_{\Td-1}>x)\lesssim x^{(1-\alpha)\delta/\alpha}.\]
Finally, by union bound, we conclude that
\begin{equation*}
\P(\Vd_{\Td}>x) \leq \P(\Vd_{\Td-1}> x/2) + \P(\Vd_\Td-\Vd_{\Td-1}> x/2)\lesssim x^{(1-\alpha)\delta/\alpha}.\qedhere
\end{equation*}
\end{proof}

\subsection{Volume generated by a step of the peeling process}
\setcounter{theo}{0}

Let $\yy^{(k)}$ denote the number of inner vertices inside a random free Boltzmann triangulation of the $k$-gon. As previously stated in \eqref{formulaVolume}, we have the explicit formula which may be deduced, for instance, by combining Equations (1) and (3) of \cite{CLG16}:
\[\P(\yy^{(k)}=n)=2\,\frac{\displaystyle{(2k-3)k(k-1)(2k+3n-4)!}}{\displaystyle{n!(2k+2n-2)!}}\,\left(\dfrac{4}{27}\right)^{n}\left(\dfrac{4}{9}\right)^{k-1}
\]
(with the convention $Y^{(k)}= 1$ for $k < 2$). We are interested in the tail asymptotic of $Y^{(k)}$ when the size $k$ of the boundary is itself random and distributed as the step of the random walk associated with the peeling process. Then, this quantity corresponds, asymptotically, to the size of the free Boltzmann triangulations discovered during the peeling procedure when the boundary becomes large\footnote{equivalently, it corresponds to the size of the free Boltzmann triangulations discovered during the peeling of the UIHPT.}. 

\begin{prop}\label{annealedVolStep}
Let $X$ be an integer random variable independent of the $\yy^{(k)}$'s, such that $1-X$ is distributed according to $(p_k)_{k\in \Z}$ given by \eqref{defrwpk}, \ie $P(X=0)=2/3$ and
\[\forall k\geqslant 2\qquad \P(X=k)=\dfrac{2(2k-2)!}{4^k(k-1)!(k+1)!}.\]
Then $\yy^{(X)}$ is in the normal domain of attraction of a stable law with index  $3/4$: 
\[\P(\yy^{(X)}>x)\sim\frac{2^{3/2}}{3^{7/4}\Gamma(1/4)}\,x^{-3/4}\qquad\mbox{as $x\to+\infty$}.\]
\end{prop}

\begin{proof} Fix $\varepsilon>0$. According to \eqref{formulaVolumeExpectation}, we have $\E[\yy^{(k)}]<(2/3)k^2$ for any $k\geqslant 2$. Thus,  Markov's inequality states that $\P(\yy^{(k)}>x)<(2/3)k^2/x$. Using the asymptotic $\P(X=k)\sim (2\sqrt{\pi})^{-1}k^{-5/2}$ when $k$ goes to $+\infty$, we get the crude upper bound:
\begin{equation*}
\P(\yy^{(k)}>x)\P(X=k)<C \,\min(k^2/x,1)\,k^{-5/2}.
\end{equation*}
Using this estimate, we can guess the interesting scale:
\begin{multline*}
\P(\yy^{(X)}>x)=\sum_{k\geq 0}\Pg(\yy^{(k)}>x)\P(X=k)\\
=\sum_{\substack{
\varepsilon\sqrt{x}\leqslant k\leqslant \sqrt{x}/\varepsilon\\
x \leq n \leq x/\varepsilon^2
}} \P(\yy^{(k)} = n)\P(X = k)+O(\varepsilon^{5/2}x^{-3/4}).
\end{multline*}
Now, since $k$ is of order $\sqrt{x}$ and $n$ of order $x$, we can get uniform estimates with Stirling formula. An easy (but tedious) computation shows that, uniformly in $k$ and $n$ in this scale, we have
\[\P(\yy^{(k)}=n)\P(X=k)\sim\frac{1}{3\pi\sqrt{3}}\,{k^{1/2}n^{-5/2}}e^{-k^2/(3n)}.\]
Setting $u=k/\sqrt{x}$ and $v=n/x$, we get
\begin{align*}
\P(\yy^{(X)}>x)&=\left(\frac{1}{3\pi\sqrt{3}}+o(1)\right)\sum_{\substack{\varepsilon \sqrt{x}\leqslant k\leqslant \sqrt{x}/\varepsilon\\ x \leq n \leq x/\varepsilon^2}}{k^{1/2}n^{-5/2}}e^{-k^2/(3n)}+O(\varepsilon^{5/2}x^{-3/4})\\
&=\frac{x^{-3/4}}{3\pi\sqrt{3}}\frac{1}{x}\frac{1}{\sqrt{x}}
\sum_{\substack{u\in\frac{1}{\sqrt{x}} \llbracket\varepsilon \sqrt{x};\sqrt{x}/\varepsilon\rrbracket\\ v\in\frac{1}{x}\llbracket x;x/\varepsilon^{2}\rrbracket}}
u^{1/2}v^{-5/2}e^{-u^2/(3v)}+O(\varepsilon^{5/2}x^{-3/4})\\
&= \frac{x^{-3/4}}{3\pi\sqrt{3}} \int_{\gfrac{\varepsilon<u<1/\varepsilon}{ 1<v<1/\varepsilon^2}}u^{1/2}v^{-5/2}e^{-u^2/(3v)}\,du\,dv+O(\varepsilon^{5/2}x^{-3/4}).
\end{align*}
Computing the integral above completes the proof of the proposition. 
\end{proof}

\subsection{Harris inequality}
\setcounter{theo}{0}

\begin{prop}[\textbf{Harris inequality}]\label{HarrisProp}
A set of trajectories $A\subset \Z^\N$ is said to be increasing (for the canonical partial order of its increment) if, for any $\textbf{x} = (x_i)_{i\in\NN}$ and $\textbf{y} = (y_i)_{i\in\NN}$ such that $x_0 = y_0$ and $x_{i+1} - x_{i} \leqslant y_{i+1} - y_i$ for all $i$, we have
\begin{equation*}
\textbf{x}\in A \quad\Longrightarrow\quad \textbf{y}\in A.
\end{equation*}
If $\mathcal{S}$ is a process with independent increments starting from some deterministic point $\mathcal{S}_0 = x_0$, then any two increasing events $A$ and $B$ are positively correlated for $\mathcal{S}$, \ie
\begin{equation}\label{harriseq}
\P\big(\mathcal{S}\in A,\mathcal{S}\in B\big)\geqslant \P\big(\mathcal{S}\in A\big)\P\big(\mathcal{S}\in B\big).
\end{equation}
In particular, suppose $\mathcal{S}$ is a random walk starting from $x_0 > 0$ which does not diverges to $-\infty$. We denote by $\P(\,\cdot\,|\,\mathcal{S} >0)$ the law under which $\mathcal{S}$ is conditioned to stay positive in the sense of the Doob's h-transform. Then, we have, for any increasing events $A,B$, 
\begin{equation}\label{harriseqcondi}
\P\big(\mathcal{S}\in A,\mathcal{S}\in B \,\big|\,\mathcal{S} >0\big)\geqslant \P\big(\mathcal{S}\in A\,\big|\,\mathcal{S} >0 \big)\P(\mathcal{S}\in B).
\end{equation}
\end{prop}

The inequality~\eqref{harriseq} is a rewritting of the celebrated FKG inequality in the case of a product measure. See for instance section $2.2$ of~\cite{Grim99}. The variant~\eqref{harriseqcondi} with the conditioning is a simple consequence of the fact that $\P(\,\cdot\,|\,\mathcal{S} >0)$ can be obtained as the limit of the conditioned measures $\P(\,\cdot\,|\, A_N)$ where $A_N \defeq \{\mathcal{S}_k >0 \hbox{ for  all } k\leqslant \N\}$ is a sequence of increasing events.
\bigskip

\textbf{Acknowledgments.} We thank Nicolas Curien for his advice and fruitful discussions during the preparation of the paper.

\bibliographystyle{alpha}
\bibliography{biblio}
\addcontentsline{toc}{section}{References}
\markboth{\uppercase{References}}{\uppercase{References}}

\end{document}